\documentclass[11pt,a4paper,oneside,reqno]{amsart}
\usepackage{lmodern}
\usepackage[T1]{fontenc}
\usepackage[main=english, french]{babel}
\usepackage{amsopn}

\newtheorem{theorem}{Theorem}[section]

\newtheorem{corollary}[theorem]{Corollary}
\newtheorem{conjecture}[theorem]{Conjecture}
\newtheorem{lemma}[theorem]{Lemma}
\newtheorem{proposition}[theorem]{Proposition}

\newtheorem{definition}[theorem]{Definition}
\newtheorem{example}[theorem]{Example}
\newtheorem{remark}[theorem]{Remark}

\usepackage[all]{xy}
\usepackage{enumerate}
\usepackage{amsfonts,amssymb,amsmath,eucal,pinlabel,array,hhline}
\usepackage{slashed}
\usepackage{tabulary}
\usepackage{fancyhdr}
\usepackage{a4wide}
\usepackage{calrsfs,bbm,dsfont}
\usepackage[position=b]{subcaption}

\newcounter{notes}%

	\newcommand{\ignore}[1]{}  

\newcommand{\Li}{\mathrm{Li}_2}
\newcommand{\Vol}{\mathrm{Vol}}

	\usepackage{comment}
\usepackage{amsfonts}
\usepackage{multirow} 

\usepackage{brunnian} 

\usetikzlibrary{knots}

\usetikzlibrary{matrix}
\usetikzlibrary{arrows}

\usetikzlibrary{patterns}
\usetikzlibrary{shapes}
\usetikzlibrary{arrows}
\usetikzlibrary{decorations.markings}
\usetikzlibrary{decorations.pathreplacing}

\newcommand{\Log}{\mathrm{Log}}

\usetikzlibrary{decorations.pathreplacing,decorations.markings}

\tikzset{
  on each segment/.style={
    decorate,
    decoration={
      show path construction,
      moveto code={},
      lineto code={
        \path [#1]
        (\tikzinputsegmentfirst) -- (\tikzinputsegmentlast);
      },
      curveto code={
        \path [#1] (\tikzinputsegmentfirst)
        .. controls
        (\tikzinputsegmentsupporta) and (\tikzinputsegmentsupportb)
        ..
        (\tikzinputsegmentlast);
      },
      closepath code={
        \path [#1]
        (\tikzinputsegmentfirst) -- (\tikzinputsegmentlast);
      },
    },
  },
  mid arrow/.style={postaction={decorate,decoration={
        markings,
        mark=at position .5 with {\arrow[#1]{>}}
      }}}
      ,
}

\tikzset{
  on each segment/.style={
    decorate,
    decoration={
      show path construction,
      moveto code={},
      lineto code={
        \path [#1]
        (\tikzinputsegmentfirst) -- (\tikzinputsegmentlast);
      },
      curveto code={
        \path [#1] (\tikzinputsegmentfirst)
        .. controls
        (\tikzinputsegmentsupporta) and (\tikzinputsegmentsupportb)
        ..
        (\tikzinputsegmentlast);
      },
      closepath code={
        \path [#1]
        (\tikzinputsegmentfirst) -- (\tikzinputsegmentlast);
      },
    },
  },
  mid arrow d/.style={postaction={decorate,decoration={
        markings,
        mark=at position .5 with {\arrow[#1]{>>}}
      }}}
      ,
}

\tikzset{
  on each segment/.style={
    decorate,
    decoration={
      show path construction,
      moveto code={},
      lineto code={
        \path [#1]
        (\tikzinputsegmentfirst) -- (\tikzinputsegmentlast);
      },
      curveto code={
        \path [#1] (\tikzinputsegmentfirst)
        .. controls
        (\tikzinputsegmentsupporta) and (\tikzinputsegmentsupportb)
        ..
        (\tikzinputsegmentlast);
      },
      closepath code={
        \path [#1]
        (\tikzinputsegmentfirst) -- (\tikzinputsegmentlast);
      },
    },
  },
  mid arrow s/.style={postaction={decorate,decoration={
        markings,
        mark=at position .5 with {\arrow[#1]{stealth}}
      }}}
      ,
}

\tikzset{
  on each segment/.style={
    decorate,
    decoration={
      show path construction,
      moveto code={},
      lineto code={
        \path [#1]
        (\tikzinputsegmentfirst) -- (\tikzinputsegmentlast);
      },
      curveto code={
        \path [#1] (\tikzinputsegmentfirst)
        .. controls
        (\tikzinputsegmentsupporta) and (\tikzinputsegmentsupportb)
        ..
        (\tikzinputsegmentlast);
      },
      closepath code={
        \path [#1]
        (\tikzinputsegmentfirst) -- (\tikzinputsegmentlast);
      },
    },
  },
  mid arrow l/.style={postaction={decorate,decoration={
        markings,
        mark=at position .5 with {\arrow[#1]{latex}}
      }}}
      ,
}

\usepackage{graphicx}

\newcommand{\C}{\mathbb{C}}
\newcommand{\R}{\mathbb{R}}

\newcommand{\Z}{\mathbb{Z}}

\newcommand{\B}{\mathsf{b}}

\newcommand{\sarrow}{\,\mathbin{\rotatebox[origin=c]{90}{$\rightarrow$}}}
\newcommand{\darrow}{\,\mathbin{\rotatebox[origin=c]{90}{$\twoheadrightarrow$}}}

\usepackage{amssymb,amsmath,bm}
\usepackage{amsthm}
\usepackage{empheq}
\usepackage{textcomp}
\usepackage{enumerate}      
\usepackage{graphicx}        
\usepackage{caption}
\usepackage{nicematrix}
\usepackage{tikz-cd}
\usepackage{tikz}

\usepackage{colorinfo}

\usetikzlibrary{knots}

\usetikzlibrary{matrix}
\usetikzlibrary{arrows}

\usetikzlibrary{patterns}
\usetikzlibrary{shapes}
\usetikzlibrary{arrows}
\usetikzlibrary{decorations.markings}
\usetikzlibrary{decorations.pathreplacing}

\usepackage{mathrsfs}

\usepackage{url}

\renewcommand{\setminus}{{\smallsetminus}}

\usepackage{amssymb,amsmath,bm}
\usepackage{amsthm}
\usepackage{hyperref}
\usepackage{mathrsfs}
\usepackage{textcomp}
\usepackage{enumerate}
\usepackage{graphicx}
\usepackage{tikz}
\usepackage{verbatim}
\usepackage{nicematrix}

\numberwithin{equation}{section}

\def\Vol{\operatorname{Vol}}
\def \CC{\mathbb C}
\def \Re{\operatorname{Re}}
\def \arg{\operatorname{arg}}

\def\SS{{\mathbb S}}
\def\CC{{\mathbb C}}
\def\RR{{\mathbb R}}
\def\ZZ{{\mathbb Z}}

\def\QQ{{\mathbb Q}}

\def\Vol{\operatorname{Vol}}
\def\CS{\operatorname{CS}}
\def \Re{\operatorname{Re}}
\def \im{\operatorname{Im}}
\def \arg{\operatorname{arg}}
\def \Li{\operatorname{Li_2}}
\def \im{\operatorname{Im}}

\def \Log{\operatorname{Log}}
\def \BLog{\operatorname{\textbf{Log}}}
\def \Arg{\operatorname{Arg}}

\def \det{\operatorname{det}}
\def \Hess{\operatorname{Hess}}

\title[Asymptotics aspects of Teichm\"{u}ller TQFT]{Asymptotics aspect of Teichm\"{u}ller TQFT for \\ generalized FAMED semi-geometric triangulations}
\author{Ka Ho Wong}

\begin{document}

\begin{abstract}
We introduce a generalized FAMED property for ideal triangulations of hyperbolic knot complements in $\SS^3$. Given a hyperbolic knot $K$ in $\SS^3$ and a semi-geometric triangulation $X$ of $\SS^3 \setminus K$ that is generalized FAMED with respect to the longitude. We prove that in the semi-classical limit $\hbar \to 0^+$, for any angle structure $\alpha$, the partition function $\mathscr{Z}_\hbar(X,\alpha)$ in Teichm\"uller TQFT decays exponentially with decrease rate the volume of $\SS^3 \setminus K$ equipped with a hyperbolic cone structure determined by $\alpha$, and that the 1-loop invariant of Dimofte-Garoufalidis emerges as the 1-loop term. With additional combinatorial conditions on the triangulations, we prove the existence of the Jones function and show that its decay rate is governed by the Neumann-Zagier potential function. In particular, the Andersen-Kashaev volume conjecture holds for every hyperbolic knot whose complement admits such kinds of triangulations.
\end{abstract}

\maketitle

\tableofcontents

\section{Introduction}
In \cite{BAW}, Ben-Aribi and the author introduce a new combinatorial notion called FAMED triangulations. The term FAMED stands for "Face Adjacency Matrices with Edge Duality". This condition plays an important role in the study of the asymptotics of partition functions in Teichm\"{u}ller TQFT defined by Andersen-Kashaev \cite{AK}. Precisely, Theorem 1.5 and Theorem 1.14 in \cite{BAW} prove the asymptotics expansion conjectures for the partition functions and Jones functions in Teichm\"uller TQFT for FAMED geometric triangulations. As a special case, the Andersen-Kashev volume conjecture holds for all FAMED geometric triangulations. 

In this paper, we generalize the results in \cite{BAW} in two directions. First, combinatorially, although it is a conjecture that every hyperbolic knot complement in $\SS^3$ admits a FAMED triangulation, there exists some ordered ideal triangulation that is not FAMED. As an attempt to resolve this problem, we introduce a \emph{generalized FAMED property} for ideal triangulations (see Definitions \ref{defgenFAMED}). Conjecturally, the condition holds for \emph{all} ordered ideal triangulations of hyperbolic knot complements in the three-sphere (computational supporting evidence will be provided in an upcoming preprint \cite{WC}). Next, geometrically, the existence of geometric triangulation is still an open problem in hyperbolic geoemetry. Nevertheless, it is known that for every cusped hyperbolic 3-manifold, there exists an ideal triangulation with geometric and flat tetrahedra. The standard way to obtain such an ideal triangulation is by triangulating the polyhedra decomposition of the Epstein-Penner decomposition of the cusped 3-manifold and inserting flat tetrahedra between two triangulated polyhedral faces if necessary. 
We call an ideal triangulation \emph{semi-geometric} if every tetrahedron in the triangulation is either geometric or flat. Our first main result provides an asymptotic expansion formula for partition functions in Teichm\"uller TQFT for any generalized FAMED semi-geometric triangulation (see Theorem \ref{mainthmZ}). For any hyperbolic knot $K$ in $\SS^3$ whose complement admits a generalized FAMED semi-geometric triangulation satisfying additional combinatorial properties (see Definition \ref{defgenFAMED5}), our second main result shows the existence of the Jones function in Teichm\"uller TQFT and provides an asymptotics expansion formula for the Jones function (see Theorem \ref{thm:Jones:genFAMED}), which in particular implies the Andersen-Kashaev volume conjecture for $K$.

\subsection{Combinatoric of ordered ideal triangulation and volume conjectures}
To state our main results, we briefly discuss the definitions of several matrices that will be used to construct the Teichm\"uller TQFT invariants of \cite{AK}. For simplicity, we only consider hyperbolic knot complements in $\SS^3$. Note that our notations are slightly different from those in \cite{AK, BAW}. 

Recall that a $3$-dimensional triangulation is \textit{ordered} (or \textit{branched}) if every tetrahedron is endowed with an order on its four vertices such that face gluings respect the relative vertex order (or equivalently, if the edges can be oriented so that no face is a $3$-cycle). See Figure \ref{fig:41:face:matrices} for an example we will develop in detail.

Let $K\subset \SS^3$ be a hyperbolic knot in $\SS^3$.
Let $X$ be an ordered ideal triangulation of $\SS^3\setminus K$ with $N$ tetrahedra $X^3 = \{T_1,\dots, T_N\}$.
If we rotate $T$ such that $0$ is in the center and $1$ at the top, then there are two possible places for vertices $2$ and $3$. We call $T$ a \textit{positive} (resp. \textit{negative}) tetrahedron if they are as shown on the left (resp. right) figure in the top row of Figure \ref{fig:41:face:matrices}. We denote $\varepsilon(T) \in \{ \pm 1\}$ the corresponding \textit{sign} of $T$.
We define $\mathcal{E}=\mathcal{E}(X)$ to be the diagonal matrix of size $N$ encoding the signs $\pm 1$ of the tetrahedra.

Let $X^2$ be the set of faces of $X$, of cardinal $2N$. For $k=0,1,2,3$, let $x_k : X^3 \to X^2$ be the map defined by sending a tetrahedron to its face that is opposite to the $k$-th vertex, and $\mathcal{X}_k \in M_{N,2N}(\Z)$ the matrix of coefficients
$(\mathcal{X}_k)_{i,j}:=\delta_{j\text{-th face},x_k(T_i)}$. Finally, define \textcolor{black}{matrices $\mathcal{B}:=\begin{pmatrix}
    0_N \\ \mathcal{E}
\end{pmatrix} \in M_{2N\times N}(\Z)$
and $\mathcal{A} = \begin{pmatrix}
    \mathcal{X}_0-\mathcal{X}_1+\mathcal{X}_2\\\mathcal{X}_2-\mathcal{X}_3
\end{pmatrix} \in M_{2N\times 2N}(\Z)$}. We informally call these matrices "face adjacency matrices" associated to the ordered triangulation $X$. See
Figure \ref{fig:41:face:matrices} for the matrices $\mathcal{X}_0, \mathcal{X}_1, \mathcal{X}_2, \mathcal{X}_3, \mathcal{A}, \mathcal{B}, \mathcal{E}$ of Thurston's ideal triangulation of the figure eight knot complement.

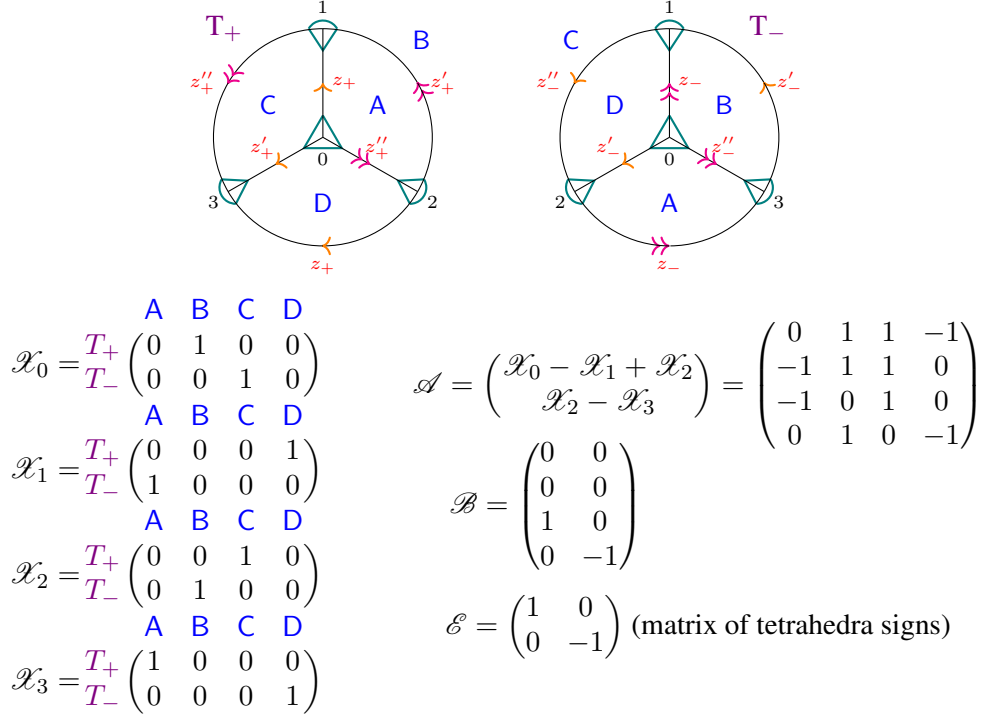
\begin{figure}[!h]
    \centering

\begin{tikzpicture}

\begin{scope}[scale=1.2]

\begin{scope}[xshift=0cm,yshift=0cm,rotate=0,scale=1.2]

\draw[color=violet] (-0.9,1) node{T$_+$} ;

\draw[color=red] (0.2,0.5) node{\tiny $z_+$} ;
\draw[color=red] (0,-1.2) node{\tiny $z_+$} ;
\draw[color=red] (0.5,-0.1) node{\tiny $z''_+$} ;
\draw[color=red] (-1.1,0.5) node{\tiny $z''_+$} ;
\draw[color=red] (-0.55,-0.1) node{\tiny $z'_+$} ;
\draw[color=red] (1.1,0.5) node{\tiny $z'_+$} ;

\draw (0,-0.2) node{\tiny $0$} ;
\draw (0,1.2) node{\tiny $1$} ;
\draw (1,-0.6) node{\tiny $2$} ;
\draw (-1,-0.6) node{\tiny $3$} ;

\begin{scope}[xshift=0cm,yshift=0cm,rotate=0,scale=.1]
    \draw[color=teal,thick] (0,2)--(-1.732,-1)--(1.732,-1)--(0,2);
\end{scope}

\begin{scope}[xshift=0cm,yshift=1cm,rotate=180,scale=.1]
    \draw[color=teal,thick ] (-1.3,0)--(0,2)--(1.3,0);
    \draw[color=teal,thick] (-1.3,0) arc (-150:-30:1.5);
\end{scope}

\begin{scope}[xshift=-0.86cm,yshift=-.5cm,rotate=-60,scale=.1]
    \draw[color=teal,thick ] (-1.3,0)--(0,2)--(1.3,0);
    \draw[color=teal,thick] (-1.3,0) arc (-150:-30:1.5);
\end{scope}

\begin{scope}[xshift=0.86cm,yshift=-.5cm,rotate=60,scale=.1]
    \draw[color=teal,thick ] (-1.3,0)--(0,2)--(1.3,0);
    \draw[color=teal,thick] (-1.3,0) arc (-150:-30:1.5);
\end{scope}

\draw[color=blue] (0.9,0.9) node{\small $\mathsf{\color{blue}B}$} ;
\draw[color=blue] (0,-0.6) node{\small $\mathsf{\color{blue}D}$} ;
\draw[color=blue] (-0.5,0.3) node{\small $\mathsf{\color{blue}C}$} ;
\draw[color=blue] (0.5,0.3) node{\small $\mathsf{\color{blue}A}$} ;


\path [draw=black,postaction={on each segment={mid arrow =orange, thick}}]
(0,0)--(-1.732/2,-0.5);

\path [draw=black,postaction={on each segment={mid arrow =orange, thick}}]
(0,0)--(0,1);

\path [draw=black,postaction={on each segment={mid arrow d =magenta, thick}}]
(0,0)--(1.732/2,-0.5);

\draw(1.732/2,-0.5) arc (-30:-89.5:1);
\draw[->,color=orange, thick](0,-1) arc (-89.5:-90:1);
\draw (0,-1) arc (-90:-150:1);

\draw(0,1) arc (90:29.5:1);
\draw[->>,color=magenta, thick](1.732/2,0.5) arc (29.5:30:1);
\draw (1.732/2,0.5) arc (30:-30:1);

\draw(0,1) arc (90:150:1);
\draw[->>,color=magenta, thick](-1.732/2,0.5) arc (149.5:150:1);
\draw (-1.732/2,0.5) arc (150:210:1);

\end{scope}

\begin{scope}[xshift=3.8cm,yshift=0cm,rotate=0,scale=1.2]

\draw[color=violet] (0.9,1) node{T$_-$} ;

\draw[color=red] (0.2,0.5) node{\tiny $z_-$} ;
\draw[color=red] (0,-1.2) node{\tiny $z_-$} ;
\draw[color=red] (0.5,-0.1) node{\tiny $z''_-$} ;
\draw[color=red] (-1.1,0.5) node{\tiny $z''_-$} ;
\draw[color=red] (-0.55,-0.1) node{\tiny $z'_-$} ;
\draw[color=red] (1.1,0.5) node{\tiny $z'_-$} ;

\draw (0,-0.2) node{\tiny $0$} ;
\draw (0,1.2) node{\tiny $1$} ;
\draw (1,-0.6) node{\tiny $3$} ;
\draw (-1,-0.6) node{\tiny $2$} ;

\begin{scope}[xshift=0cm,yshift=0cm,rotate=0,scale=.1]
    \draw[color=teal,thick] (0,2)--(-1.732,-1)--(1.732,-1)--(0,2);
\end{scope}

\begin{scope}[xshift=0cm,yshift=1cm,rotate=180,scale=.1]
    \draw[color=teal,thick ] (-1.3,0)--(0,2)--(1.3,0);
    \draw[color=teal,thick] (-1.3,0) arc (-150:-30:1.5);
\end{scope}

\begin{scope}[xshift=-0.86cm,yshift=-.5cm,rotate=-60,scale=.1]
    \draw[color=teal,thick] (-1.3,0)--(0,2)--(1.3,0);
    \draw[color=teal,thick] (-1.3,0) arc (-150:-30:1.5);
\end{scope}

\begin{scope}[xshift=0.86cm,yshift=-.5cm,rotate=60,scale=.1]
    \draw[color=teal,thick ] (-1.3,0)--(0,2)--(1.3,0);
    \draw[color=teal,thick] (-1.3,0) arc (-150:-30:1.5);
\end{scope}

\draw[color=blue] (-0.9,0.9) node{\small $\mathsf{\color{blue}C}$} ;
\draw[color=blue] (0,-0.6) node{\small $\mathsf{\color{blue}A}$} ;
\draw[color=blue] (-0.5,0.3) node{\small $\mathsf{\color{blue}D}$} ;
\draw[color=blue] (0.5,0.3) node{\small $\mathsf{\color{blue}B}$} ;


\path [draw=black,postaction={on each segment={mid arrow =orange, thick}}]
(0,0)--(-1.732/2,-0.5);

\path [draw=black,postaction={on each segment={mid arrow d =magenta, thick}}]
(0,0)--(0,1);

\path [draw=black,postaction={on each segment={mid arrow d =magenta, thick}}]
(0,0)--(1.732/2,-0.5);

\draw(1.732/2,-0.5) arc (-30:-89.5:1);
\draw[<<-,color=magenta, thick] (0,-1) arc (-89.5:-90:1);
\draw(0,-1) arc (-90:-150:1);

\draw(0,1) arc (90:29.5:1);
\draw[->,color=orange, thick](1.732/2,0.5) arc (29.5:30:1);
\draw (1.732/2,0.5) arc (30:-30:1);

\draw(0,1) arc (90:149.5:1);
\draw[->,color=orange, thick](-1.732/2,0.5) arc (149.5:150:1);
\draw (-1.732/2,0.5) arc (150:210:1);

\end{scope}

\end{scope}

\end{tikzpicture}

\begin{tikzpicture}

\draw (0,0) node {
$\mathcal{X}_0=\begin{pNiceMatrix}[first-row, first-col]
 &  \mathsf{\color{blue}A} & \mathsf{\color{blue}B} & \mathsf{\color{blue}C} & \mathsf{\color{blue}D} \\
\textcolor{violet}{T_+} & 0 & 1 & 0 & 0 \\
\textcolor{violet}{T_-} & 0 & 0 & 1 & 0
\end{pNiceMatrix}$
};

\draw (0,-1.4) node {
$\mathcal{X}_1=\begin{pNiceMatrix}[first-row, first-col]
 &  \mathsf{\color{blue}A} & \mathsf{\color{blue}B} & \mathsf{\color{blue}C} & \mathsf{\color{blue}D} \\
\textcolor{violet}{T_+} & 0 & 0 & 0 & 1 \\
\textcolor{violet}{T_-} & 1 & 0 & 0 & 0
\end{pNiceMatrix}$};

\draw (0,-2.8) node {
$\mathcal{X}_2=\begin{pNiceMatrix}[first-row, first-col]
 &  \mathsf{\color{blue}A} & \mathsf{\color{blue}B} & \mathsf{\color{blue}C} & \mathsf{\color{blue}D} \\
\textcolor{violet}{T_+} & 0 & 0 & 1 & 0 \\
\textcolor{violet}{T_-} & 0 & 1 & 0 & 0
\end{pNiceMatrix}$};

\draw (0,-4.2) node {
$\mathcal{X}_3=\begin{pNiceMatrix}[first-row, first-col]
 &  \mathsf{\color{blue}A} & \mathsf{\color{blue}B} & \mathsf{\color{blue}C} & \mathsf{\color{blue}D} \\
\textcolor{violet}{T_+} & 1 & 0 & 0 & 0 \\
\textcolor{violet}{T_-} & 0 & 0 & 0 & 1
\end{pNiceMatrix}$};

\draw (7,-.5) node {
$\mathcal{A} = \begin{pmatrix}
    \mathcal{X}_0-\mathcal{X}_1+\mathcal{X}_2\\\mathcal{X}_2-\mathcal{X}_3
\end{pmatrix}= \begin{pmatrix}
0&1&1&-1\\
-1&1&1&0\\
-1&0&1&0\\
0&1&0&-1
\end{pmatrix}$};

\draw (5,-1.6-.5) node {
$\mathcal{B}= \begin{pmatrix}
0&0\\
0&0\\
1&0\\
0&-1
\end{pmatrix}$};

\draw (7,-3.2-.5) node {
$\mathcal{E}=\begin{pmatrix}
    1&0\\0&-1
\end{pmatrix}$ (matrix of {tetrahedra signs})};

\end{tikzpicture}

\caption{Thurston's triangulation of $M=\SS^3 \setminus 4_1$, and the face adjacency matrices}
    \label{fig:41:face:matrices}
\end{figure}

Let rank$(\mathcal{A}) = r \leq 2N$. Let $\mathcal{E}_1, \mathcal{E}_2$ be product of elementary matrices such that 
\begin{align}\label{defmathcale}
\mathcal{E}_1 \mathcal{A} \mathcal{E}_2 = 
\begin{pmatrix}
I_r & O_1 \\
O_2 & O_3
\end{pmatrix},
\end{align}
where $O_1,O_2,O_3$ are zero matrices of suitable sizes.
Let $n = 2N-r$ be the nullity of $\mathcal{A}$. Write  
$$\mathcal{X}_0\mathcal{E}_2 = ( (\mathcal{X}_0\mathcal{E}_2)_r \mid (\mathcal{X}_0\mathcal{E}_2)_n) \quad \text{and} \quad \mathcal{E}_1\mathcal{B} = \begin{pmatrix} (\mathcal{E}_1\mathcal{B})_r \\ (\mathcal{E}_1\mathcal{B})_n \end{pmatrix}, $$
where $(\mathcal{X}_0\mathcal{E}_2)_r \in M_{N \times r}(\QQ), (\mathcal{X}_0\mathcal{E}_2)_n \in M_{N \times n}(\QQ)$, $(\mathcal{E}_1\mathcal{B})_r \in M_{r \times N}(\QQ)$ and $(\mathcal{E}_1\mathcal{B})_n \in M_{n\times N}(\QQ)$ respectively. In particular, we have 
\begin{align}\label{deltapart}
\begin{pmatrix}
(\mathcal{X}_0\mathcal{E}_2)_n^{\!\top} \\
(\mathcal{E}_1\mathcal{B})_n
\end{pmatrix} \in M_{2n \times N}(\QQ)
\end{align}
and
\begin{align}\label{defmathscrG}
\mathscr{G} = 
(\mathcal{X}_0\mathcal{E}_2)_r (\mathcal{E}_1\mathcal{B})_r  + \big((\mathcal{X}_0\mathcal{E}_2)_r (\mathcal{E}_1\mathcal{B})_r\big)^{\!\top}  
+
\frac{\mathcal{E}+ \rm{Id}_N}{2} \in M_{N\times N}(\QQ),
\end{align}
where the notation ${\!\top}$ denotes the transpose of a matrix. The matrices in (\ref{deltapart}) and (\ref{defmathscrG}) will show up in the kinematical kernel and the potential function of the Teichm\"uller TQFT partition function of $X$ (see Lemma \ref{KKdelta} and Proposition \ref{Tpartiexpress2} for details). In particular, the matrix in (\ref{deltapart}) imposes $2n$ linear equations on the integration variables due to the presence of the corresponding Dirac delta functions (see Definition \ref{newdefK}). 

We relate the matrices in (\ref{deltapart}) and (\ref{defmathscrG}) with Neumann-Zagier matrices with respect to the longitude as follows (a more detailed review will be provided in Section \ref{NZD}). Let $X^1=\{E_1,\dots, E_N\}$ be the set of 1-cells (edges) of the ideal triangulation $X$. It is known that for a knot complement, there exists a set of $N-1$ independent edges, in the sense that if the edge equations of these edges are satisfied, the last edge equation will automatically be satisfied. Furthermore, we assign $z_k$ to the edge with endpoints 0 and 1 and $z_k,z_k',z_k''$ have to cycle positively around each vertex in this order.  

Let $l$ be the the preferred longitude of $K$, which is null-homologous in the knot complement. Let $\mathrm{H}^\C_{X,l}(\mathbf{z})$ be the logarithmic holonomy of $l$ and let $\xi \in\CC$ be a complex number. We consider the following $N$ equations on logarithmic shape parameters $\BLog \mathbf{z},\BLog \mathbf{z'},\BLog \mathbf{z''}$ (where $\BLog \mathbf{z^{\cdot}}$ is the vector of parameters $\Log(z^{\cdot}_k)$) given by
\begin{itemize}
    \item the first $N-1$ gluing equations coming from edges of $X^1$ \textcolor{black}{and}
    \item the holonomy equation $\mathrm{H}^\C_{X,l}(\mathbf{z}) = \xi$.
\end{itemize}
This system of equations can be written
\begin{align}\label{equ}
 \mathbf{G} \BLog \mathbf{z} +\mathbf{G'} \BLog \mathbf{z'} + \mathbf{G''} \BLog \mathbf{z''} =
\begin{pmatrix}
    2i\pi \\ \vdots \\ 2i\pi \\ 
    \xi 
\end{pmatrix},
\end{align}
where $\mathbf{G},\mathbf{G'},\mathbf{G''} \in M_{N\times N}(\Z)$ are the \textit{Neumann-Zagier matrices} associated to the previous numbering of edges and the choice of curve $l$. We also define $\mathbf{A}:=\mathbf{G}-\mathbf{G'}$ and $\mathbf{B}:=\mathbf{G''}-\mathbf{G'}$ (see Section \ref{NZD}). By using the equation $\Log z + \Log z' + \Log z'' = \pi i,$ the system of equations can be rewritten as
\begin{align}\label{introglu}
    \mathbf{A} \BLog \mathbf{z} + \mathbf{B} \BLog \mathbf{z''} = i\boldsymbol{\nu} + \boldsymbol{\tilde{u}},
\end{align}
where $\boldsymbol{\nu} \in \pi\ZZ^{N}$ and $\boldsymbol{\tilde{u}} = (0,\dots,0, \xi)^{\!\top}$. 

Now, assume that nullity$(\mathbf{B})=2n = 2$ nullity$(\mathcal{A})$. Let $\mathbf{E} \in M_{N\times N}(\QQ)$ be the product of elementary matrices such that $(\mathbf{EB} | \mathbf{EA})$ is in the reduced row echelon form. It is known that rank$(\mathbf{A}|\mathbf{B}) = N$ (see \cite{NZ}). We write
$$
\mathbf{E}
=
\begin{pmatrix}
\mathbf{ E }_{N-2n}\\
\mathbf{ E }_{2n}
\end{pmatrix},
\mathbf{EB}
=
\begin{pmatrix}
\mathbf{ (EB) }_{N-2n}\\
\mathbf{O}
\end{pmatrix},
\mathbf{EA}
=
\begin{pmatrix}
 (\mathbf{EA})_{N-2n} \\
 (\mathbf{EA})_{2n}
\end{pmatrix}
,
$$
where $\mathbf{ (EA) }_{N-2n}\in M_{N-2n,N}(\QQ), \mathbf{ (EA) }_{2n}\in M_{2n,N}(\QQ), \mathbf{ (EB) }_{N-2n}\in M_{N-2n,N}(\QQ)$ and $\mathbf{O} \in M_{2n \times N}(\RR)$ is the $2n\times N$ zero matrix respectively. In particular, Equation (\ref{introglu}) becomes
\begin{empheq}[left=\empheqlbrace]{align}\label{introglu2}
    (\mathbf{EA})_{N-2n} \BLog \mathbf{z} + (\mathbf{EB})_{N-2n} \BLog \mathbf{z''} &= \mathbf{E}_{N-2n}( i \boldsymbol{\nu} + \boldsymbol{\tilde{u}}), \\
    (\mathbf{EA})_{2n} \BLog \mathbf{z} &= \mathbf{E}_{2n}( i \boldsymbol{\nu} + \boldsymbol{\tilde{u}}),
\end{empheq}
and the matrix $(\mathbf{EA})_{2n}$ imposes $2n$ linear equations on $\Log z_1, \dots, \Log z_N$. Using the symplectic properties of the Neumann-Zagier matrices, this linear subspace can be parametrized by using the matrix $(\mathbf{EB})_{N-2n}^{\!\top}$ (see Lemma \ref{NZpara}).

The following condition is a generalization of the FAMED condition introduced by Ben-Aribi and the author in \cite{BAW}. Conditions (3) and (4) in Definition \ref{defgenFAMED} below respectively ensure that
\begin{itemize}
    \item the system of linear equations imposed by the matrix in (\ref{deltapart}) and those imposed by $(\mathbf{EA})_{2n}$ are equivalent, and
    \item there is a correspondence between the critical point equation of the potential function and Equation (\ref{introglu2}). 
\end{itemize}
\begin{definition}\label{defgenFAMED} Let $K\subset \SS^3$ be a hyperbolic knot, $X$ be an ordered ideal triangulation of $\SS^3\setminus K$ and $\mathcal{A}$ be the face adjacency matrix associated to $X$ (as above).
Let $l$ be the preferred longitude of $K$. Let $\mathbf{A}, \mathbf{B}$ be the Neumann-Zagier matrices with respect to $l$ (as above). 
We say that $X$ is generalized FAMED with respect to $l$ if the following conditions hold.
\begin{enumerate}
\item The space of angle structures (see Section \ref{sub:angle:str}) is non empty: $\mathscr{A}_{X} \neq \emptyset$,
\item nullity$(\mathbf{B})$ $=$ 2nullity$(\mathscr{A}) < N$, 
\item The two $2n \times N$ matrices
$$
\mathbf{(EA)}_{2n}
\text{ and }
\begin{pmatrix}
(\mathcal{X}_0\mathcal{E}_2)_n^{\!\top} \\
(\mathcal{E}_1\mathcal{B})_n
\end{pmatrix}
$$
are row equivalent, and
\item The two $N \times 2N$ matrices
$$
\begin{pmatrix}
   (\mathbf{EB})_{N-2n} & (\mathbf{EA})_{N-2n}  \\
   \mathbf{O} & (\mathbf{EA})_{2n} 
\end{pmatrix} 
\text{ and }
\begin{pmatrix}
    (\mathbf{EB})_{N-2n} & \mathbf{(EB)}_{N-2n} \mathscr{G}  \\
    \mathbf{O} & (\mathbf{EA})_{2n} 
\end{pmatrix} 
$$
are row equivalent.
\end{enumerate}
\end{definition}

\begin{conjecture}\label{allgenFAMEDl}
    Every ordered ideal triangulation of a hyperbolic knot complement in $\SS^3$ that admits angle structures is generalized FAMED with respect to $l$.
\end{conjecture}

\begin{remark}
In the special case where $n=$ nullity$(\mathscr{A}) = 0$, (2) implies that $\det(\mathbf{B})$ and $ \det(\mathcal{A}) $ are non-zero. Besides, (3) becomes trivial. Moreover, $\mathbf{E} = \mathbf{B}^{-1}$ and (4) implies that $\mathbf{B}^{-1}\mathbf{A}=\mathscr{G}$. To compute $\mathscr{G}$, in Equation (\ref{defmathcale}), we can take $\mathcal{E}_1 = \mathcal{A}^{-1}$ and $\mathcal{E}_2 = \mathrm{Id}$. Altogether, the generalized FAMED conditions recover the FAMED condition introduced by Ben-Aribi and the author in \cite[Definition 1.1]{BAW}. 
\end{remark}

\begin{remark}\label{rmk4}
    In Definition \ref{defgenFAMED} (4), note that the first matrix is already in the reduced row echelon form. As a result, to transform the second matrix into its reduced row echelon form, we only need to apply elementary row operations that change the first $N-2n$ rows of the matrix. Those operations do not change $(\mathbf{EB})_{N-2n}$ and $(\mathbf{EA})_{2n}$. In particular, we can find a $N-2n \times 2n$ matrix $\mathbf{\tilde E}$ such that
\begin{align*}
    \begin{pmatrix}
        \mathrm{Id}_{N-2n} & \mathbf{\tilde E} \\
        \mathbf{O} & \mathrm{Id}_{2n}
    \end{pmatrix}
        \begin{pmatrix}
   (\mathbf{EB})_{N-2n} & (\mathbf{EA})_{N-2n}  \\
   \mathbf{O} & (\mathbf{EA})_{2n} 
\end{pmatrix}
= 
    \begin{pmatrix}
    (\mathbf{EB})_{N-2n} & \mathbf{(EB)}_{N-2n} \mathscr{G}  \\
    \mathbf{O} & (\mathbf{EA})_{2n} 
\end{pmatrix} 
\end{align*}
\end{remark}
We define a matrix $\mathbf{E'}$ and break it into two parts $\mathbf{E}'_{N-2n}$ and $\mathbf{E}'_{2n}$ by 
\begin{align}\label{defnE'}
    \mathbf{E'} =     \begin{pmatrix}
        \mathrm{Id}_{N-2n} & \mathbf{\tilde E} \\
        \mathbf{O} & \mathrm{Id}_{2n}
    \end{pmatrix} \mathbf{E} \quad \text{and} \quad 
    \mathbf{E'} =  \begin{pmatrix}
        \mathbf{E}'_{N-2n} \\
        \mathbf{E}'_{2n} 
    \end{pmatrix}.
\end{align}
Note that $\mathbf{E}'_{2n} = \mathbf{E}_{2n}$. We impose the following additional conditions stated in Definition \ref{defgenFAMED5} to study the Andersen-Kashaev volume conjecture (see Conjecture \ref{conj:vol:BAW}). Definition \ref{defgenFAMED5}(2) implies that
$$     (\mathbf{EA})_{2n} \BLog \mathbf{z} = \mathbf{E}_{2n}( i \boldsymbol{\nu} + \boldsymbol{\tilde{u}}) = \mathbf{E}_{2n}( i \boldsymbol{\nu}) $$
and it allows us to parametrize this solution subspace in a way that is independent of $\xi$ (see Lemma \ref{NZpara}). Definition \ref{defgenFAMED5}(3) implies that the last column of the matrix $\mathbf{E}_{N-2n}'$ in (\ref{defnE'}) captures the Neumann-Zagier data of the meridian. 
\begin{definition}\label{defgenFAMED5}
Let $K\subset \SS^3$ be a hyperbolic knot, $l$ be the longitude of $K$, $m$ be the meridian of $K$, and $X$ be an ordered ideal triangulation of $\SS^3\setminus K$. We say that $X$ is generalized FAMED with respect to $(l,m)$ if the following conditions are satisfied.
\begin{enumerate}
    \item $X$ is generalized FAMED with respect to $l$ (see Definition \ref{defgenFAMED}),
    \item If $n\neq 0$, all entries of the last column of $\mathbf{E}_{2n}$ are zero.
    \item Let $k_1<\dots< k_{N-2n}$ be the positions of pivots of $(\mathbf{EB})_{N-2n}$. Let $(C_1,\dots, C_{N-2n})^{\!\top}$ be the last column of $-2\mathbf{E}_{N-2n}'$. Let $m$ be the meridian of $K$ with the orientation such that the algebraic intersection number of $l$ and $m$ is $1$. There exists a representative in the homotopy class of $m$ with logarithmic holonomy equal to 
$$C_1\Log z_{k_1} + \dots + C_{N-2n} \Log z_{k_{N-2n}} + k\pi i$$
for some $k\in \QQ$.
\end{enumerate}
\end{definition}
\begin{figure}
    \centering
    \includegraphics[width=0.8\linewidth]{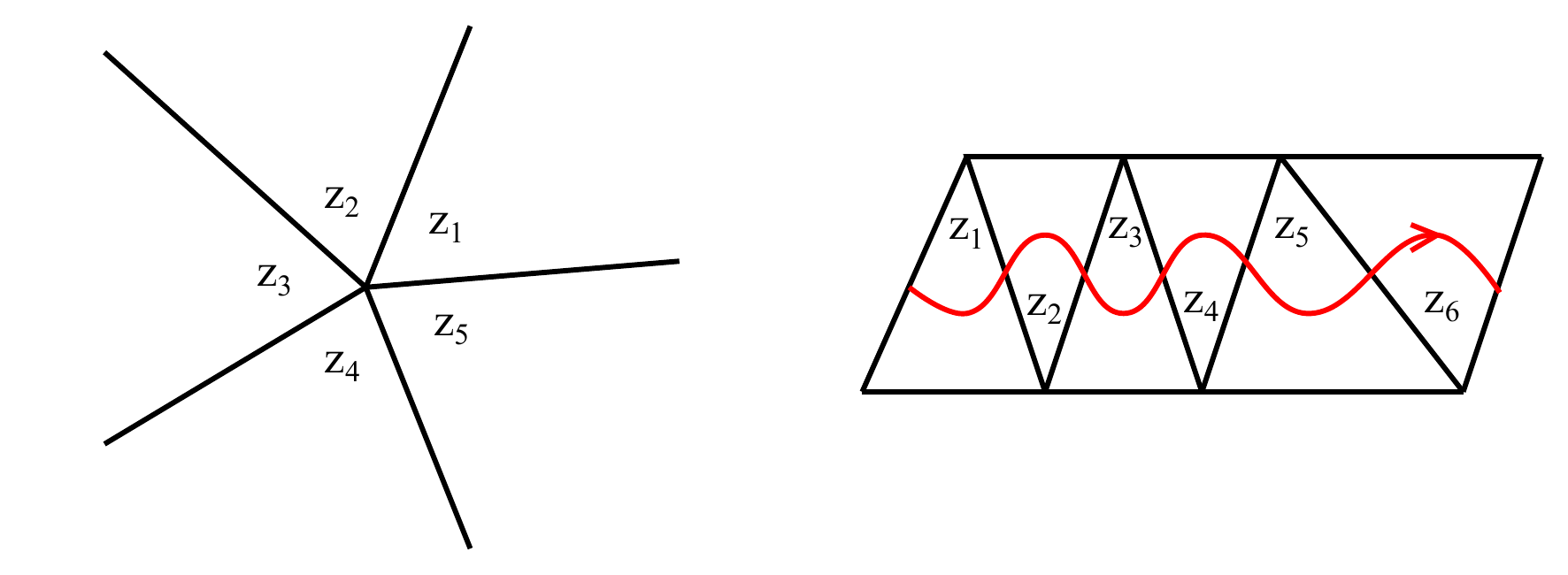}
    \caption{The first and second figures respectively show the combinatorics of the triangulation that is allowed and disallowed in Definition \ref{defgenFAMED5}(2). In the first figure, the edge imposes a linear equation $\Log z_1 + \Log z_2 + \dots + \Log z_5 = 2\pi i$ to the variables $\Log z_1,\dots, \Log z_N$. Note that the equation is independent of the choice of the logarithmic holonomy of the longitude $\xi$. In contrast, in the second figure, if the red curve represents the longitude of the knot on the cusp triangulation of the boundary torus, then the holonomy equation is given by $\Log z_1 + \dots + \Log z_6 = \xi$, which depends on the choice of $\xi$. }
    \label{alinear}
\end{figure}
\begin{remark}\label{12implies3}
    We expect that in Definition \ref{defgenFAMED5}, if (1) and (2) are satisfied, then (3) will be satisfied. See Lemma \ref{pretildexinter} for a result supporting this expectation.
\end{remark}
\begin{remark}
    For the case where $n=0$, condition (2) is automatically satisfied. For condition (3), since $\mathbf{E}=\mathbf{B}^{-1}$, we have $k_i = i$ for $i=1,2,\dots, N$. In this case, since $\mathbf{E}=\mathbf{E'}$, the vector $(C_1,\dots,C_{N})^{\!\top}$ is the last column of $-2\mathbf{E}$. It is proved in \cite[Lemma 4.1]{BAW} that the there exists $n_1\in \QQ$ and a representative in the homotopy class of $m+n_1l$ with logarithmic holonomy equal to 
$$C_1\Log z_{k_1} + \dots + C_{N} \Log z_{k_{N}} + i n_2\pi $$
for some $n_2\in \QQ$ (see Lemma \ref{pretildexinter} for a similar result under Definition \ref{defgenFAMED5}(1) and (2)). It is also verified in \cite[Lemma 5.1]{BAW} that $n_1=0$ for the triangulations of hyperbolic twist knot complements studied in \cite{BAGPN}. Since we expect that $n_1=0$ for all hyperbolic knot complements in $\SS^3$, we include condition (3) as a part of Definition \ref{defgenFAMED5}.
\end{remark}

We stress the fact that there exist ordered triangulations that are generalized FAMED with respect to $l$ but fail Definition \ref{defgenFAMED5}(2). Nevertheless, we propose the following conjecture about the generality of the conditions in Definition \ref{defgenFAMED5}. 
\begin{conjecture}\label{allgenFAMED}
    Every hyperbolic knot complement in $\SS^3$ admits a semi-geometric triangulation that is generalized FAMED with respect to $(l,m)$.
\end{conjecture}

\begin{remark}
    The partition function in Teichm\"uller TQFT can also be defined for certain knot complements in manifolds other than $\SS^3$, such as hyperbolic once-punctured torus bundle. In those cases, Definition \ref{defgenFAMED} and \ref{defgenFAMED5} should be modified by replacing the longitude and the meridian with a suitable pair peripheral curves that generate the fundamental group of the boundary torus. For simplicity, in this paper, we only consider hyperbolic knot complements in the three-sphere.
\end{remark}

\begin{remark}
The conditions in Definition \ref{defgenFAMED} and \ref{defgenFAMED5} may look complicated. Nonetheless, the conditions can be described by elementary terms, be checked by using a computer, and are good enough to write a formal mathematical proof for our main results. Computational supporting evidence for Conjecture \ref{allgenFAMEDl} and \ref{allgenFAMED} will be provided in a subsequent preprint \cite{WC}.
\end{remark}

\subsection{Asymptotics aspects of Teichm\"uller TQFT}\label{sub:intro:expansion}
Our first main result concerns the asymptotics of the partition function in Teichm\"uller TQFT. Let $K$ be a hyperbolic knot in $\SS^3$. Let $X$ be an ordered ideal triangulation of $\SS^3 \setminus K$. 
Let $\alpha \in \mathscr{A}_X$ be an angle structure of $X$. For $\hbar>0$, let $\mathscr{Z}_{\hbar}(X, \alpha)$ be the Teichm\"uller TQFT partition function for $(X,\alpha)$ defined in \cite{AK} (see Section \ref{defAKTQFT} for a review.) Let $\mathrm{H}^\R_{X,l}(\alpha)$ be the angular holonomy associated to the curve $l$ and the angle structure $\alpha$ (see Section \ref{sub:angle:str}).
\begin{theorem}\label{mainthmZ}
Suppose $X$ is generalized FAMED with respect to $l$ (see Definition \ref{defgenFAMED}). Then we have
\begin{align*}
\limsup_{\hbar\to 0} 2\pi \hbar \log|\mathscr{Z}_{\hbar}(X, \alpha) |
\leq -\sup\{ \Vol(\alpha') \mid \alpha' \in \mathscr{A}_X, 
\mathrm{H}^\R_{X,l}(\alpha')= \mathrm{H}^\R_{X,l}(\alpha)
\},
\end{align*}
where $\Vol(\alpha')$ is the sum of hyperbolic volumes of the tetrahedra of $X$ with the angle structure $\alpha'$.
If we further assume that there exist shape parameters $\mathbf{z}$ with non-negative imaginary parts that satisfy Equation (\ref{introglu}) with $\xi = i\mathrm{H}^\R_{X,l}(\alpha)$, then we have
\begin{align*}
&\ \ |\mathscr{Z}_{\hbar}(X, \alpha) | \\
=&\ 
\left| C \cdot\exp\left( \frac{i}{\pi}R(\mathbf{z})\right) \cdot \frac{\exp\Big( -\frac{1}{2\pi {\hbar}}\Vol\left(\SS^3\setminus K; l,\mathrm{H}^\R_{X,l}(\alpha)\right)\Big)}{\sqrt{\pm \tau(\SS^3\setminus K, l, X, \mathbf{z})}}  \Big(1 + O_{{\hbar}\to 0^+}({\hbar})\Big) \right|,
\end{align*}
where 
\begin{itemize}
\item $\Vol\left (\SS^3\setminus K; l,\mathrm{H}^\R_{X,l}(\alpha)\right)=\Vol(\mathbf{z})$ is the volume of $\SS^3\setminus K$ with the (possibly incomplete) hyperbolic cone structure satisfying $\mathrm{H}^\C_{X,l}(\mathbf{z})=i \mathrm{H}^\R_{X,l}(\alpha)$, i.e. the sum of hyperbolic volumes of the tetrahedra of $X$ under the shape structure $\mathbf{z}$,
\item $\tau(\SS^3\setminus K, l, X, \mathbf{z})$ is the 1-loop invariant of $\SS^3 \setminus K$ associated to the shape parameters $\mathbf{z}$ and the peripheral curve $l$ (see Section \ref{1loopsection} for details), 
\item $R$ is some non-zero analytic function independent of $\hbar$ (see \ref{hvalue} for the definition), and
\item $C$ is some non-zero constant independent of $\hbar$.
\end{itemize}
\item In particular, we have
\begin{align*}
\lim_{\hbar\to 0} 2\pi \hbar \log|\mathscr{Z}_{\hbar}(X, \alpha) |
= -\Vol\left(\SS^3\setminus K; l, \mathrm{H}^\R_{X,l}(\alpha)\right ).
\end{align*}
\end{theorem}
\begin{remark}
We refer readers to see \cite[Section 1.5]{BAW} for a possible application of Theorem \ref{mainthmZ} to study the Casson conjecture.
\end{remark}
Theorem \ref{mainthmZ} generalizes \cite[Theorem 1.5]{BAW} by respectively replacing the FAMED condition and the geometricity of the triangulations with the generalized FAMED condition and semi-geometricity. We highlight several new ideas in the proof of Theorem \ref{mainthmZ}. First, in Lemma \ref{KKdelta}, we find a new formula to compute the kinematical kernel that generalizes the formula in \cite[Lemma 5.5]{BAGPN}. Together with the generalized FAMED condition, we can write the partition function and its potential function using the Neumann-Zagier matrices (Proposition \ref{Tpartiexpress2}). The generalized FAMED condition connects the partition function with the geometry of the triangulation by providing a correspondence between the critical point equations of the potential function and Equation (\ref{introglu2}) (see Proposition \ref{critThurscorrespondence}). For a semi-geometric triangulation, the critical point that corresponds to the solution of Equation (\ref{introglu2}) could lie outside the domain of definition of the integrand of the partition function. To overcome this problem, following the approach in \cite{LMSWY}, we holomorphically extend the domains of the classical dilogarithm function and the quantum dilogarithm function to $\CC$ with two branch cuts (see Section \ref{sectclassicaldilog}-\ref{sectquantumdilog}). We show in Proposition \ref{dilogVolgen} and \ref{Hesstotor} that in the extended domain, the critical value and the determinant of the Hessian matrix of the potential function give negative of the hyperbolic volume and the 1-loop invariant respectively. To apply the saddle point approximation, we need to construct an integration contour satisfying several technical assumptions. In \cite{BAW}, the construction of the contour is based on the strict concavity of the real part of the potential function on the domain of the integrand. However, in the extended domain, we no longer have the strict concavity. We overcome this problem by using the volume rigidity theorem \cite[Theorem 1.2 and 1.4]{SF} and the gradient flow (see Proposition \ref{constructZcont} and \ref{Zcont}). 

Next, in \cite{AK,AKnew,AKicm}, the \textbf{Andersen--Kashaev volume conjecture} predicts that there exists a Jones function $\mathfrak{J}_{X}$, which is an analogue of the Kashaev invariant of a knot $K\subset \SS^3$, such that the Andersen--Kashaev TQFT is roughly a Laplace transform of this Jones function, and such that this Jones function asymptotically yields the hyperbolic volume. For simplicity, we only consider the conjecture for hyperbolic knot complements in $\SS^3$. Note that our Jones function $\mathfrak{J}_X$ and that in \cite{AK, BAGPN} are related by a change of variable.
\begin{conjecture}\cite[Conjecture 1(1)\&(3)]{AK}  \label{conj:vol:BAW}
For every hyperbolic knot $K\subset \SS^3$, there exist an ideal triangulation $X$ of $\SS^3 \setminus K$ with non-empty space of angle structure $\mathcal{A}_X$, together with a Jones function $\mathfrak{J}_X\colon \R_{>0} \times \mathcal{U} \to \C$ (where $\mathcal{U}\subset \C$ is an open horizontal band in $\C$),
such that the following properties hold:
\begin{enumerate}
\item  For all angle structures $\alpha \in \mathcal{A}_{X}$ and all $\hbar>0$, we have:
\begin{equation*}
\left |\mathcal{Z}_{\hbar}(X,\alpha) \right |
= \left | 
\int_{\mathbb{R} +i\mu_X(\alpha)} 
\mathfrak{J}_{X}(\hbar,\mathrm{w})
e^{\frac{\mathrm{w}  \lambda_{X}(\alpha)}{4\pi \hbar}  } 
d\mathrm{w} \right |,
\end{equation*}
where $\mu_X(\alpha)$ and $\lambda_X(\alpha)$ are respectively the angular holonomies of the meridian and the preferred longitude with respect to the angle structure $\alpha$.
\item In the semi-classical limit $\hbar \to 0^+$, we retrieve the hyperbolic volume of $K$ as:
$$
\lim_{\hbar \to 0^+} 2\pi \hbar  \log \vert \mathfrak{J}_{X}(\hbar,0) \vert
= -\Vol(\SS^3 \setminus K).$$
\end{enumerate}
\end{conjecture}

Given $K\subset \SS^3$, let $\nu(K)$ be an open tubular neighborhood of $K$ in $\SS^3$. Recall from \cite{NZ} that for any simple closed curve $\gamma \in \pi_1(\partial (\SS^3 \setminus \nu(K)))$, locally near the complete hyperbolic structure, the deformation space of hyperbolic structure of $\SS^3 \setminus K$ can be parametrized in a generically 2:1 way with one complex variable $w_\gamma$ 
such that, for any complex number $\xi \in \C$ close enough to $0$, the point $w_{\gamma}=\xi$ corresponds to the character with a representative  $\rho:\pi_1(\SS^3 \setminus K) \to \mathrm{PSL}_2(\CC)$ that sends $[\gamma] \in \pi_1(\partial (\SS^3 \setminus \nu(K)))$ to the matrix of the form
$$ \rho([\gamma]) = \pm \begin{pmatrix}
    e^{\frac{\xi}{2}} & * \\
    0 & e^{-\frac{\xi}{2}}
\end{pmatrix}.$$ 
 Besides, for the meridian $m$ and longitude $l$ of $K$, since $w_m$ and $w_l$ are both parameterizations of the same space, one can regard $w_l$ as a function in $w_m$ by considering the transition function. The Neumann-Zagier potential function associated to $(m,l)$ can be defined by using the antiderivative of $w_l$ as a function in $w_m$ (see \cite{NZ} and Section \ref{NZpotentintro} for a review). Our second result relates the asymptotics of Jones functions with the Neumann-Zagier potential functions as follows.

\begin{theorem}
\label{thm:Jones:genFAMED}
Let $K$ be a hyperbolic knot in $\SS^3$
and let $X$ be an ordered ideal triangulation of 
$\SS^3 \setminus K$. If $X$ is generalized FAMED with respect to $(l,m)$, then Conjecture \ref{conj:vol:BAW}(1) holds. If we further assume that $X$ is semi-geometric with shape parameters $\mathbf{z}$, then there exist a small neighborhood $\mathcal{O}\subset \CC$ around $0 \in \CC$ such that 
for any $\mathrm{w} \in \mathcal{O}$, 
we have 
\begin{align*}
|\mathfrak{J}_{X}(\hbar, \mathrm{w})|
= \left|\frac{C}{2\pi {\sqrt{\hbar}} } \cdot\exp\left(\frac{i}{\pi} R(\mathbf{z}_\mathrm{w})\right) \cdot \frac{\exp\left(\frac{i}{2\pi {\hbar}}\phi_{m,l}(\mathrm{w})\right)}{\sqrt{\pm \tau(\SS^3\setminus K, m, X, \mathbf{z}_\mathrm{w})}}  \Big(1 + O_{{\hbar}\to 0^+}(\hbar)\Big)\right|,
\end{align*} 
where 
\begin{itemize}  
\item $\mathbf{z}_\mathrm{w}$ is a continuous family of shape parameters satisfying the edge equations such that $\mathbf{z}_\mathrm{0}=\mathbf{z}$ and $\mathrm{H}_{X,m}^\CC(\mathbf{z}_\mathrm{w}) = \mathrm{w}$, where $\mathrm{H}_{X,m}^\CC(\mathbf{z}_\mathrm{w})$ is the logarithmic holonomy of $m$,
\item $\phi_{m,l}: \mathcal{O} \to \CC$ is the Neumann-Zagier potential function associated to $m,l$ (see Section \ref{NZpotentintro}), i.e. the holomorphic function satisfying
$$
\frac{\partial \phi_{m,l}}{\partial w_{m}} = \frac{w_l}{2} ,\quad \phi_{m,l}(0) = i(\Vol(\SS^3\setminus K) + i \CS(\SS^3\setminus K)),
$$ 
where $w_m, w_l$ are the logarithmic holonomies of $m$ and $l$ and $\CS(\SS^3\setminus K)$ is the Chern--Simons invariant of $\SS^3\setminus K$ respectively, 
\item $\tau(\SS^3\setminus K, m, X, \mathbf{z}_\mathrm{w})$ is the 1-loop invariant of $\SS^3\setminus K$ associated to the meridian $m$ and the shape parameters $\mathbf{z}_\mathrm{w}$  (see Section \ref{1loopsection} for details)
\item $R$ is some analytic function that is independent of $\hbar$ and is non-zero for all $\mathbf{z}_\mathrm{w}$ with $\mathrm{w} \in \mathcal{O}$ (see \ref{hvalue} for the definition), and
\item $C$ is a non-zero constant independent of $\hbar$ and $\mathrm{w}$.
\end{itemize}
In particular, we have
$$
 \lim_{\hbar\to 0} 2\pi \hbar \log|\mathfrak{J}_{X}(\hbar, 0)|
= - \Vol(M).
$$
\end{theorem}
\begin{remark}
We refer readers to see \cite[Section 1.3]{BAW} for a possible application of Theorem \ref{mainthmZ} to study the AJ conjecture in Teichm\"uller TQFT.
\end{remark}
\begin{corollary}\label{AKcor}
Suppose a hyperbolic knot $K\subset \SS^3$  admits a semi-geometric triangulation that is generalized FAMED with respect to $(l,m)$. Then 
    Conjecture \ref{conj:vol:BAW} holds for $K$.
\end{corollary}
In contrast to the result in \cite[Theorem 1.6]{BAW}, when $\mathrm{w} \neq 0$, Theorem \ref{thm:Jones:genFAMED} still holds even if some of the shape parameters have negative imaginary parts. The key idea in the proof is to first apply the volume rigidity theorem \cite[Theorem 1.2 and 1.4]{SF} and a gradient flow argument to construct a suitable integration contour for the case where $\mathrm{w}=0$. This allows us to apply the saddle point methods to obtain the asymptotics of the Jones functions and prove Theorem \ref{thm:Jones:genFAMED} for $\mathrm{w}=0$. When $\mathrm{w}\neq 0$, we apply Complex Morse Lemma (see Lemma \ref{OPFCML}) to modify the integration contour constructed above (see Proposition \ref{constructJcont}, \ref{Jcont} and \ref{Jcont2}) and verify that the modified contour satisfies the technical conditions for applying the saddle point approximation. 

It is known that every hyperbolic knot complement in $\SS^3$ has a semi-geometric triangulation that admits angle structures \cite{HRS}.  If such a triangulation satisfies Definition \ref{defgenFAMED5} with respect to some sutiable ordering, then Theorem \ref{thm:Jones:genFAMED} and Corollary \ref{AKcor} apply, and Conjecture \ref{conj:vol:BAW} follows. In this sense, the study of the asymptotics of the partition functions and the Jones functions is reduced into the combinatorial problems about understanding the generalized FAMED conditions described in Definition \ref{defgenFAMED} and \ref{defgenFAMED5}.

\section*{Acknowledgements}
I would like to thank Fathi Ben-Aribi, Feng Luo, Andrew Nietzke and Tian Yang for valuable discussions and suggestions. 

\section{Preliminaries}\label{sectprelim}

\subsection{Triangulations}

In this section, we follow the notations in \cite{AK, KaWB}. An ordered tetrahedron $T$ is a tetrahedron with a choice of an \textit{order} on the four vertices of $T$. We call them $0_T,1_T,2_T,3_T$ (or $0,1,2,3$ if the context makes it obvious). If we rotate $T$ such that $0$ is in the center and $1$ at the top, then there are two possible places for vertices $2$ and $3$. We call $T$ a \textit{positive} tetrahedron if they are as shown on the left figure in the top row of Figure \ref{fig:41:face:matrices}, and a \textit{negative} tetrahedron otherwise. We denote $\varepsilon(T) \in \{ \pm 1\}$ the corresponding \textit{sign} of $T$. We \textit{orient the edges} of $T$  {according} to the order on vertices, and we endow each edge with a parametrisation by $[0,1]$ respecting the orientation. Note that such a structure was called a \textit{branching} in \cite{BB} and equivalently described as an orientation of edges without $3$-cycles on faces. Up to isotopies fixing the $1$-skeleton pointwise, there is only one way of \textit{gluing} two triangular faces together while \textit{respecting the order of the vertices} and the edge parametrisations, and that is the only type of face gluing we consider in this paper.

An \textit{ordered triangulation} $X=(T_1,\ldots,T_N,\sim)$ consists of $N$ ordered tetrahedra $T_1, \ldots, T_N$ and an equivalence relation $\sim$ first defined on the faces by pairing and the only gluing that respects vertex order, and also induced on edges  {and} vertices by the combined identifications. We call $M_X$ the (pseudo-)$3$-manifold 
$ M_X = T_1 \sqcup \cdots \sqcup T_N / \sim$ obtained by quotient. Note that $M_X$ may fail to be a manifold only at  (the image by the quotient map of) a vertex  of the triangulation, whose regular  {neighbourhood} might not be a $3$-ball (but for instance a cone over a torus for exteriors of links).

We denote $X^{k}$ (for $k=0, \ldots, 3$) the set of $k$-cells of $X$ after identification by $\sim$. In this paper we always  {assume} that \textit{no face is left unpaired  by $\sim$}, thus $X^{2}$ is always of  {cardinality} $2N$. By a slight abuse of notation we also call $T_j$ the $3$-cell inside the tetrahedron $T_j$, so that $X^{3} = \{T_1, \ldots, T_N\}$. Elements of $X^{1}$ are usually represented by distinct types of arrows, which are drawn on the corresponding preimage edges, see Figure \ref{fig:41:face:matrices} for an example.

An \textit{ordered ideal triangulation} $X$ is an ordered triangulation such that all tetrahedra are ideal tetrahedra. In this case the quotient space minus its vertices $M_X \setminus X^0$ is an open manifold. In this case we will denote $M=M_X \setminus X^0$ and say that the open manifold $M$ admits the ideal triangulation $X$.


Finally, for $X$ a triangulation and $k=0,1,2,3,$ we define $x_k\colon X^3 \to X^2$ the  {function} such that $x_k(T)$ is the equivalence class of the face of $T$ opposed to its vertex $k$.

\begin{example}\label{ex:41}
Figure \ref{fig:41:face:matrices} displays the classical ideal triangulation of the complement of the figure-eight knot $M=S^3 \setminus 4_1$, with one positive and one negative tetrahedron. Here $X^3 =\{T_+, T_-\}$, $X^2 =\{\mathsf{A},\mathsf{B},\mathsf{C},\mathsf{D}\}$, $X^1 =\{ \sarrow , \darrow \}$ and $X^0$ is a singleton. 
\end{example}

\subsection{Angle structures}\label{sub:angle:str}

For a given triangulation $X=(T_1,\ldots,T_N,\sim)$ we denote $\mathcal{S}_X$ the set of \textit{shape structures on $X$}, defined as
$$
\mathcal{S}_X 
 = 
 \left \{ 
\alpha = \left (a_1,b_1,c_1,\ldots,
a_N,b_N,c_N\right ) \in (0,\pi)^{3N} \ \big | \
\forall k \in \{1,\ldots,N\}, \
a_k+b_k+c_k = \pi
\right \}.
$$
An angle $a_k$ (respectively $b_k,c_k$) represents the value of a dihedral angle on the edge $\overrightarrow{01}$ (respectively $\overrightarrow{02}$, $\overrightarrow{03}$) and its opposite edge in the tetrahedron $T_k$. If a particular shape structure $\alpha=(a_1,\ldots,c_N)\in \mathcal{S}_X$ is fixed, we define three associated functions $\alpha_j\colon X^3 \to (0,\pi)$ (for $j=1,2,3$) that send $T_k$ to the $j$-th element of $\{a_k,b_k,c_k\}$ for each $k \in \{1,\ldots,N\}$.

Let $(X,\alpha)$ be a triangulation with a shape structure as before. We denote $\omega_{X,\alpha}\colon X^1 \to \mathbb{R}$ the associated \textit{weight function}, which sends an edge $e\in X^1$ to the sum of angles $\alpha_j(T_k)$ corresponding to tetrahedral edges that are preimages of $e$ by $\sim$. 
For example, if we denote 
$\alpha=(a_+,b_+,c_+,a_-,b_-,c_-)$ a shape structure on the triangulation $X$ of Figure \ref{fig:41:face:matrices}, then $\omega_{X,\alpha}(\sarrow) =
2 a_+ + c_+ + 2 b_- + c_-.$ The space of angle structures on $X$ is defined by
$$
\mathcal{A}_X := \left \{
\alpha \in \mathcal{S}_X \ \big | \
\forall e \in X^1, \ \omega_{X,\alpha}(e)=2\pi
\right \}.
$$
Let $\sigma$ an oriented normal closed curve on the boundary torus $\partial( \SS^3 \setminus \nu(K))$. Truncating the tetrahedra $T_j$ at each vertex yields a triangulation of the boundary torus by triangles coming from vertices of $X$ (called the \textit{cusp triangulation}). If the curve $\sigma$ intersects these triangles transversely (without back-tracking), then $\sigma$ cuts off corners of each such encountered triangle. The \textit{angular holonomy} $\mathrm{H}^\R_{X,\sigma}(\alpha)$ of $\sigma$ is defined to be the signed sum of those angles, where the sign of an angle is equal to $+1$ if the cut corner lies on the left of $\sigma$ and is equal to $-1$ otherwise. As an example, in Figure \ref{alinear}, if we denote the red curve by $l$, then we have 
$ \mathrm{H}^{\RR}_{X,l}(\alpha) = a_1 - a_2 + a_3 - a_4 + a_5 - a_6$.

\subsection{Shape parameters}\label{sub:thurston}

To a shape structure $(a,b,c)$ on an ordered tetrahedron $T$, we can associate bijectively a \textit{complex shape structure} $z \in \R+i\R_{>0}$, as well as two companion complex numbers of positive imaginary part
$$z':=\frac{1}{1-z} \text{\ and \ } z'':=\frac{z-1}{z}.$$
Each of the $z, z', z''$ is associated to an edge, in a slightly different way according to $\varepsilon(T)$:
\begin{itemize}
\item In all cases, $z$ corresponds to the same two edges as the angle $a$.
\item If $\varepsilon(T)=1$, then $z'$ corresponds to $c$ and $z''$ to $b$.
\item If $\varepsilon(T)=-1$, then $z'$ corresponds to $b$ and $z''$ to $c$.
\end{itemize}
Another way of phrasing it is that $z, z', z''$ are always in a counterclockwise order around a vertex, whereas $a,b,c$ need to follow the specific vertex ordering of $T$.

In this article we will use the following definition of the complex logarithm:
\[
\Log(z) := \log\vert z \vert + i\arg(z)  \ \textrm{for} \ z \in \C^{*},
\]
where $\arg(z) \in (-\pi,\pi]$. 

We will actually allow some tetrahedra in the triangulation to be flat. For a shape parameter $z \in \CC\setminus\{0,1\}$ with $\im z \geq 0$, we let
$$y := \Log(z) - i \pi \in \R  - i [0,\pi],$$
 which lives in a horizontal strip of the complex plane. In particular, we have 
 $$z = -e^{y}, \ \ \Log(z) = -y + i\pi, \ \ \Log(z') = -\Log(1+e^{y}) \ \ \text{and} \ \ \Log(z'') = \Log(1+e^{-y}).$$

\subsection{The classical dilogarithm}\label{sectclassicaldilog}

For the dilogarithm function, we will use the definition:
\begin{align}\label{defLi}
\Li(z) := - \int_0^z \Log(1-u) \frac{du}{u} \ \ \ \textrm{for} \ z \in \C \setminus [1,\infty)
\end{align}
(see for example \cite{Za}).
For $z$ in the unit disk, $\Li(z)=\sum_{n\geq 1} n^{-2} z^n$.
We will use the following properties of the dilogarithm function, referring for example to \cite[Appendix A]{AH} for the proofs.

\begin{proposition}[Properties of $\Li$]\label{prop:dilog}
\

\begin{enumerate}
\item (inversion relation) $$ \forall z \in \C \setminus [1,\infty), \
 \Li\left (\frac{1}{z}\right ) = - \Li(z) - \frac{\pi^2}{6} - \frac{1}{2}\Log(-z)^2.
$$
\item (integral form) For all $y \in \R +i(-\pi,\pi)$,
$$ \frac{-i}{2 \pi} \Li(-e^y) =
\int_{v \in \R + i 0^+}
\dfrac{\exp\left (-i \frac{y v}{\pi}\right )}{4 v^2 \sinh(v)} \,  dv.
$$
 In the previous formula and in the remaining of the paper, $\R + i 0^+$ denotes a contour in $\C$ that is deformed from the horizontal line $\R \subset \C$ by avoiding $0$ via the upper half-plane (with a small half-circle for example).
\end{enumerate}
\end{proposition}

Consider the function $\Li(-e^{y})$ defined on $\RR + i (-\pi,\pi)$. Following \cite{LMSWY}, we extend the function analytically to a function $\mathrm{L}(y)$ defined on $\CC \setminus ( i(-\infty,-\pi) \cup i(\pi,\infty))$ as follows. Note that our function $\mathrm{L}$ is different from the function $L$ in \cite[Equation (2.7)]{LMSWY}. 

First, since $\Li$ is analytic inside the unit disk, for $y$ with $\Re y < 0$, we can extend the definition of $\Li(-e^{y})$ periodically to $\RR_{<0} + i \RR$ periodically with period $2\pi i$ by 
$$ \mathrm{L}(y+2\pi i) = \mathrm{L}(y)$$
for all $y \in \RR_{<0} + i \RR$. In particular, for any $y$ in this region, the derivative is also $2\pi i$ periodic. 

Next, for $y$ with $\Re y>0$, note that $-e^{y}$ approach to the branch cut $[1,\infty)$ of $\Li(z)$ when $y$ approach to $\RR_{>0}-i\pi$ or $\RR_{>0} + i\pi$. One can extend the function analytically by changing the branch cut of the logarithm in (\ref{defLi}) continuously. Besides, by Proposition \ref{prop:dilog} (1), for all $y \in \RR + i(-\pi,\pi)$, we have
$$
\Li(-e^{y}) = - \Li(-e^{-y}) - \frac{\pi^2}{6} - \frac{y^2}{2}.
$$
In particular, for any $y= a + ib$ with $a>0$ and $b\in (-\pi,\pi)$, we have
$$
\Li(-e^{a\pm bi}) = - \Li(-e^{-a\mp bi}) - \frac{\pi^2}{6} - \frac{(a\pm bi)^2}{2},
$$
which implies that
\begin{align}\label{intoout}
    \Li(-e^{a + bi}) - \Li(-e^{a - bi})
= - \left( \Li(-e^{-a- bi}) - \Li(-e^{-a + bi}) \right)
- 2abi. 
\end{align}
By continuity, if we take $b\to \pi$, since $-a<0$, by peridocitiy we have
$$
\lim_{b\to \pi}  \left( \Li(-e^{-a- bi}) - \Li(-e^{-a + bi}) \right)
= \mathrm{L}(-a - \pi i) - \mathrm{L}(-a + \pi i)
= 0.
$$
Hence, by taking $b\to \pi$ in (\ref{intoout}), we have
$$
\mathrm{L}(a+\pi i) = \mathrm{L}(a-\pi i) - 2\pi i a 
$$
for all $a>0$. By analyticity, for every $y\in \CC$ with $\Re y>0$, $\mathrm{L}(y)$ satisfies the functional equation
\begin{align}\label{dilogFE}
    \mathrm{L}(y+2\pi i)  = \mathrm{L}(y) - 2\pi i (y+\pi i). 
\end{align}
In this region, we have
\begin{align}\label{L'period}
    \mathrm{L}'(y + 2\pi i) = \mathrm{L}'(y)  - 2\pi i .
\end{align}
Since 
$$\frac{d}{dy} \Li(-e^{y}) = - \Log (1+e^{y})$$
for $y \in \RR_{>0} + i(-\pi, \pi)$, $\mathrm{L}'(y)$ coincides with the analytic continuation of $-\Log (1+e^{y})$. Finally, from (\ref{L'period}), for any positive integer $n\geq 2$, the higher derivative of $L$ satisfies
\begin{align*}
    \mathrm{L}^{(n)}(y+2\pi i) = \mathrm{L}^{(n)}(y).
\end{align*}

The function $\mathrm{L}(z)$ is continuous at $z= \pm i\pi$ but not holomorphic. The next lemma describes the behavior of the function near this singularity.
\begin{lemma}\label{singularbehaviour} Let $y= a +i b$.
    We have
    $$ \lim_{\substack{y \to -i\pi\\ \im y > -\pi}}
    \frac{\partial}{\partial b} (\mathrm{Re} ( i \Li(-e^{y}))) = -\infty.$$
    In particular, near $y=-i\pi$, $\mathrm{Re} (i \Li(e^{y}))$ decreases when $b$ increases.
\end{lemma}
\begin{proof}
    By Cauchy-Riemann equation, we have
    \begin{align*}
         \frac{\partial}{\partial b} (\mathrm{Re} (i \Li(-e^{y})) )
         = -\im\frac{\partial}{\partial y} (i \Li(-e^{y}))
         = \im \left( i\log(1+e^{y}) \right)
         = \log|1+e^{y}|
         \to -\infty.
    \end{align*}
    as $y \to -i\pi$. This completes the proof.
\end{proof}

\subsection{The Bloch--Wigner function}

The \emph{Bloch--Wigner function} $D:\C \rightarrow \R $ is defined by
\[
D(z) := \Im(\Li(z)) + \arg(1-z)\log \vert z \vert \quad \text{ if $z\in \C \smallsetminus \R$, and $0$ otherwise.}
\]
$D$ is continuous on $\C$, and real-analytic on $\mathbb{C} \backslash \{0,1\}$ (see \cite[Section 3]{Za} for details). 
\begin{proposition}
Let $T$ be an ideal tetrahedron in $\mathbb{H}^3$ with complex shape structure $z$. Then, its volume is given by
\[
\Vol(T)= D(z) = D \left( \frac{z-1}{z} \right) = D \left( \frac{1}{1-z} \right).
\]
\end{proposition}

The next lemma relates the Bloch--Wigner function with the analytically continued dilogarithm function.
\begin{lemma}\label{tildeLitoD}
For any $y \in \CC \setminus ( i(-\infty,-\pi) \cup i(\pi,\infty))$,
    \begin{align*}
        \im \left(\mathrm{L}(y)\right) - \left(\im \mathrm{L}'(y)\right)\log\left|-e^{y}\right|
        = D(-e^{y}).
    \end{align*}
\end{lemma}
\begin{proof}
    First, we claim that on $\CC \setminus i\RR$, the left hand side of the equation is $2\pi i$ periodic. For $y$ with $\Re y<0$, this follows from the fact that $\mathrm{L}(y) = \mathrm{L}(y+2\pi i)$ and $\im \mathrm{L}'(y+2\pi i) = \im \mathrm{L}'(y)$. 
    For $y$ with $\Re y >0$, from (\ref{dilogFE}), since $\mathrm{L}(y+2\pi i) = \mathrm{L}(y) - 2\pi i (y+\pi i)$, we have
    $$ \im \left(\mathrm{L}(y+2\pi i)\right) = \im \left(\mathrm{L}(y)\right) - 2\pi \Re y = \im \left(\mathrm{L}(y)\right) - 2\pi\log\left|-e^{y}\right|.$$
    Besides,
    $$ \im\mathrm{L}'(y+2\pi i) = \im\left(\mathrm{L}'(y) - 2\pi i\right) = \im\mathrm{L}'(y) - 2\pi. $$
    Thus, we have
    \begin{align*}
     &\ \im \left(\mathrm{L}(y+2\pi i)\right) - \left(\im\mathrm{L}'(y+2\pi i)\right)\log\left|-e^{-(y+2\pi i)}\right| \\
    =&\  \left(\im \left(\mathrm{L}(y)\right) - 2\pi\log\left|-e^{y}\right|\right)
    - \left(\im\mathrm{L}'(y) - 2\pi \right)\log\left|-e^{y}\right|\\
    =&\ \im \left(\mathrm{L}(y)\right) - \left(\im\mathrm{L}'(y)\right)\log\left|-e^{y}\right|.
    \end{align*}
    This proves the claim. Note that for $y\in \RR + i (-\pi, \pi)$, we have
    \begin{align*}
        \im \left(\mathrm{L}(y)\right) -\left(\im\mathrm{L}'(y)\right)\log\left|-e^{y}\right|  
        = \im \left(\Li(-e^{y})\right) + \arg\left(1+e^{y}\right)\log|-e^{y}| 
        = D(-e^{y}).
    \end{align*}
    By continuity, the above equation still holds for $y\in \RR + i [-\pi, \pi]$. By periodicity on $\CC \setminus i\RR$, the equation holds for all $y\in \CC \setminus ( i(-\infty,-\pi) \cup i(\pi,\infty))$. This completes the proof.
\end{proof}

\subsection{Faddeev's quantum dilogarithm}\label{sectquantumdilog}
Given $\hbar >0$ and $\B >0$ such that $$(\B+\B^{-1}) \sqrt{\hbar} = 1,$$
 the \emph{Faddeev's quantum dilogarithm} $\Phi_\B$ is the holomorphic function on $\R + i \left (\frac{-1}{2 \sqrt{\hbar}}, \frac{1}{2 \sqrt{\hbar}}\right )$ given by
$$
\Phi_\B(z) = \exp\left (
\frac{1}{4} \int_{w \in \R + i 0^+}
\dfrac{e^{-2 i z w} dw}{\sinh(\B w) \sinh({\B}^{-1}w) w}
\right ) \ \ \ \ \text{for} \ z \in \R + i \left (\frac{-1}{2 \sqrt{\hbar}}, \frac{1}{2 \sqrt{\hbar}}\right ),
$$
where $\R + i 0^+$ denotes a contour in $\C$ that is deformed from the horizontal line $\R \subset \C$ by avoiding $0$ by above. 
The quantum dilogarithm function can be extended to a meromorphic function for $z\in \C$ via the functional equation 
\begin{align}\label{qdilogFE}
    \Phi_\B\left (z-i \frac{\B^{\pm  1}}{2}\right )= \left (1+e^{2\pi \B^{\pm 1} z}\right )
\Phi_\B\left (z + i \frac{\B^{\pm 1}}{2}\right ).
\end{align}
In particular, we have
$$
\Phi_\B\left(\frac{z}{2\pi \B} \right) 
= \left(1+e^{\frac{z+\pi i}{\B^2}}\right) \Phi_\B\left(\frac{z+2\pi i}{2\pi \B} \right) .
$$

Note that $\Phi_\B$ depends only on $\hbar = (\B+\B^{-1})^{-2}$. Furthermore, as a consequence of the functional equation, the poles of $\Phi_\B$  lie on $ i  [1/(2 \sqrt{\hbar}), \infty ) $ and the zeroes lie symmetrically on $i  (-\infty, -1/(2 \sqrt{\hbar}) ]$. We now list several useful properties of Faddeev's quantum dilogarithm. We refer to \cite[Appendix A]{AK} for these properties (and several more), and to \cite[Lemma 3]{AH} for an alternate proof of the semi-classical limit property.

\begin{proposition}[Properties of $\Phi_\B$]\label{prop:quant:dilog}\text{}\\
\begin{enumerate}
\item (inversion relation) For any $\B \in \R_{>0}$ and any  $z \in \R + i \left (\frac{-1}{2 \sqrt{\hbar}}, \frac{1}{2 \sqrt{\hbar}}\right )$, 
$$\Phi_\B(z) \Phi_\B(-z) = e^{i\frac{\pi}{12}(\B^2 + \B^{-2})} e^{i \pi z^2}.$$
\item (unitarity) For any $\B \in \R_{>0}$ and any  $z \in \R + i \left (\frac{-1}{2 \sqrt{\hbar}}, \frac{1}{2 \sqrt{\hbar}}\right )$, 
$$\overline{\Phi_\B(z)} = \frac{1}{\Phi_\B(\overline{z})}.$$
\item (semi-classical limit in $\B$) For any $z \in \R + i \left (-\pi,\pi \right )$,
$$\Phi_\B\left (\frac{z}{2 \pi \B}\right ) = \exp\left (\frac{-i}{2 \pi \B^2} \Li (- e^z)\right ) \left ( 1 + O_{\B \to 0^+}(\B^2)\right ).$$
\item (behavior at infinity) For any  $\B \in \R_{>0}$,
\begin{align*}
 \Phi_\B(z) \ \ \underset{\Re(z)\to -\infty}{\sim} & \ \ 1, \\
 \Phi_\B(z) \ \  \underset{\Re(z)\to \infty}{\sim} & \ \ e^{i\frac{\pi}{12}(\B^2 + \B^{-2})} e^{i \pi z^2}.
\end{align*}
In particular, for any  $\B \in \R_{>0}$ and any $d \in \left (\frac{-1}{2 \sqrt{\hbar}}, \frac{1}{2 \sqrt{\hbar}}\right )$,
\begin{align*}
|\Phi_\B(x+id) | \ \ \underset{\R \ni x \to -\infty}{\sim} & \ \ 1, \\
|\Phi_\B(x+id) | \ \  \underset{\R \ni x \to +\infty}{\sim} & \ \ e^{-2 \pi x d}.
\end{align*}
\end{enumerate}
\end{proposition}

Following the idea in \cite{LMSWY}, we relate the analytic continuation of the quantum and classical dilogarithm functions as follows. For $z_1,z_2\in \CC,$ let 
$$||z_1 - z_2 ||_\infty = \max\{ |\Re(z_1-z_2)|, |\im(z_1-z_2)|\}$$
be the sup-norm of $z_1-z_2$. Given $\delta>0$, let $L_\delta$ be the $\delta$-neighborhood of $i(-\infty,-\pi) \cup i(\pi,\infty)$ in $\CC$. 

The following lemma gives a uniform estimate of the difference between quantum and classical dilogarithm functions. The proof can be found in \cite[Lemma 2.2 and Lemma 2.3]{LMSWY}. 
\begin{lemma}\label{unifoutsidebox}
For any $\delta>0$, there exists $B_\delta >0$ such that for all $\B \in (0,1)$ and for all $z \in (\RR + i [-\pi,\pi])  \setminus L_{\delta}$, we have
\begin{align*}
    \left|
\Log \left(
\Phi_{\B}\left(\frac{z}{2\pi \B}\right)
\right)
-
\left(
\frac{-i}{2\pi \B^2} \mathrm{L}(z)
\right)
    \right|
    \leq B_\delta \B^2.
\end{align*}
\end{lemma}

\begin{proposition}\label{unifoutsidebox2}
    For $z \in \CC \setminus L_\delta$, we have
    $$\Phi_\B\left (\frac{z}{2 \pi \B}\right ) = \exp\left (\frac{-i}{2 \pi \B^2}\mathrm{L}(z)\right ) \left ( 1 + O_{\B \to 0^+}(\B^2)\right ).$$
\end{proposition}
\begin{proof}
    First, for $z \in (\RR + i [-\pi,\pi])  \setminus L_{\delta}$, the result follows from Lemma \ref{unifoutsidebox}. Next, by using the functional equations of both the classical and quantum dilogarithm functions, we have    
\begin{align*}
    &\ \left|\Log \left(
\Phi_{\B}\left(\frac{z + 2\pi i}{2\pi \B}\right)
\right)
-
\left(
\frac{-i}{2\pi \B^2} \mathrm{L}(z+2\pi i) \right) \right| \\
\leq &\ \left|
\Log \left(
\Phi_{\B}\left(\frac{z}{2\pi \B}\right)
\right)
-
\left(
\frac{-i}{2\pi \B^2} \mathrm{L}(z)
\right)
\right|
+
\left| - \Log \left(1+e^{\frac{z+\pi i}{\B^2}} \right) \right| 
\end{align*}
if $\Re z <0$ and 
\begin{align*}
       &\ \left|\Log \left(
\Phi_{\B}\left(\frac{z + 2\pi i}{2\pi \B}\right)
\right)
-
\left(
\frac{-i}{2\pi \B^2} \mathrm{L}(z+2\pi i) \right) \right| \\
\leq & \ \left|
\Log \left(
\Phi_{\B}\left(\frac{z}{2\pi \B}\right)
\right)
-
\left(
\frac{-i}{2\pi \B^2} \mathrm{L}(z)
\right)
\right|
+
\left| - \Log \left(1+e^{\frac{z+\pi i}{\B^2}} \right) + \frac{z+\pi i}{\B^2} \right| \\
\leq &\ \left|
\Log \left(
\Phi_{\B}\left(\frac{z}{2\pi \B}\right)
\right)
-
\left(
\frac{-i}{2\pi \B^2} \mathrm{L}(z)
\right)
\right|
+
\left| - \Log \left(1+e^{-\frac{z+\pi i}{\B^2}} \right)  \right|
\end{align*}
if $\Re z >0$. Note that for $z \in \CC \setminus L_{\delta}$ with $\Re z <0$, we have
$$
\left| - \Log \left(1+e^{\frac{z+\pi i}{\B^2}} \right) \right|
< e^{- \frac{\delta}{\B^2}} < \frac{\B^2}{\delta}.
$$
Besides, for $z \in \CC \setminus L_{\delta}$ with $\Re z >0$, we have
$$
\left| - \Log \left(1+e^{-\frac{z+\pi i}{\B^2}} \right) \right|
< e^{- \frac{\delta}{\B^2}} < \frac{\B^2}{\delta}.
$$
This completes the proof.
\end{proof}

To study the asymptotics of partition functions in $\hbar$, we need the following result about the semi-classical limit of the quantum dilogarithm function in $\hbar$. The result has also been used in \cite[Lemma 3]{AK} by Andersen and Kashaev when they prove their conjecture for $4_1$ and $5_2$. 
\begin{proposition}\label{semihbar} (semi-classical limit in $\hbar$) For any $z \in \CC \setminus L_\delta$,
$$\Phi_\B\left (\frac{z}{2 \pi \sqrt{\hbar}}\right ) = \exp\left (-\frac{iz}{2\pi} \dfrac{d}{dz}\mathrm{L}(z) + \frac{i}{\pi} \mathrm{L}(z) - \frac{i}{2 \pi \hbar^2}  \mathrm{L}(z)\right ) \left ( 1 + O_{\hbar \to 0^+}(\hbar)\right ).$$
\end{proposition}
\begin{proof}
For any $z \in \CC \setminus L_\delta$, by Proposition \ref{unifoutsidebox2}, we have 
$$\Phi_\B\left (\frac{z}{2 \pi \B}\right ) = \exp\left (\frac{-i}{2 \pi \B^2} \mathrm{L}(z)\right ) \left ( 1 + O_{\B \to 0^+}(\B^2)\right ).$$
In particular, since $(\sqrt{\hbar})^{-1} = \B + \B^{-1}$, we have
\begin{align*}
   \Phi_\B\left (\frac{z}{2 \pi \sqrt{\hbar}}\right ) 
   = \Phi_\B\left (\frac{z(1+\B^2)}{2 \pi \B}\right ) 
   = \exp\left (\frac{-i}{2 \pi \B^2} \mathrm{L}(z(1+\B^2)) \right ) \left ( 1 + O_{\B \to 0^+}(\B^2)\right ).
\end{align*}
By Taylor's theorem, we have
$$
\mathrm{L}(z(1+\B^2))
= \mathrm{L}(z)
+ \B^2z \frac{d}{dz}\mathrm{L}(z) + O_{\B \to 0^+}(\B^2).
$$
Thus,
\begin{align*}
   \Phi_\B\left (\frac{z}{2 \pi \sqrt{\hbar}}\right ) 
   = &\ \exp\left ( - \frac{iz}{2\pi}\frac{d}{dz}\mathrm{L}(z)  - \frac{i}{2 \pi \B^2}  \mathrm{L}(z) \right ) \left ( 1 + O_{\B \to 0^+}(\B^2)\right ).
\end{align*}
Next, since
$$
\frac{1}{\hbar} = \B^2 + 2 + \frac{1}{\B^2}, 
$$
we can write
\begin{align*}
\Phi_\B\left (\frac{z}{2 \pi \sqrt{\hbar}}\right ) 
   = &\ \exp\left ( -\frac{i z}{2\pi}\frac{d}{dz}\mathrm{L}(z) + \frac{i}{\pi}\mathrm{L}(z) - \frac{i}{2 \pi \hbar^2}  \mathrm{L}(z)\right ) \left ( 1 + O_{\B \to 0^+}(\B^2)\right ).
\end{align*}
Finally, since
$$
\lim_{\B\to 0} \frac{\B^2}{\hbar}
= \lim_{\B\to 0}  (\B^4 + 2\B^2 + 1 ) 
= 1,
$$
we can replace $O_{\B \to 0^+}(\B^2)$ by $O_{\hbar \to 0^+}(\hbar)$. This completes the proof.
\end{proof}

\subsection{The Teichm\"uller TQFT of Andersen--Kashaev}\label{defAKTQFT}

In this section we follow \cite{AK, KaWB, Kan}. Let $\mathcal{S}(\R^d)$ denote the Schwartz space of smooth 
functions from $\R^d$ to $\C$ that are rapidly decaying (in the sense that any derivative decays faster than any negative power of the norm of the input). 
Its continuous dual $\mathcal{S}'(\R^d)$ is the space of tempered distributions.

Recall that the \emph{Dirac delta function} is the tempered distribution $\mathcal{S}(\R) \to \C$ denoted by $\delta(x)$ or $\delta$ and defined by
$
\delta(x) \cdot f:= \int_{x \in \R} \delta(x) f(x) dx =
f(0)
$ for all $f \in \mathcal{S}(\R)$ (where $x \in \R$ denotes the argument of $f\in \mathcal{S}(\R)$).
Furthermore, we have the equality of tempered distributions
\[
\delta(x)=\int_{w \in \R} e^{-2 i \pi  x w} \,dw,
\] 
in the sense that for all $f \in \mathcal{S}(\R)$, 
$$
\left (\int_{w \in \R} e^{-2 i \pi x w} \,dw\right ) (f) =
\int_{x \in \R} \int_{w \in \R} e^{-2 i \pi x w} f(x)  \,dw \, dx \ = f(0) = \delta(x) \cdot f.
$$
The second equality follows from applying the Fourier transform $\mathcal{F}$ twice and using the fact that $\mathcal{F}(\mathcal{F}(f))(x) = f(-x)$ for $f\in \mathcal{S}(\R), x \in \R$. Recall also that the definition of the Dirac delta function and the previous argument have multi-dimensional analogues (see for example \cite{Kan} for details).

Given a triangulation $X$, 
writing $X^k$ for its collection of $k$-cells ($k\in \{0,1,2,3\}$), we assign to the tetrahedra $T_1, \ldots,T_N \in X^3$ formal real variables $t_1, \ldots, t_N$. 
We name $\mathsf{t}\colon T_j \mapsto t_{j}$ the corresponding bijection, and $\mathbf{t} = (t_{1},\ldots,t_{N})$ the corresponding
formal vector in $\R^{X^{3}}$.
Recall the notation $x_i(T) \in X^2$ for the $i$-th face ($i\in \{0,1,2,3\}$) of the tetrahedron $T\in X^3$.

\begin{definition}\label{newdefK}
Let $X$ be a triangulation such that $H_2(M_X\smallsetminus X^0,\Z)=0$. The \textit{kinematical kernel of $X$} is a tempered distribution $\mathcal{K}_X \in \mathcal{S}'\left (\R^{X^{3}}\right )$ defined by the integral
\begin{align*}
&\ \ \mathcal{K}_X(\mathbf{t}) \\
=&\ \ \int_{\boldsymbol{x} \in \R^{X^{2}}} d\boldsymbol{x} \prod_{T \in X^3} e^{ 2 i \pi x_0(T) \mathsf{t}(T)}
\delta\left ( x_0(T)- x_1(T)+ x_2(T)\right )
\delta\left ( x_2(T)- x_3(T)+ \epsilon(T)\mathsf{t}(T)\right ).
\end{align*}
\end{definition}
  
More formally, one
should understand the integral of the previous formula as the following equality of tempered distributions, similarly as above ( ${\!\top}$ denoting the transpose):
$$
\mathcal{K}_X(\mathbf{t}) =
\int_{\boldsymbol{x} \in \R^{X^{2}}} d\boldsymbol{x} 
\int_{\boldsymbol{w} \in \R^{2 N}} d\boldsymbol{w} \
e^{ 2 i \pi \mathbf{t}^{\!\top} \mathscr{X}_0 \boldsymbol{x}}
e^{ -2 i \pi \boldsymbol{w}^{\!\top} \mathscr{A} \boldsymbol{x}}
e^{ -2 i \pi \boldsymbol{w}^{\!\top} \mathscr{B} \mathbf{t}} \
\in \mathcal{S}'\left (\R^{X^{3}}\right ),
$$
where
$\boldsymbol{w}=(w_1,\ldots,w_N,w'_1, \ldots,w'_N)$ is a vector of $2N$ new real variables, such that $w_j,w'_j$ are associated to 
$\delta\left ( x_0(T_j)- x_1(T_j)+ x_2(T_j)\right )$ and
$\delta\left ( x_2(T_j)- x_3(T_j)+ \epsilon(T_j)\mathsf{t}(T_j)\right )$, and where
 $\mathscr{X}_0,\mathscr{A},\mathscr{B}$ are matrices with integer coefficients depending on the values $x_k(T_j)$, i.e.\ on the combinatorics of the face gluings. More precisely, the rows (resp.\ columns) of $\mathscr{X}_0$ are indexed by the vector of tetrahedron variables $\mathbf{t}$ (resp.\ of face variables $\boldsymbol{x}$) and $\mathscr{X}_0$ has a coefficient $1$ at coordinate $(t_j,x_0(T_j))$ and zero everywhere else; $\mathscr{B}$ is indexed by  $\boldsymbol{w}$ (rows) and  $\mathbf{t}$ (columns) and has a $\varepsilon(T_j)$ at the coordinate $(w'_j,t_j)$; finally, $\mathscr{A}$ is such that $\mathscr{A} \boldsymbol{x} + \mathscr{B}  \mathbf{t}$ is a column vector indexed by  $\boldsymbol{w}$ containing the values 
 $\left (x_0(T_j)- x_1(T_j)+ x_2(T_j)\right )_{1\leq j \leq N}$ followed by $\left (x_2(T_j)- x_3(T_j)+ \epsilon(T_j)t_j \right)_{1\leq j \leq N}$. As an example, see Figure \ref{fig:41:face:matrices} for the matrices $\mathscr{X}_0, \mathscr{A}, \mathscr{B}$ for Thurston's ideal triangulation of $\SS^3 \setminus 4_1$.

\begin{remark}\label{diffKdef}
The definition of the kinematic kernel in Definition \ref{newdefK} differs from those in \cite{AK, BAGPN, BAW} by a sign. Precisely, they are related by the transformation $(t_1,\dots, t_N) \mapsto (\varepsilon(T_1)t_1,\dots, \varepsilon(T_N)t_N)$.
\end{remark}

\begin{definition}
Let $X$ be a triangulation. Its \textit{dynamical content} associated to $\hbar>0$ is a function $\mathcal{D}_{\hbar,X}\colon \mathcal{A}_X \to  \mathcal{S}\left (\R^{X^{3}}\right )$ defined on each set of angles $\alpha \in \mathcal{A}_X$ by
$$\mathcal{D}_{\hbar,X}(\mathbf{t},\alpha)= \prod_{T\in X^{3}} 
\dfrac{\exp \left( \hbar^{-1/2} \alpha_3(T) \varepsilon(T)\mathsf{t}(T) \right )}
{\Phi_\B\left ( \varepsilon(T) \left(\mathsf{t}(T) - \dfrac{i}{2 \pi \sqrt{\hbar}} (\pi-\alpha_1(T))\right) \right)^{\varepsilon(T)}}.
$$
\end{definition}

Note that $\mathcal{D}_{\hbar,X}(\cdot,\alpha)$ is in $\mathcal{S}\left (\R^{X^{3}}\right )$ thanks to the properties of $\Phi_\B$ and the positivity of the dihedral angles in $\alpha$ (see \cite{AK} for details). More precisely, each term in the dynamical content has exponential decrease as described in the following lemma, which immediately follows from Proposition \ref{prop:quant:dilog} (4).

\begin{lemma}\label{lem:dec:exp}\cite[Lemma 2.11]{BAGPN}\label{BAGPNlemma2.11}
	Let $\B \in \R_{>0}$ and $a,b,c \in (0,\pi)$ such that $a+b+c=\pi$. Then
	$$
	\left |
	\dfrac{e^{\frac{1}{ \sqrt{\hbar}} c x}}{\Phi_\B\left (x-\frac{i}{ 2 \pi \sqrt{\hbar}}(b+c)\right )}
	\right | \underset{\R \ni x \to \pm \infty}{\sim} \left |
	e^{\frac{1}{ \sqrt{\hbar}} c x} \Phi_\B\left (x+\frac{i}{ 2 \pi \sqrt{\hbar}}(b+c)\right )
	\right | \ \ \left \{
	\begin{matrix}
	\underset{\R \ni x \to -\infty}{\sim} e^{\frac{1}{ \sqrt{\hbar}} c x}. \\
	\ \\
	\underset{\R \ni x \to +\infty}{\sim} e^{-\frac{1}{ \sqrt{\hbar}} b x}.
	\end{matrix} \right .
	$$
\end{lemma}


Lemma \ref{lem:dec:exp} illustrates why we need the three angles $a,b,c$ to be in $(0,\pi)$: $b$ and $c$ must be positive in order to have exponential decrease in both directions, and $a$ must be  {positive} as well so that $b+c < \pi$ and $\Phi_\B\left (x \pm \frac{i}{ 2 \pi \sqrt{\hbar}}(b+c)\right )$ is always defined.

Alternatively, using the inversion formula of the quantum dilogarithm function (Lemma \ref{prop:quant:dilog} (1)), the dynamical content can be written in the following form.
\begin{lemma}\label{altDC}
Let $X$ be a triangulation. Up to a multiplicative constant with norm 1, we have
\begin{align*}
&\ \ \mathcal{D}_{\hbar,X}(\mathbf{t},\alpha) \\
=&\ \ \displaystyle\prod_{T\in X^{3}}\frac{ 
\exp \left( \hbar^{-1/2} \alpha_3(T) \epsilon(T) \mathsf{t}(T) - i\pi \left(\frac{\epsilon(T)-1}{2}\right)\left(\mathsf{t}(T) - \dfrac{i}{2 \pi \sqrt{\hbar}} (\pi-\alpha_1(T))\right)^2
\right )}{
\Phi_\B\left ( \mathsf{t}(T) - \dfrac{i}{2 \pi \sqrt{\hbar}} (\pi-\alpha_1(T))\right )}. 
\end{align*}

\end{lemma}
\begin{proof}
The result follows immediately by using Lemma \ref{prop:quant:dilog} (1) and the fact that $|e^{i\frac{\pi}{12}(\B^2 + \B^{-2})}|=1$ for all $\B\in \RR_{>0}$. 
\end{proof}

Now, for $X$ a triangulation such that $H_2(M_X\setminus X_0,\Z)=0$, $\hbar>0$ and $\alpha \in \mathcal{A}_X$ an angle structure,  the associated \textit{partition function of the Teichm\"uller TQFT} is the complex number:
$$\mathcal{Z}_{\hbar}(X,\alpha)= \int_{\mathbf{t} \in \R^{X^3}}  \mathcal{K}_X(\mathbf{t}) \mathcal{D}_{\hbar,X}(\mathbf{t},\alpha) d\mathbf{t} \  \in \C. $$

Andersen and Kashaev proved in \cite{AK} that the  {modulus} $\left |\mathcal{Z}_{\hbar}(X,\alpha) \right | \in \R_{>0}$ is invariant under certain Pachner moves with positive angles. In particular, we can use the formula in Lemma \ref{altDC} to compute the modulus of the partition functions. Together with Remark \ref{diffKdef}, the real number $|\mathcal{Z}_{\hbar}(X,\alpha)  |$ is this paper coincides with those defined in \cite{BAGPN, BAW} using slightly different definitions.

\subsection{Neumann-Zagier datum and gluing equations}\label{NZD}
Let $K \subset \SS^3$ be a hyperbolic knot. Choose a simple closed curve $\gamma \in \pi_1(\partial(\SS^3\setminus\nu(K)))$. Let $X = \{T_1,\dots,T_N\}$ be an ideal triangulation of $\SS^3 \setminus K$ and let $X^1 = \{E_1,\dots, E_N\}$ be the set of edges. For each $T_i$, we choose a quad type (i.e. a pair of opposite edges) and assign a shape parameter $z_i\in \CC\setminus \{0,1\}$ to the edges. For $z_i \in \CC\setminus \{0,1\}$, we define $z_i ' = \frac{1}{1-z_i}$, $z_i'' = 1-\frac{1}{z_i}$. Recall that for each ideal tetrahedron, opposite edges share the same shape parameters. By \cite{NZ}, there exists $N-1$ linearly independent edge equations, in the sense that if these $N-$ edge equations are satisfied, the remaining edge equation will automatically be satisfied. Without loss of generality we assume that $\{E_1,\dots,E_{N-1}\}$ is a set of linearly independent edges. For each edge $E_i$, we let $E_{i,j}$ be the numbers of edges with shape parameter $z_j$ that is incident to $E_i$. We define $E_{i,j}'$ and $E_{i,j}''$ be respectively the corresponding counting with respect to $z_j'$ and $z_j''$. The gluing variety $\mathcal{V}_{X}$ is the affine variety in $(z_1,z_1',z_1'', \dots, z_N, z_N', z_N'') \in \CC^{3N}$ defined by the zero sets of the polynomials
\begin{align*}
p_i = z_i(1-z_i'')- 1,\quad p_i' = z_i'(1-z_i) -1,\quad p_i'' = z_i''(1-z_i') -1 
\end{align*}
for $i=1,\dots, N$ and the polynomials
\begin{align*}
\prod_{j=1}^N z_j^{E_{ij}} (z_j')^{E_{ij}'} (z_j'')^{E_{ij}''}  -1
\end{align*}
for $i=1,\dots, N-1$. By using the equations $p_i=p_i'=p_i'' =0$ for $i=1,\dots,N$, for simplicity we will use $(z_1,\dots, z_N) \in (\CC\setminus \{0,1\})^N$ to represent a point in $\mathcal{V}_{X}$. Besides, we let $C_{i,j}$ be the numbers of edges with shape $z_j$ on the left hand side of $\gamma$ minus the numbers of edges with shape $z_j$ on the right hand side of $\gamma$. We define $C_{i,j}'$ and $C_{i,j}''$ be respectively the corresponding counting with respect to $z_j'$ and $z_j''$. Given $(z_1,\dots, z_n) \in \mathcal{V}_{X}$, the logarithmic holonomy of $\gamma$ is given by
$$
\mathrm{H}^\C_{X,\gamma}(\mathbf{z}) = \sum_{j=1}^N \left( C_{j} \Log(z_i) + C_{j}' \Log(z_i') + C_{j}'' \Log (z_i'') \right).
$$
There is a well-defined map
$$ \mathcal{P}_{X} : \mathcal{V}_{X} \to \chi(\SS^3 \setminus K)$$
from the gluing variety of $X$ to the $\mathrm{PSL}_2(\CC)$-character variety of $\SS^3 \setminus K$ that sends $\mathbf{z}=(z_1,\dots, z_n)  \in \mathcal{V}_{X}$ to the character $[\rho_{\mathbf{z}}]$ of the pseudo-developing map $\rho_{\mathbf{z}}$ with 
$$
\rho_{\mathbf{z}}([\gamma]) = \pm 
\begin{pmatrix}
\exp\left(\frac{\mathrm{H}^\C_{X,\gamma}(\mathbf{z})}{2}\right) & * \\
0 & \exp\left(-\frac{\mathrm{H}^\C_{X,\gamma}(\mathbf{z})}{2}\right)  
\end{pmatrix}$$
up to conjugation.

We define three $N \times N$ matrices $\mathbf{G},\mathbf{G'},\mathbf{G''} \in M_{N\times N}(\ZZ)$ by
\begin{align*}
\mathbf{G}=
\begin{pmatrix}
E_{1,1} & E_{1,2} & \dots & E_{1,N} \\
\vdots & \vdots &  & \vdots \\
E_{N-1,1} & E_{N-1,2} & \dots & E_{N-1,N} \\
C_{1} & C_{2} & \dots & C_{N} \\
\end{pmatrix},\qquad
\mathbf{G'}=
\begin{pmatrix}
E_{1,1}' & E_{1,2}' & \dots & E_{1,N}' \\
\vdots & \vdots &  & \vdots \\
E_{N-1,1}' & E_{N-1,2}' & \dots & E_{N-1,N}' \\
C_{1}' & C_{2}' & \dots & C_{N}' \\  
\end{pmatrix}
\end{align*}
and
\begin{align*}
\mathbf{G''}=
\begin{pmatrix}
E_{1,1}'' & E_{1,2}'' & \dots & E_{1,N}'' \\
\vdots & \vdots  &  & \vdots \\
E_{N-1,1}'' & E_{N-1,2}'' & \dots & E_{N-1,N}'' \\
C_{1}'' & C_{2}'' & \dots & C_{N}'' \\
\end{pmatrix}.
\end{align*}
Given $\xi\in \CC$, the edges equations and holonomy equation can be written in the form
$$
\mathbf{G} \BLog \mathbf{z} + \mathbf{G'} \BLog \mathbf{z'} + \mathbf{G''} \BLog \mathbf{z''}
= (
    2\pi i,
    \dots,
    2\pi i,
    \xi
)^{\!\top}.
$$
By using the equation $\Log z + \Log z' + \Log z'' = \pi i$, the equation can be rewritten as
$$\mathbf{A} \BLog \mathbf{z} + \mathbf{B} \BLog \mathbf{z''} 
= i \boldsymbol{\nu} + \tilde{\boldsymbol u},$$
where $\mathbf{A} = \mathbf{G-G'}, \mathbf{B} = \mathbf{G''-G'}$, $\boldsymbol{\nu} \in \pi \ZZ^N$ and $\tilde{\boldsymbol u} = (0,\dots,0,\xi)^{\!\top}$.

\subsection{Combinatorial flattening}
\begin{definition}\label{CF}
For a simple closed curve $\gamma \in \pi_1(\partial (\SS^3 \setminus \nu(K)) )$, a combinatorial flattening with respect to $\gamma$ consists of three vectors 
\begin{align*}
\mathbf{f} = (f_1,\dots,f_N),\quad
\mathbf{f'} = (f_1',\dots,f_N'),\quad
\mathbf{f}'' = (f_1'',\dots,f_N'') \in \mathbb{Z}^N
\end{align*}
such that 
\begin{itemize}
\item for $i=1,\dots, N$, we have $f_i + f_i' + f_i'' = 1$ and
\item the $i$-th entry of the vector 
$$\mathbf{G}\cdot \mathbf{f}^{\!\top} + \mathbf{G'} \cdot {\mathbf{f}'}^{\!\top} + \mathbf{G''} \cdot {\mathbf{f}''}^{\!\top} $$ is equal to $2$ for $i=1,\dots, N-1$ and is equal to $0$ for $i=N$.
\end{itemize}
\end{definition}

We have the following stronger version of combinatorial flattening, which requires the last condition in Definition \ref{CF} to be satisfied for all simple closed curves.
\begin{definition}\label{SCF}
A strong combinatorial flattening consists of three vectors 
\begin{align*}
\mathbf{f} = (f_1,\dots,f_N),\quad
\mathbf{f'} = (f_1',\dots,f_N'),\quad
\mathbf{f}'' = (f_1'',\dots,f_N'') \in \mathbb{Z}^N
\end{align*}such that 
\begin{itemize}
\item for $i=1,\dots, N$, we have $f_i + f_i' + f_i'' = 1$ and
\item for {\textbf {\textit any}} simple closed curve $\gamma $, the $i$-th entry of the vector 
$$\mathbf{G}\cdot \mathbf{f}^{\!\top} + \mathbf{G'} \cdot {\mathbf{f}'}^{\!\top} + \mathbf{G''} \cdot {\mathbf{f}''}^{\!\top} $$  is equal to $2$ for $i=1,\dots, N-1$ and is equal to $0$ for $i=N$.
\end{itemize}
\end{definition}
\begin{remark}
By \cite[Lemma 6.1]{N}, a strong combinatorial flattening exists for any ideal triangulation.
\end{remark}

\subsection{1-loop invariant and torsion as rational functions on the gluing variety}\label{1loopsection}
\begin{definition}\label{defnrhoregular}
Let $\rho : \pi_1(M) \to \mathrm{PSL}(2;\CC)$ be a representation. An ideal triangulation $X$ of $M$ is $\rho$-regular if there exists $\mathbf{z} \in \mathcal{V}(X)$ such that $\mathcal{P}_{X}(\mathbf z) = [\rho]$.
\end{definition}
\begin{definition}[\cite{DG}] 
Let $X$ be an ideal triangulation of $M$ and $\mathbf{z}$ a shape structure on $X$.
Then the 1-loop invariant of $(M,\boldsymbol\gamma,X, \mathbf{z})$ is defined by
$$\tau(M, \boldsymbol\gamma, X, \mathbf{z}) = 
\pm \frac{1}{2} \mathrm{det}\Big( \mathbf{A} \Delta_{\mathbf{z}''} + \mathbf{B} \Delta_{\mathbf{z}}^{-1}\Big) \prod_{i=1}^N \Big(z_i^{f_i''} z_i''^{-f_i}\Big)
,$$
where $\mathbf{A}=\mathbf{G} - \mathbf{G'}, \mathbf{B} = \mathbf{G''} - \mathbf{G}$ and $(\mathbf{f},\mathbf{f}',\mathbf{f}'')$ is a strong combinatorial flattening, 
$$\Delta_{\mathbf{z}} 
= \begin{pmatrix}
z_1 & 0 & 0 & \dots & 0 \\
0 & z_2 & 0 & \dots & 0 \\
\vdots & \vdots & \vdots & \vdots & \vdots \\
0 & 0 & 0 & 0 & z_N
\end{pmatrix}
\quad \text{and} \quad
\Delta_{\mathbf{z}''} 
= \begin{pmatrix}
z_1'' & 0 & 0 & \dots & 0 \\
0 & z_2'' & 0 & \dots & 0 \\
\vdots & \vdots & \vdots & \vdots & \vdots \\
0 & 0 & 0 & 0 & z_N''
\end{pmatrix}.$$
\end{definition}

The 1-loop conjecture proposed by Dimofte and Garoufalidis suggests that the 1-loop invariant coincides with the adjoint twisted Reidemeister torsion $\mathbb T_{(\SS^3\setminus K,\boldsymbol\gamma)}([\rho_{\mathbf{z}}])$ defined in \cite{P}. We used the following formulation of the conjecture from \cite{PW}.
\begin{conjecture}\label{1loopconjstatement}
Let $K\subset \SS^3 $ be a hyperbolic knot. Let $\gamma$ be a simple closed curves on $\partial (\nu(K))$. Let $\rho_0$ be the unique discrete faithful representation corresponding to the complete hyperbolic structure of $\SS^3\setminus K$. Let $X$ be a $\rho_0$-regular ideal triangulation of $M$ with $\mathcal{P}_{X}(\mathbf z_0) = [\rho_0]$. Let $\mathcal{V}_0(X)$ be the irreducible component of $\mathcal{V}(X)$ containing $\mathbf{z_0}$. For any $\mathbf{z} \in \mathcal{V}_0(X)$ with ${\mathcal{P}}_{X}(\mathbf{z}) = [\rho_{\mathbf{z}}]$, we have
$$
\tau(\SS^3\setminus K, \boldsymbol\gamma, X, \mathbf z) 
= \pm \mathbb T_{(\SS^3\setminus K,\boldsymbol\gamma)}([\rho_{\mathbf{z}}]).
$$
\end{conjecture}
In Proposition \ref{Hesstotor}, we will show that the 1-loop invariant naturally shows up in the asymptotic expansion formula of the partition functions of the Teichm\"{u}ller TQFT invariants.

\subsection{Neumann-Zagier potential function}\label{NZpotentintro}
Let $K \subset \SS^3$ be a hyperbolic knot. Recall from \cite{NZ} that for any simple closed curve $\gamma \in \pi_1(\partial (\nu(K)))$, locally near the complete hyperbolic structure, the deformation space of hyperbolic structure of $M$ can be parametrized in a generically 2:1 way with one complex variable $w_\gamma$ 
such that, for any complex number $w_\gamma \in \C$ close enough to $0$, 
the point $w_\gamma$ corresponds to the hyperbolic structure where the complex logarithmic holonomy of $\gamma$ is equal to $w_\gamma$ (if $X$ is an ideal triangulation of $M$, this means the structure such that
$\mathrm{H}^\C_{X,\gamma}(\mathbf{z})=w_\gamma$).


Furthermore, given a pair of simple closed curves $(\gamma_1, \gamma_2)$ such that the elements $[\gamma_1], [\gamma_2]$ generate $\pi_1(\partial (\SS^3 \setminus \nu(K)))$, the transition map 
$\Psi_{\gamma_1,\gamma_2}$
from $w_{\gamma_1}$ to 
$w_{\gamma_2}
=\Psi_{\gamma_1,\gamma_2}
(w_{\gamma_1})$
is a locally biholomorphic map around 0 that sends 0 to 0. 
For example, when $\mathbf{z} \in \mathcal{V}_{X}$ is  sufficiently close to the complete one, we have $\Psi_{\gamma_1,\gamma_2}(H^\C_{X,\gamma_1}(\mathbf{z}))=\mathrm{H}^\C_{X,\gamma_2}(\mathbf{z})$.

In particular, one can consider the holomorphic function $\phi_{\gamma_1,\gamma_2}$ defined locally on a simply connected neighborhood around 0 by
$$
\phi_{\gamma_1,\gamma_2}(w_{\gamma_1})
=
i\Big(\Vol(M) + i\CS(M)\Big) + \frac{1}{2}
\int_0^{w_{\gamma_1}}
\Psi_{\gamma_1,\gamma_2}
(t)
dt
,
$$
where $w_{\gamma_2}=\Psi_{\gamma_1,\gamma_2}(t)$ is regarded as a function in a complex variable $t$, the integral is along any contour from $0$ to $w_{\gamma_1}$, and $\Vol(\SS^3\setminus K)$ and $\CS(\SS^3\setminus K)$ are the hyperbolic volume and the Chern-Simons invariant of $\SS^3\setminus K$ respectively. 
Note that the holomorphic function $\phi_{\gamma_1,\gamma_2}$ satisfies the properties that
\begin{align}\label{NZprop}
\phi_{\gamma_1,\gamma_2}(0) =  i\Big(\Vol(\SS^3\setminus K) + i\CS(\SS^3\setminus K)\Big) \quad \text{and} \quad \frac{d \phi_{\gamma_1,\gamma_2}\big(w_{\gamma_1}\big)}{d\big(w_{\gamma_1}\big)} = \frac{
\Psi_{\gamma_1,\gamma_2}
(w_{\gamma_1})}{2},
\end{align}
which uniquely characterizes the function.

\subsection{Complex Morse Lemma and Saddle point approximation}
The following version of complex Morse Lemma can be found in \cite[Lemma A.3]{WY2}.
\begin{lemma}\label{OPFCML}(Complex Morse Lemma) Let $D_{\mathbf z}$ be a region in $\mathbb C^n,$ let $D_{\mathbf a}$ be a region in $\mathbb R^k,$ and let $f: D_{\mathbf z} \times D_{\mathbf a} \to \mathbb C$ be a complex valued function that is holomorphic in $\mathbf z\in D_{\mathbf z}$ and smooth in $\mathbf a\in D_{\mathbf a}.$ For $\mathbf a\in D_{\mathbf a},$ let $f^{\mathbf a}: D_{\mathbf z}\to\mathbb C$ be the function defined by $f^{\mathbf a}(\mathbf z)=f(\mathbf z,\mathbf a).$ Suppose for each $\mathbf a\in D_{\mathbf a},$ $f^{\mathbf a}$ has a non-degenerate critical point $c_{\mathbf a}$ which smoothly depends on $\mathbf a.$ Then for each $\mathbf a_0\in D_{\mathbf a},$ there exists an open set $V \subset \mathbb C^n$ containing $\mathbf 0,$ an open set $A \subset D_{\mathbf a}$ containing $\mathbf a_0,$ and a smooth function $\psi: V \times A \to   D_{\mathbf z}$ such that, if we denote $\psi^{\mathbf a}(\mathbf Z) = \psi(\mathbf Z, \mathbf a),$ then for each $\mathbf a\in D_{\mathbf a},$ $\mathbf z = \psi^{\mathbf a}(\mathbf Z)$ is a holomorphic change of variable on $V$ such that 
$$\psi^{\mathbf a}(\mathbf 0) = \mathbf c_{\mathbf a},$$
$$ f^{\mathbf a}(\psi^{\mathbf a}(\mathbf Z)) = f^{\mathbf a}(\mathbf c_{\mathbf a}) - Z_1^2 - \dots - Z_n^2.$$
\end{lemma}

The following version of saddle point approximation is a special case of \cite[Proposition 5.1]{WY2}. 
\begin{proposition}(Saddle point approximation)\label{saddle}
Let $D$ be a region in $\mathbb C^{N}$. Let $f(\mathbf z)$ and $g(\mathbf z)$ be complex valued functions on $D$  which are holomorphic in $\mathbf z$. For each positive $r,$ let $f_r(\mathbf z)$ be a complex valued function on $D$ holomorphic in $\mathbf z$ of the form
$$ f_r(\mathbf z) = f(\mathbf z) + \frac{\upsilon_r(\mathbf z)}{r^2}.$$
Let $S$ be an embedded $n$-dimensional disk and $\mathbf c$ be a critical point of $f$ on $S$. If for each $r>0$
\begin{enumerate}[(1)]
\item $\mathrm{Re}f(\mathbf c) > \mathrm{Re}f(\mathbf z)$ for all $\mathbf z \in S\setminus \{\mathbf c\},$
\item the domain $\{\mathbf z\in D\ |\ \mathrm{Re} f(\mathbf z) < \mathrm{Re} f(\mathbf c)\}$ deformation retracts to $S\setminus\{\mathbf c\},$
\item $g(\mathbf c)\neq 0$,
\item $|\upsilon_r(\mathbf z)|$ is bounded from above by a constant independent of $r$ on $D,$ and
\item  the Hessian matrix $\mathrm{Hess}(f)$ of $f$ at $\mathbf c$ is non-singular,
\end{enumerate}
then as $r\to \infty$,
\begin{equation*}
\begin{split}
 \int_{S} g(\mathbf z) e^{rf_r(\mathbf z)} d\mathbf z= \Big(\frac{2\pi}{r}\Big)^{\frac{{N}}{2}}\frac{g(\mathbf c)}{\sqrt{(-1)^{{N}}\det\mathrm{Hess}(f)(\mathbf c)}} e^{rf(\mathbf c)} \Big( 1 + O \Big( \frac{1}{r} \Big) \Big).
 \end{split}
 \end{equation*}
In particular, we have
$$
\lim_{r \to \infty} \frac{1}{r}\log\left|\int_{S} g(\mathbf z) e^{rf_r(\mathbf z)} d\mathbf z \right|
= \Re f(\mathbf{c}).
$$
\end{proposition}

\section{Computation of the partition functions}

Recall that kinematical kernel is given by
$$
\mathcal{K}_X(\mathbf{t}) =
\int_{\boldsymbol{x} \in \R^{X^{2}}} d\boldsymbol{x} 
\int_{\boldsymbol{w} \in \R^{2 N}} d\boldsymbol{w} \
e^{ 2 i \pi \mathbf{t}^{\!\top} \mathcal{X}_0 \boldsymbol{x}}
e^{ -2 i \pi \boldsymbol{w}^{\!\top} \mathcal{A} \boldsymbol{x}}
e^{ -2 i \pi \boldsymbol{w}^{\!\top} \mathcal{B} \mathbf{t}} 
.$$

Let rank$(\mathcal{A}) = r \leq 2N$ and $n = 2N - r$. Let $\mathcal{E}_1, \mathcal{E}_2$ be product of elementary matrices such that 
$$
\mathcal{E}_1 \mathcal{A} \mathcal{E}_2 = 
\begin{pmatrix}
\rm{Id}_r & O_1 \\
O_2 & O_3
\end{pmatrix},
$$
where $O_1 \in M_{r\times n}(\ZZ),O_2\in M_{n\times r}(\ZZ),O_3 \in M_{n\times n}(\ZZ)$ are zero matrices of the corresponding sizes respectively. We write 
$$\mathscr{X}_0\mathcal{E}_2 = ( (\mathscr{X}_0\mathcal{E}_2)_r (\mathscr{X}_0\mathcal{E}_2)_n) \quad \text{and} \quad \mathcal{E}_1\mathcal{B} = \begin{pmatrix} (\mathcal{E}_1\mathcal{B})_r \\ (\mathcal{E}_1\mathcal{B})_n \end{pmatrix}, $$
where $(\mathscr{X}_0\mathcal{E}_2)_r \in M_{N \times r}(\QQ), (\mathscr{X}_0\mathcal{E}_2)_r \in M_{N \times n}(\QQ)$, $(\mathcal{E}_1\mathcal{B})_r \in M_{r \times N}(\QQ)$ and $(\mathcal{E}_1\mathcal{B})_n \in M_{n\times N}(\QQ)$ respectively. 

\begin{lemma}\label{KKdelta} The kinematical kernel in Definition \ref{newdefK} is given by
    \begin{align*}
\mathcal{K}_X(\mathbf{t}) 
&= \det(\mathcal{E}_1\mathcal{E}_2) \delta\Big( ((\mathscr{X}_0\mathcal{E}_2)_n)^{\!\top}\mathbf{t}  \Big) 
\delta\Big((\mathcal{E}_1\mathcal{B})_n \mathbf{t}\Big) 
e^{- i\pi \mathbf{t}^{\!\top} Q \mathbf{t}},
\end{align*}
where
$$Q= \big( (\mathscr{X}_0\mathcal{E}_2)_r (\mathcal{E}_1\mathcal{B})_r \big) + \big((\mathscr{X}_0\mathcal{E}_2)_r(\mathcal{E}_1\mathcal{B})_r\big)^{\!\top} = \mathscr{G} - \frac{\mathcal{E}+\mathrm{Id}_N}{2}  \in M_{N\times N}(\QQ)$$
with the matrix $\mathscr{G}$ defined in (\ref{defmathscrG}).
\end{lemma}
\begin{proof}
We apply the change of variables that replace $\boldsymbol{x}$ and $\boldsymbol{w}$ by $\mathcal{E}_2\boldsymbol{x}$ and $\mathcal{E}_1^T\boldsymbol{w}$ respectively. Then the kinematical kernel becomes 
\begin{align*}
\mathcal{K}_X(\mathbf{t}) 
&= \det(\mathcal{E}_1\mathcal{E}_2)\int_{\boldsymbol{x} \in \R^{X^{2}}} d\boldsymbol{x} 
\int_{\boldsymbol{w} \in \R^{2 N}} d\boldsymbol{w} 
e^{ 2 i \pi \mathbf{t}^{\!\top} \mathscr{X}_0\mathcal{E}_2 \boldsymbol{x}}
e^{ -2 i \pi \boldsymbol{w}^{\!\top} \mathcal{E}_1\mathcal{A}\mathcal{E}_2 \boldsymbol{x}}
e^{ -2 i \pi \boldsymbol{w}^{\!\top} \mathcal{E}_1\mathcal{B} \mathbf{t}}.
\end{align*}
Let $n = 2N-r$ be the nullity of $A$. Write $\boldsymbol{x}=(\boldsymbol{x_r}, \boldsymbol{x_n})$ and $\boldsymbol{w}=(\boldsymbol{w_r}, \boldsymbol{w_n})$, where $\boldsymbol{x_r}, \boldsymbol{w_r} \in \RR^r$ and $\boldsymbol{x_n}, \boldsymbol{w_n} \in \RR^n$ respectively. 
Then 
\begin{align*}
\mathcal{K}_X(\mathbf{t}) 
&= \det(\mathcal{E}_1\mathcal{E}_2) \int_{\boldsymbol{x_n}} d\boldsymbol{x_n} 
\int_{\boldsymbol{w_n}} d\mathbf{w_n}
e^{ 2 i \pi \mathbf{t}^{\!\top} (\mathscr{X}_0\mathcal{E}_2)_n \boldsymbol{x_n}}
e^{ -2 i \pi \mathbf{w_n}^{\!\top} (\mathcal{E}_1\mathcal{B})_n \mathbf{t}} \\
&\quad \times \int_{\boldsymbol{x_r}} d\boldsymbol{x_r} 
\int_{\boldsymbol{w_r}} d\boldsymbol{w_r}
e^{ 2 i \pi \mathbf{t}^{\!\top} (\mathscr{X}_0\mathcal{E}_2)_r \boldsymbol{x_r}}
e^{ -2 i \pi \boldsymbol{w_r}^{\!\top}  \boldsymbol{x_r}}
e^{ -2 i \pi \boldsymbol{w_r}^{\!\top} (\mathcal{E}_1\mathcal{B})_r \mathbf{t}}\\
&= \det(\mathcal{E}_1\mathcal{E}_2) \delta\Big( ((\mathscr{X}_0\mathcal{E}_2)_n)^{\!\top}\mathbf{t}  \Big) 
\delta\Big((\mathcal{E}_1\mathcal{B})_n \mathbf{t}\Big) 
e^{-2i\pi \mathbf{t}^{\!\top} (\mathscr{X}_0\mathcal{E}_2)_r (\mathcal{E}_1\mathcal{B})_r \mathbf{t}},
\end{align*}
where the last equality follows by applying Fourier transformation twice. By symmetrizing the matrix $(\mathscr{X}_0\mathcal{E}_2)_r (\mathcal{E}_1\mathcal{B})_r $ and using (\ref{defmathscrG}), we get the desired result.
\end{proof}
\begin{remark}
When $n=0$, we can take $\mathcal{E}_1 = \mathcal{A}^{-1}$ and $\mathcal{E}_2 = \rm{Id}_N$. Then Lemma \ref{KKdelta} recovers \cite[Lemma 2.9]{BAGPN}. 
\end{remark}

Let $\alpha =(a_1,b_1,c_1,\dots, a_N, b_N, c_N) \in \mathcal{A}_X$. Define
$C(\alpha)=
(\varepsilon(T_1)c_1,\dots, \varepsilon(T_N)c_N)^{\!\top}$ and $\mathscr{W}(\alpha) = Q (\boldsymbol{\pi - a}) + C(\alpha),$ where the first term in $\mathscr{W}(\alpha)$ is the product of the matrix $Q \in M_{N\times N}(\QQ)$ and the vector $\boldsymbol{\pi - a} = (\pi-a_1,\dots, \pi-a_N)^{\!\top} \in M_{N\times 1}(\RR)$.

\begin{proposition}\label{Tpartiexpress1}
For any ordered ideal triangulation $X$ with $\alpha \in \mathcal{A}_X$, the modulus of the partition function of $(X,\alpha)$ is given by
\begin{align*}
|\mathcal{Z}_{\hbar}(X,\alpha) |
=  \left| \det(\mathcal{E}_1\mathcal{E}_2)
\left(\frac{1}{2\pi \sqrt{\hbar}}\right)^{N-2n}
\int_{\mathscr{S_{\alpha} }} 
F(\mathbf{y}) d\mathbf{y} \right|,
\end{align*}
where
\begin{align*}
 F_\hbar(\mathbf{y})
= 
\frac{\exp\left(\frac{1}{2\pi \hbar} \left( - \frac{i}{2} \mathbf{y}^{\!\top} Q \mathbf{y} + \mathbf{y}^{\!\top} \mathscr{W}(\alpha) - \frac{i}{2} \sum_{k=1}^N \left(\frac{\varepsilon(T_k)-1}{2}\right) y_k^2 \right) \right)}{\prod_{k=1}^N \Phi_\B\left(\frac{y_k}{2\pi \sqrt{\hbar}}\right) }
\end{align*}
is a holomorphic function in $\mathbf{y}$ defined on
 
$$
\mathscr{Y}_{\alpha} = \prod_{k=1}^N \big(\RR - i(\pi-a_k)\big) 
$$
and $\mathscr{S_{\alpha} }\subset{\mathscr{Y}_{\alpha}}$ is the linear subspace defined by 
$$
 (\mathscr{X}_0\mathcal{E}_2)_n^{\!\top} 
 \left(\mathbf{y} + i  (\boldsymbol{\pi - a})    \right)
 =
(\mathcal{E}_1\mathcal{B})_n 
 \left(\mathbf{y} + i (\boldsymbol{\pi - a})    \right)
 = 0. $$
\end{proposition}

\begin{proof}
We apply the change of variables $\tilde{\mathbf{y}} = \mathbf{t} - \frac{i}{2\pi\sqrt{\hbar}}(\boldsymbol{\pi - a} )$. Denote
$$
\tilde{\mathscr{Y}}_{\hbar, \alpha} = \prod_{k=1}^N \left(\RR - \frac{i}{2\pi\sqrt{\hbar}}(\pi-a_k)\right) .
$$
By Lemma \ref{altDC} and \ref{KKdelta}, we have
\begin{align*}
&\ \ \left|\mathcal{Z}_{\hbar}(X,\alpha)\right|\\
=&\ \ \left|\int_{\mathbf{t} \in \R^{X^3}}  \mathcal{K}_X(\mathbf{t}) \mathcal{D}_{\hbar,X}(\mathbf{t},\alpha) d\mathbf{t} \right| \\
=&\ \ \left|\int_{\mathbf{\tilde y} \in \tilde{\mathscr{Y}}_{\hbar, \alpha}}  \mathcal{K}_X\left( \tilde{\mathbf{y}} + \frac{i}{2\pi\sqrt{\hbar}}(\boldsymbol{\pi - a})\right) \mathcal{D}_{\hbar,X}\left( \tilde{\mathbf{y}} + \frac{i}{2\pi\sqrt{\hbar}}(\boldsymbol{\pi - a}) \right) d\mathbf{\tilde y} \right| \\
=&\ \  |\det(\mathcal{E}_1\mathcal{E}_2) | \times \\
&\ \ \Bigg|\int_{\mathbf{y} \in \tilde{\mathscr{Y}}_{\hbar, \alpha} } \delta\left( (\mathscr{X}_0\mathcal{E}_2)_n^{\!\top} \left( \tilde{\mathbf{y}} + \frac{i}{2\pi\sqrt{\hbar}} (\boldsymbol{\pi - a} ) \right)  \right) 
\delta\left((\mathcal{E}_1\mathcal{B})_n \left( \tilde{\mathbf{y}} + \frac{i}{2\pi\sqrt{\hbar}} (\boldsymbol{\pi - a}) \right)\right) \\
& \qquad\qquad 
\exp\left(- i\pi \tilde{\mathbf{y}}^{\!\top} Q \tilde{\mathbf{y}} + \frac{1}{\sqrt{\hbar}} \mathscr{W}(\alpha)^{\!\top}  \tilde{\mathbf{y}}  - i\pi \sum_{k=1}^N \left(\frac{\epsilon(T_k)-1}{2}\right) \tilde{\mathbf{y}_k}^2 \right) \left(\prod_{i=1}^N \Phi_\B(\tilde{y_i})\right)^{-1}d\mathbf{\tilde y} \Bigg|  ,
\end{align*}
where $\mathscr{W}(\alpha)= Q(\boldsymbol{\pi - a}) + C(\alpha)$. By applying the change of variable $\mathbf{y} = 2\pi \sqrt{\hbar}\tilde{\mathbf{y}}$, we get the desired result.
\end{proof}

For generalized FAMED ideal triangulations, the following lemma provides a geometric interpretation of the linear affine  $\mathscr{S_{\alpha} }$ in Proposition \ref{Tpartiexpress1}.
\begin{lemma}\label{lincorrespond}
Assume that the ideal triangulation $X$ is generalized FAMED with respect to $l$ (see Definition \ref{defgenFAMED}. Then via the bijection 
$$y_k = \Log z_k - i \pi$$ 
defined in Section \ref{sub:thurston}, each point on $\mathscr{S_{\alpha} }$ corresponds to a solution of the equations 
$$\mathbf{EA}_{2n} (\text{\bf Log } \mathbf{z}) = i\mathbf{E}_{2n}(\boldsymbol{\nu+u}).$$
In particular, the affine space $\mathscr{S_{\alpha}}$ only depends on $\lambda_X(\alpha)$. Moreover, if $X$ satisfies Definition \ref{defgenFAMED5}(2), then the affine space $\mathscr{S_{\alpha}}$ is independent of the choices of angle structures.
\end{lemma}
\begin{proof}
    Recall that the defining equations of $\mathscr{S_{\alpha} }$ are given by 
    $$
 ((\mathscr{X}_0\mathcal{E}_2)_n)^{\!\top} 
 \left( \mathbf{y} + i  (\boldsymbol{\pi - a})    \right)
 =
(\mathcal{E}_1\mathcal{B})_n 
 \left( \mathbf{y} + i (\boldsymbol{\pi - a})    \right)
 = 0, $$
which under the bijection $y_k = \Log z_k - i \pi$ can be written as
$$
\begin{pmatrix}
((\mathscr{X}_0\mathcal{E}_2)_n)^{\!\top} \\
(\mathcal{E}_1\mathcal{B})_n
\end{pmatrix}
\BLog \mathbf{z}
= 
i
\begin{pmatrix}
((\mathscr{X}_0\mathcal{E}_2)_n)^{\!\top} \\
(\mathcal{E}_1B)_n
\end{pmatrix}
(\boldsymbol{a}).
$$
By Definition \ref{defgenFAMED}(3), this system of equations is equivalent to 
$$(\mathbf{EA})_{2n} (\BLog \mathbf{z}) = i(\mathbf{EA})_{2n}(\boldsymbol{a}).$$
Since $\alpha \in \mathcal{A}_X$, from the discussion in Section \ref{sub:thurston}, we have
\begin{align*}
   \mathbf{A}
\begin{pmatrix}
\boldsymbol{a_+}\\
\boldsymbol{a_-}
\end{pmatrix}
+
\mathbf{B}
\begin{pmatrix}
\boldsymbol{b_+}\\
\boldsymbol{c_-}
\end{pmatrix} &=  \boldsymbol\nu + \boldsymbol u.
\end{align*}
By multiplying both sides by $\mathbf{E}$, we have
\begin{align*}
    \mathbf{E}\mathbf{A}
\begin{pmatrix}
\boldsymbol{a_+}\\
\boldsymbol{a_-}
\end{pmatrix}
+
\mathbf{E}\mathbf{B}
\begin{pmatrix}
\boldsymbol{b_+}\\
\boldsymbol{c_-}
\end{pmatrix} &=  
\mathbf{E}(\boldsymbol\nu + \boldsymbol u).
\end{align*}
In particular, the last $2n$ equations give
$(\mathbf{EA})_{2n}(\boldsymbol{a}) = \mathbf{E}_{2n}(\boldsymbol{\nu+u})$. This proves the first two claims. Moreover, if $X$ satisfies Definition \ref{defgenFAMED5}(2), then we have 
$(\mathbf{EA})_{2n}(\boldsymbol{a}) = \mathbf{E}_{2n}\boldsymbol{\nu}$. This proves the last claim.
\end{proof}

The next lemma tells us how to parameterize the affine subspace $\mathscr{S_{\alpha} }$ in Proposition \ref{Tpartiexpress1}. Up to renumbering, assume that the pivot positions of $\mathbf{(EB)}^{\!\top}_{N-2n}$ are $1,2,\dots, N-2n$. 
\begin{lemma}\label{preNZpara}
    $\ker (\mathbf{EA})_{2n} = \im( (\mathbf{EB})_{N-2n}^{\!\top})$, $\mathrm{rank}((\mathbf{EA})_{2n}) = 2n$ and $\mathrm{rank}((\mathbf{EB})_{N-2n}^{\!\top}) = N-2n$.
\end{lemma}
\begin{proof}
    It is known that the Neumann-Zagier symplectic form $\omega$ restricted on the row spaces of $(\mathbf{A} | \mathbf{B})$ is zero \cite{NZ}. Since the matrix $\mathbf{E}$ corresponds to a finite sequence of elementary row operations, the row spaces of $(\mathbf{EA} | \mathbf{EB})$ is the same as that of $(\mathbf{A} | \mathbf{B})$. Furthermore, for any two rows $r = (r_A | r_B)$ and $s = (s_A | s_B)$ of $(\mathbf{EA} | \mathbf{EB})$ with $r_A, r_B, s_A, s_B \in M_{1\times N}(\QQ)$, we have
$$
\omega( r,s) 
= r_A \cdot s_B - r_B \cdot s_A,
$$
where $\cdot$ denotes the usual dot product of $\CC^N$. In particular, if $r = (r_A | 0)$ is a row of $(\mathbf{EA}_{2n} | \mathbf{O})$ and $s = (s_A | s_B)$ is a row of $( (\mathbf{EA})_{N-2n} | (\mathbf{EB})_{N-2n})$, then 
$$ r_A\cdot s_B = \omega(r,s) = 0. $$
As a result, we have
\begin{align}\label{NZmultizero}
    \mathbf{(EA)_{2n}}
 (\mathbf{(EB)}^{\!\top}_{N-2n}) = 0.
\end{align}
Recall that $\mathbf{E}$ is a product of elementary matrix such that
$$
(\mathbf{EB}| \mathbf{EA})
=
\begin{pmatrix}
 \mathbf{ (EB) }_{N-2n} & (\mathbf{EA})_{N-2n} \\
 \mathbf{O} & (\mathbf{EA})_{2n}
\end{pmatrix}
$$
is in the reduced row echelon form. Thus, we have rank$(\mathbf{EB}|\mathbf{EA})=$ rank$(\mathbf{B}|\mathbf{A}) =N$ \cite{NZ}. This implies that $\text{rank}((\mathbf{EA})_{2n}) = 2n$ and 
 $\text{rank}\left(\mathbf{(EB)}^{\!\top}_{N-2n}\right) 
= \text{rank}\left(\mathbf{(EB)}_{N-2n} \right)
= N-2n.$ Since nullity$((\mathbf{EA})_{2n})$ = $N -$ rank$((\mathbf{EA})_{2n})$ $= N-2n =$ rank $\left(\mathbf{(EB)}_{N-2n}^{\!\top} \right)$, together with (\ref{NZmultizero}), we have the desired result.
\end{proof}

Let $\mathscr{A}_X^l(\lambda_X(\alpha))$ be the subset of $\mathscr{A}_X$ defined by 
$$
\mathscr{A}_X^l(\lambda_X(\alpha))
= \{ \alpha' \in \mathscr{A}_X \mid  \mathrm{H}^\RR_{X,l}(\alpha') = \lambda_X(\alpha) \}.
$$
\begin{lemma}\label{NZpara}
Let $\alpha^0 = (a_1^0, \dots, c_N^0) \in \mathscr{A}_X^l(\lambda_X(\alpha))$ be a fixed angle structure in $\mathscr{A}_X^l(\lambda_X(\alpha))$. Let $\boldsymbol{a^0}=(a_1^0,a_2^0,\dots,a_N^0)$ be the $a$-angles of $\alpha^0$.
    The affine space in $\CC^{N}$ defined by 
   $$
 ((\mathscr{X}_0\mathcal{E}_2)_n)^{\!\top} 
 \left( \mathbf{y} + i  (\boldsymbol{\pi - a})    \right)
 =
(\mathcal{E}_1\mathcal{B})_n 
 \left( \mathbf{y} + i (\boldsymbol{\pi - a})    \right)
 = 0 $$
can be parameterized by the affine map $\varphi$ given by
$$
\boldsymbol{x} = (x_1,\dots,x_{N-2n}) \in \CC^{N-2n} \mapsto  \mathbf{y} = (y_1,\dots, y_N) =   (\varphi_1(\boldsymbol{x}),\dots, \varphi_N(\boldsymbol{x})) = \varphi(\boldsymbol{x}), 
$$
where
$$\varphi(\boldsymbol{x}) = (\mathbf{(EB)}^{\!\top}_{N-2n}) \boldsymbol{x} - i(\boldsymbol{\pi - a^0}).
$$
Furthermore, we have
$$
\varphi\left( \RR^{N-2n} + i \mathbf{v}_\alpha \right) =  \prod_{k=1}^N \big(\RR - i(\pi-a_k)\big) ,
$$
where 
$$\mathbf{v}_\alpha = (a_1 - a_1^0, \dots, a_{N-2n} - a_{N-2n}^0) .$$
Moreover, if $X$ satisfies Definition \ref{defgenFAMED5}(2), then $\alpha^0$ can be any angle structure in $\mathcal{A}_X$. 
\end{lemma}
\begin{proof}
By Definition \ref{defgenFAMED}(3), the equation
   $$
 ((\mathscr{X}_0\mathcal{E}_2)_n)^{\!\top} 
 \left( \mathbf{y} + i  (\boldsymbol{\pi - a})    \right)
 =
(\mathcal{E}_1\mathcal{B})_n 
 \left( \mathbf{y} + i (\boldsymbol{\pi - a})    \right)
 = 0 $$
 is equivalent to 
 $$
 (\mathbf{EA})_{2n}\left( \mathbf{y} + i  (\boldsymbol{\pi - a})    \right) = 0.
 $$
 By Lemma \ref{preNZpara}, for any $\boldsymbol{x}\in \CC^{N-2n}$, 
\begin{align*}
    (\mathbf{EA})_{2n}\left( \varphi(\boldsymbol{x}) + i  (\boldsymbol{\pi - a})    \right) 
    = i(\mathbf{EA})_{2n}(\boldsymbol{a^0 - a}) = 0,
\end{align*}
where the last equality follows from the fact that $(\mathbf{EA})_{2n}(\boldsymbol{a^0}) = (\mathbf{EA})_{2n}(\boldsymbol{a}) = \mathbf{E}_{2n}(\boldsymbol{\nu+u})$. This shows that the image of the affine map lies in the affine subspace. By Lemma \ref{preNZpara}, since 
$ \text{rank}\left(\mathbf{(EB)}^{\!\top}_{N-2n}\right) = N-2n,$ we prove the first claim. For the second claim, since the pivot positions of $\mathbf{(EB)}_{N-2n}$ are $1,2,\dots, N-2n$, for $k=1,\dots, N-2n$, by solving the equation
$$  i v_k - i (\pi -a^0_k) =  - i (\pi -a_k) ,$$
we have $v_k = a_k -a^0_k $. Finally, the last claim follows from the observation that by Definition \ref{defgenFAMED5} (2),  
$$(\mathbf{EA})_{2n}(\boldsymbol{a^0}) = (\mathbf{EA})_{2n}(\boldsymbol{a}) = \mathbf{E}_{2n}\boldsymbol{\nu}$$
for any angle structure $\alpha^0 \in \mathcal{A}_X$.
\end{proof}

The following lemma relates $\mathscr{W}(\alpha)$ with the angular holonomy of the longitude.
\begin{lemma}\label{computeW}
Let $X$ be a generalized FAMED ideal triangulation. We have
\begin{align*}
-\mathbf{(EB)}_{N-2n}
\left(
\mathscr{W}(\alpha) \right)
=
\mathbf{E_{N-2n}'(\boldsymbol \nu +\boldsymbol{u})}
- 
\mathbf{(EB)}_{N-2n}\mathscr{G}
\boldsymbol{\pi}.
\end{align*}
In particular, this quantity only depends on $\lambda_X(\alpha)$. 
\end{lemma}
\begin{proof}
Recall that for positively (resp. negatively) ordered tetrahedron, we have $a = {\Arg(z)}$ and $b = {\Arg(z'')}$ (resp. $a = {\Arg(z)}$ and $c = {\Arg(z'')}$). Besides, recall that
$$\mathscr{G} = 
Q + \frac{\mathcal{E}+\mathrm{Id}_N}{2} 
=
Q  
+
\begin{pmatrix}
\mathbf{1} & \mathbf{0} \\
\mathbf{0} & \mathbf{0}
\end{pmatrix},
$$
where $Q = (\mathscr{X}_0\mathcal{E}_2)_r (\mathcal{E}_1\mathscr{B})_r  + \big((\mathscr{X}_0\mathcal{E}_2)_r (\mathcal{E}_1\mathscr{B})_r\big)^{\!\top}$.
Note that 
\begin{align*}
-\mathscr{W}(\alpha)
&= 
Q
\begin{pmatrix}
\boldsymbol{a_+} - \boldsymbol\pi \\
\boldsymbol{a_-} - \boldsymbol\pi
\end{pmatrix}
+
\begin{pmatrix}
\boldsymbol{a_+} + \boldsymbol{b_+} - \boldsymbol{\pi}\\
\boldsymbol{c_-}
\end{pmatrix}
\\
&= 
\left(
\mathscr{G}
-
\begin{pmatrix}
\mathbf{1} & \mathbf{0} \\
\mathbf{0} & \mathbf{0}
\end{pmatrix}
\right)
\begin{pmatrix}
\boldsymbol{a_+} - \boldsymbol\pi \\
\boldsymbol{a_-} - \boldsymbol\pi
\end{pmatrix}
+
\begin{pmatrix}
\boldsymbol{a_+} + \boldsymbol{b_+} - \boldsymbol{\pi}\\
\boldsymbol{c_-}
\end{pmatrix}
\\
&=
\left[
\mathscr{G}
\begin{pmatrix}
\boldsymbol{a_+}\\
\boldsymbol{a_-}
\end{pmatrix}
+
\begin{pmatrix}
\boldsymbol{b_+}\\
\boldsymbol{c_-}
\end{pmatrix}
\right]
-
\mathscr{G}
\begin{pmatrix}
\boldsymbol{\pi}\\
\boldsymbol{\pi}
\end{pmatrix}.
\end{align*}
Besides, by the definitions of $\mathbf{E}, \mathbf{A}$ and $\mathbf{B}$, we have
\begin{align*}
   \mathbf{(EA)}_{N-2n}
\begin{pmatrix}
\boldsymbol{a_+}\\
\boldsymbol{a_-}
\end{pmatrix}
+
\mathbf{(EB)}_{N-2n}\begin{pmatrix}
\boldsymbol{b_+}\\
\boldsymbol{c_-}
\end{pmatrix} &= ( \mathbf{E}( \boldsymbol\nu + \boldsymbol u) )_{N-2n} \\
    (\mathbf{EA})_{2n} \begin{pmatrix}
\boldsymbol{a_+}\\
\boldsymbol{a_-}
\end{pmatrix} &= ( \mathbf{E}(\boldsymbol\nu + \boldsymbol u) )_{2n},
\end{align*}
where $( \mathbf{E}( \boldsymbol\nu + \boldsymbol u) )_{N-2n}$ and $( \mathbf{E}(\boldsymbol\nu + \boldsymbol u) )_{2n}$ are the first $N-2n$ and last $2n$ entries of $\mathbf{E}(\boldsymbol\nu + \boldsymbol u) $, respectively. 
By Definition \ref{defgenFAMED} (4), Remark \ref{rmk4} and Equation (\ref{defnE'}), 
\begin{align*}
   \mathbf{(EB)}_{N-2n}\mathscr{G}
\begin{pmatrix}
\boldsymbol{a_+}\\
\boldsymbol{a_-}
\end{pmatrix}
+
\mathbf{(EB)}_{N-2n}\begin{pmatrix}
\boldsymbol{b_+}\\
\boldsymbol{c_-}
\end{pmatrix} &= \mathbf{E}_{N-2n}'( \boldsymbol\nu + \boldsymbol u)  .
\end{align*}
As a result, 
\begin{align*}
\ -\mathbf{(EB)}_{N-2n}
\left( 
\mathscr{W}(\alpha) \right) 
=&  \ 
\left[
\mathbf{(EB)}_{N-2n}\mathscr{G}
\begin{pmatrix}
\boldsymbol{a_+}\\
\boldsymbol{a_-}
\end{pmatrix}
+
\mathbf{(EB)}_{N-2n}\begin{pmatrix}
\boldsymbol{b_+}\\
\boldsymbol{c_-}
\end{pmatrix}
\right]
-
\mathbf{(EB)}_{N-2n}\mathscr{G}
\begin{pmatrix}
\boldsymbol{\pi}\\
\boldsymbol{\pi}
\end{pmatrix}
\\
=&\ \  
(\mathbf{E'(\boldsymbol \nu +\boldsymbol{u})})_{N-2n}
- 
\mathbf{(EB)}_{N-2n}\mathscr{G}
\begin{pmatrix}
\boldsymbol{\pi}\\
\boldsymbol{\pi}
\end{pmatrix}
.
\end{align*}
This completes the proof.
\end{proof}

\begin{proposition}\label{Tpartiexpress2}
Let $X$ be a generalized FAMED ideal triangulation. 
Let $\alpha \in \mathscr{A}_{X}$ be an angle structure given by $\alpha = (a_1,b_1,c_1,\dots,a_N, b_N, c_N)$. For all $\hbar>0$, we have 
\begin{align*}
\ \ \left|\mathcal{Z}_{\hbar}(X,\alpha)\right| = &\left(\frac{1}{2\pi \sqrt{\hbar}}\right)^{N-2n} D_1\\
\ \  \Bigg|\int_{\boldsymbol{x} \in \RR^{N-2n} + i\mathbf{v}_\alpha } 
&\exp\left(\frac{1}{2\pi \hbar} \left( - \frac{i}{2} \varphi(\boldsymbol{x})^{\!\top} Q \varphi(\boldsymbol{x})  - \frac{i}{2} \sum_{k=1}^N \left(\frac{\varepsilon(T_k)-1}{2}\right) \varphi_k(\boldsymbol{x})^2 \right) \right)   \\
& \exp\left(\frac{1}{2\pi\hbar}(- \boldsymbol{x}^{\!\top}  (\mathbf{E_{N-2n}'(\boldsymbol \nu +\boldsymbol{u})- 
\mathbf{(EB)}_{N-2n}\mathscr{G}
\boldsymbol{\pi}) )}\right) \\
& \left(\prod_{k=1}^N \Phi_\B\left(\frac{\varphi_k(\boldsymbol{x})}{2\pi \sqrt{\hbar}}\right)\right)^{-1}  d\boldsymbol{x}\Bigg| 
\end{align*}
where $\varphi(\boldsymbol{x})$ is the linear map defined in Lemma \ref{NZpara} and $D_1$ is some constant independent of $\hbar$ and $\alpha$.
Furthermore, $|\mathscr{Z}_{\hbar}(X, \alpha)|$ depends only on $\lambda_X(\alpha)$.
\end{proposition}
\begin{proof}
For the first claim, by Proposition \ref{Tpartiexpress1} and Lemma \ref{NZpara}, we have 
\begin{align*}
&\ \ \left|\mathcal{Z}_{\hbar}(X,\alpha)\right| = \left(\frac{1}{2\pi \sqrt{\hbar}}\right)^{N-2n}|\det(\mathcal{E}_1\mathcal{E}_2)|\\
&\ \  \left|\int_{\boldsymbol{x} \in \RR^{N-2n} + i\mathbf{v}_\alpha } 
\frac{\exp\left(\frac{1}{2\pi \hbar} \left( - \frac{i}{2} \varphi(\boldsymbol{x})^{\!\top} Q \varphi(\boldsymbol{x}) + \varphi(\boldsymbol{x})^{\!\top} \mathscr{W}(\alpha) - \frac{i}{2} \sum_{k=1}^N \left(\frac{\varepsilon(T_k)-1}{2}\right) \varphi_k(\boldsymbol{x})^2 \right) \right)}{\prod_{k=1}^N \Phi_\B\left(\frac{\varphi_k(\boldsymbol{x})}{2\pi \sqrt{\hbar}}\right) } d\boldsymbol{x} \right| .
\end{align*}
By Lemma \ref{NZpara} and \ref{computeW}, we have
\begin{align*}
    \left| \exp\left( \frac{1}{2\pi\hbar}( \varphi(\boldsymbol{x})^{\!\top}\mathscr{W}(\alpha)  \right) \right|
    &= \left| \exp\left( \frac{1}{2\pi\hbar}(  \boldsymbol{x}^{\!\top} (\mathbf{EB})_{N-2n}\mathscr{W}(\alpha)  \right) \right|\\
    &= \left| \exp\left( \frac{1}{2\pi\hbar}( - \boldsymbol{x}^{\!\top}  (\mathbf{E_{N-2n}'(\boldsymbol \nu +\boldsymbol{u})}
- 
\mathbf{(EB)}_{N-2n}\mathscr{G}
\boldsymbol{\pi}))  \right) \right|.
\end{align*}
As a result, 
\begin{align*}
\ \ \left|\mathcal{Z}_{\hbar}(X,\alpha)\right| = &\left(\frac{1}{2\pi \sqrt{\hbar}}\right)^{N-2n} D_1\\
\ \  \Bigg|\int_{\boldsymbol{x} \in \RR^{N-2n} + i\mathbf{v}_\alpha } 
&\exp\left(\frac{1}{2\pi \hbar} \left( - \frac{i}{2} \varphi(\boldsymbol{x})^{\!\top} Q \varphi(\boldsymbol{x})  - \frac{i}{2} \sum_{k=1}^N \left(\frac{\varepsilon(T_k)-1}{2}\right) \varphi_k(\boldsymbol{x})^2 \right) \right)  \\
& \exp\left(\frac{1}{2\pi\hbar}(- \boldsymbol{x}^{\!\top}  (\mathbf{E_{N-2n}'(\boldsymbol \nu +\boldsymbol{u})- 
\mathbf{(EB)}_{N-2n}\mathscr{G}
\boldsymbol{\pi}) )}\right) \\
& \left(\prod_{k=1}^N \Phi_\B\left(\frac{\varphi_k(\boldsymbol{x})}{2\pi \sqrt{\hbar}}\right)\right)^{-1}  d\boldsymbol{x}\Bigg| 
\end{align*}
for some constant $D_1$ independent of $\hbar$ and $\alpha$. 

For the second claim, note that although the integrand only depends on $\lambda_X(\alpha)$, a priori the integration multi-contour depends on $\alpha$. We will show that we can deform the integration multi-contour without changing the integral as follows.
Consider the subspace of angle structure
$\mathscr{A}^l_{X}(\lambda_X(\alpha))=  \{ \alpha \in \mathscr{A}_X \mid \mathrm{H}^\R_{X,l}(\alpha) = \lambda_X(\alpha) \}$. Let $\alpha_1,\alpha_2 \in \mathscr{A}^l_{X}(\lambda_X(\alpha))$. Since $\mathscr{A}_X$ is convex, the affine subspace $\mathscr{A}^l_{X}(\lambda_X(\alpha))$, which is a restriction of a convex set on a affine subspace, is also convex. Consider the projection map $\mathrm{proj_a}: \mathscr{A}^l_{X}(\lambda_X(\alpha)) \to \RR^{N-2n}$ defined by $\mathrm{proj}_a(\alpha) = (a_1,\dots, a_{N-2n})$. Since $\mathrm{proj}_a$ is an affine map and $\mathscr{A}^l_{X}(\lambda_X(\alpha))$ is convex, the image $\mathrm{proj_a}(\mathscr{A}^l_{X}(\lambda_X(\alpha)) )$ is also convex. Moreover, for any $(a_1,\dots, a_{N-2n})$ sufficiently close to $\mathrm{proj_a}(\alpha_1)$, we can find an angle structure $\alpha \in \mathscr{A}^l_{X}(\lambda_X(\alpha))$ such that $\mathrm{proj}_a(\alpha) = (a_1,\dots, a_{N-2n})$. This shows that the image $\mathrm{proj_a}(\mathscr{A}^l_{X}(\lambda_X(\alpha)) )$ has dimension $N-2n$. Altogether, for any $\alpha' \in \mathscr{A}^l_{X}(\lambda_X(\alpha))$ with $\mathrm{proj}_a(\alpha')=(a_1', \dots, a_{N-2n}')$ with $a_k' \in [\min\{a^1_k, a^2_k\}, \max\{a^1_k,a^2_k\}]$, the formula for $|\mathscr{Z}_{\hbar}(X, \alpha')|$ holds. In particular, this implies the exponential decay properties of the integrand at infinity and the absolute convergence of the integral for all such $(a_1', \dots, a_{N-2n}')$. As a result, by Bochner-Martinelli formula (see \cite{Kr}), we have
$|\mathscr{Z}_{\hbar}(X, \alpha_1)| = |\mathscr{Z}_{\hbar}(X, \alpha_2)|$
for all $\alpha_1, \alpha_2 \in \mathscr{A}^l_X(\lambda_X(\alpha))$. This completes the proof. 
\end{proof}

\subsection{Potential function and its properties}\label{pofunc}
Consider the potential function 
\begin{align*}
    S\left (\boldsymbol{x}; \lambda_X(\alpha)\right ) 
    =&\ -\frac{i}{2} \varphi(\boldsymbol{x})^{\!\top}
Q \varphi(\boldsymbol{x})
 - \frac{i}{2}\sum_{k=1}^N \left(\frac{\varepsilon(T_k)-1}{2}\right) \varphi_k(\boldsymbol{x})^2
 + i \sum_{k=1}^N \mathrm{L}(\varphi_k(\boldsymbol{x})) \\
&- \boldsymbol{x}^{\!\top}  (\mathbf{E_{N-2n}'(\boldsymbol \nu +\boldsymbol{u})- 
\mathbf{(EB)}_{N-2n}\mathscr{G}
\boldsymbol{\pi}}),
\end{align*}
where $\boldsymbol u
= (0,\dots,0, \lambda_X(\alpha))^{\!\top}$.
The function $S$ plays an important role in the asymptotics of the partition function. More generally, we can complexified the parameter $\lambda_X(\alpha)$ and consider the potential function $\tilde{S}$ defined by
\begin{align*}
    \tilde{S}\left (\boldsymbol{x}; \xi\right ) 
    =&\ -\frac{i}{2} \varphi(\boldsymbol{x})^{\!\top}
Q \varphi(\boldsymbol{x})
 - \frac{i}{2}\sum_{k=1}^N \left(\frac{\varepsilon(T_k)-1}{2}\right) \varphi_k(\boldsymbol{x})^2
 + i \sum_{k=1}^N \mathrm{L}(\varphi_k(\boldsymbol{x})) \\
&\ - \boldsymbol{x}^{\!\top}  (\mathbf{E_{N-2n}'(\boldsymbol \nu -i\boldsymbol{\tilde u})- 
\mathbf{(EB)}_{N-2n}\mathscr{G}
\boldsymbol{\pi}}),
\end{align*}
where $\boldsymbol{\tilde u}
= (0,\dots,0,
\xi)$ with $\xi \in \C$.
The functions $S$ and $\tilde{S}$ are related by 
$$
S\left(\boldsymbol{x};\lambda_X(\alpha)\right)
= \tilde{S}\left (\boldsymbol{x}; i\lambda_X(\alpha) \right ) .
$$
We relate the critical point equations of the potential function with the gluing equations as follows. 
\begin{proposition}\label{critThurscorrespondence}
Assume that the ideal triangulation $X$ is generalized FAMED. 
\begin{enumerate}
    \item For each fixed $\xi \in \CC$, the bijection 
\begin{align}\label{yzrelate}
    y_k = \Log z_k - i \pi
\end{align}
defined in Section \ref{sub:thurston} gives a correspondence between a solution of the gluing equations
\begin{align*}
\mathbf{A} \BLog \mathbf{z}
+ \mathbf{B}
\BLog \mathbf{z}''
=
i \boldsymbol \nu + \tilde{\boldsymbol{u}}.
\end{align*}
with non-negative imaginary parts and a solution of the critical point equation 
$$\nabla_{\boldsymbol{x}} \tilde{S}(\boldsymbol{x}; \xi) = 0.$$
     \item Given an angle structure $\alpha' \in \mathcal{A}_{X}^l(\lambda_X(\alpha))$, let $\mathbf{z}_{\alpha'}$ be the corresponding shape parameters. Under the notations in Lemma \ref{NZpara}, let $\boldsymbol{x}_{\alpha'} \in \RR^{N-2n} + i \mathbf{v}_{\alpha'}$ be the point such that $\mathbf{y}_{\alpha'}= \varphi(\boldsymbol{x}_{\alpha'})$ and $\mathbf{z}_{\alpha'}$ are related by the bijection (\ref{yzrelate}). Then  $\boldsymbol{x}_{\alpha'}$ is a critical point of  $\Re S\left(\boldsymbol{x}; \lambda_X(\alpha)\right)$ on the horizontal plane $\RR^{N-2n} + i \mathbf{v}_{\alpha'}$.
\end{enumerate}
\end{proposition}
\begin{proof}
Note that $Q$ is symmetric and the Jacobian of $\varphi$ is $\mathbf{EB}_{N-2n}$. Recall the notations in Lemma \ref{NZpara} that $\mathbf{y} = (y_1,\dots,y_N) = (\varphi_1(\boldsymbol{x}),\dots,\varphi_N(\boldsymbol{x}))$.  By chain rule, we have
\begin{align*}
\nabla_{\boldsymbol{x}} \tilde S (\boldsymbol{x}; \xi)
&= (\mathbf{EB})_{N-2n}\left[- i Q 
\begin{pmatrix}
\mathbf{y_+} \\
 \mathbf{y_-}
\end{pmatrix}
+ i \begin{pmatrix}
    \mathbf{0} & \mathbf{0} \\
    \mathbf{0} & \mathbf{1}
\end{pmatrix}
\begin{pmatrix}
\mathbf{y_+} \\
 \mathbf{y_-}
\end{pmatrix}
- i 
\begin{pmatrix}
 \mathbf{\BLog(1+e^{y_+})} \\
\mathbf{ \BLog(1+e^{y_-})}
\end{pmatrix} \right] \\
& \qquad -(\mathbf{E}_{N-2n}'(\boldsymbol \nu -i\boldsymbol{\tilde u})- 
\mathbf{(EB)}_{N-2n}\mathscr{G}
\boldsymbol{\pi}) \\
&= (\mathbf{EB})_{N-2n}[- i \mathscr{G} 
 \mathbf{y}
+ i 
 \mathbf{y}
- i 
\mathbf{ \BLog(1+e^{y})}]
 -[\mathbf{E}_{N-2n}'(\boldsymbol \nu -i\boldsymbol{\tilde u})- 
\mathbf{(EB)}_{N-2n}\mathscr{G}
\boldsymbol{\pi}]
\\
&= -i(\mathbf{EB})_{N-2n}\left[\mathscr{G} 
(\mathbf{y} + \boldsymbol{i\pi} )
+ 
\mathbf{ {\BLog}(1+e^{-y})}
 \right]
-\mathbf{E}_{N-2n}'(\boldsymbol \nu -i\boldsymbol{\tilde u})
.
\end{align*}
Recall that $y_k = \Log z_k - i \pi$ and $\Log(1+e^{-y_k}) = \Log(z_k'')$ for $k=1,\dots, N$. 
Thus, 
\begin{align}\label{nablatildeS}
\nabla_{\boldsymbol{x}} \tilde S (\boldsymbol{x}; \xi)
&= - i[(\mathbf{EB})_{N-2n}\mathscr{G}  \BLog \mathbf{z} + (\mathbf{EB})_{N-2n} \BLog(\mathbf{z}'')] 
-\mathbf{E}_{N-2n}'(\boldsymbol \nu -i\boldsymbol{\tilde u}).
\end{align}
In particular, $\nabla_{\boldsymbol{x}} \tilde{S}(\boldsymbol{x}; \xi) =0$ can be written as
\begin{align}\label{gluingeq}
   (\mathbf{EB})_{N-2n}\mathscr{G}
{\BLog}(\mathbf{z})
+
\mathbf{(EB)}_{N-2n}
{\BLog}(\mathbf{z}'')
&= \mathbf{E}_{N-2n}'( i\boldsymbol\nu + \tilde{\boldsymbol u}) .
\end{align}
By Remark \ref{rmk4}, Equation (\ref{gluingeq}) is equivalent to 
$$
(\mathbf{EA})_{N-2n}
{\BLog}(\mathbf{z})
+
\mathbf{(EB)}_{N-2n}
{\BLog}(\mathbf{z}'')
= \mathbf{E}_{N-2n}( i\boldsymbol\nu + \tilde{\boldsymbol u}).
$$
As a result, a solution of the gluing equations give a solution of the critical point equation. This proves the first claim. 

For the second claim, from Equation (\ref{nablatildeS}), we have
\begin{align}
&\ \nabla_{\Re \boldsymbol{x}} \Re S (\boldsymbol{x};\lambda_X(\alpha)) \notag\\
=&\  
\Re\nabla_{\boldsymbol{x}} S (\boldsymbol{x};\lambda_X(\alpha)) \notag\\
=&\    (\mathbf{EB})_{N-2n}\mathscr{G}
\cdot \mathbf{Arg}(\mathbf{z})
+
(\mathbf{EB})_{N-2n}
\cdot \mathbf{Arg}(\mathbf{z}'')
-\mathbf{E}_{N-2n}'(\boldsymbol \nu + \boldsymbol{u})
\label{Recrit1},
\end{align}
where $\boldsymbol u
= (0,\dots,0,
\mathrm{H}^\R_{X,l}(\alpha))^{\!\top}$.
Besides, since $\alpha' \in \mathcal{A}_{X}^l(\lambda_X(\alpha))$, we have
\begin{align*}
   \mathbf{A}
\begin{pmatrix}
\boldsymbol{a_+}\\
\boldsymbol{a_-}
\end{pmatrix}
+
\mathbf{B}
\begin{pmatrix}
\boldsymbol{b_+}\\
\boldsymbol{c_-}
\end{pmatrix} &=  \boldsymbol\nu + \boldsymbol u,
\end{align*}
which implies that
\begin{align}\label{Recrit2}
        (\mathbf{EB})_{N-2n}\mathscr{G}
\begin{pmatrix}
\boldsymbol{a_+}\\
\boldsymbol{a_-}
\end{pmatrix}
+
(\mathbf{EB})_{N-2n}
\begin{pmatrix}
\boldsymbol{b_+}\\
\boldsymbol{c_-}
\end{pmatrix} &=  \mathbf{E'}(\boldsymbol\nu + \boldsymbol u).
\end{align}
From (\ref{Recrit1}) and (\ref{Recrit2}), we can see that the shape parameters $\mathbf{z}_{\alpha'}$ coming from angle structure $\alpha' \in \mathcal{A}_{X}^l(\lambda_X(\alpha))$ give a solution to the equation $\nabla_{\Re \boldsymbol{x}} \Re S (\boldsymbol{x};\lambda_X(\alpha))=0$. This completes the proof.
\end{proof}

We first study the critical value of $S$.
\begin{proposition}\label{dilogVolgen}
Write $x_l = h_l + i d_l \in \CC$ for $l=1,\dots,N-2n$.
Under the assumption and the correspondence described in Proposition \ref{critThurscorrespondence}, we have 
\begin{align*}
\mathrm{Re}S(\boldsymbol{x}; \lambda_X(\alpha)) 
= - \sum_{l}^N D(z_l) - \sum_{l=1}^{N-2n} h_l \frac{\partial}{\partial h_l} \mathrm{Re}S(\mathbf{y}; \lambda_X(\alpha)) ,
\end{align*}
where $D(z)$ is the Bloch-Wigner dilogarithm function given by
$$ D(z) 
= \mathrm{Im} \mathrm{Li}_2(z) + \log|z| \mathrm{Arg}(1-z).$$
In particular, 
\begin{enumerate}
    \item if $\mathbf{x^c}$ be the critical point of $S$ described in Proposition \ref{critThurscorrespondence}, then we have
$$
\mathrm{Re}S(\mathbf{x^c};  \lambda_X(\alpha)) =  -\Vol\Big(\SS^3 \setminus K, \mathrm{H}^\CC_{X,l}(\mathbf{z^c}) = i \lambda_X(\alpha)\Big),
$$
where $\Vol(\SS^3 \setminus K, \mathrm{H}^\CC_{X,l}(\mathbf{z^c}) = i \lambda_X(\alpha))$ is the hyperbolic volume of $\SS^3\backslash K$ with (possibly incomplete) hyperbolic structure satisfying $\mathrm{H}^\CC_{X,l}(\mathbf{z^c}) = i \lambda_X(\alpha)$;
    \item for every angle structure $\alpha'\in \mathcal{A}_{X}^l(\lambda_X(\alpha))$, if $\boldsymbol{x}_{\alpha'}$ is the corresponding point described in Proposition \ref{critThurscorrespondence}(2), we have
$$
\mathrm{Re}S(\boldsymbol{x}_{\alpha'}; \lambda_X(\alpha)) 
= - \Vol(\alpha'),
$$
where $\Vol(\alpha')$ is the volume of the angle structure $\alpha'$.
\end{enumerate}
\end{proposition}
\begin{proof}
Let $(\mathbf{EB})_{N-2n}^{\!\top} = (m_{kl})$ for some $m_{kl}\in \RR$, where $k=1,\dots,N$ and $l=1,\dots, N-2n$. 
For $l=1,\dots, N-2n$, by using the Cauchy-Riemann equation, we have
\begin{align*}
\frac{\partial}{\partial h_l} \mathrm{Re}\left(i \mathrm{L}(\varphi_k(\boldsymbol{x}))\right)  
=&\ \ \frac{\partial}{\partial h_l} \mathrm{Re}\left(i \mathrm{L}\left(\sum_{l=1}^{N-2n} m_{kl} x_l + i(\pi-a_k)\right)\right) \\
=&\ \  \mathrm{Re} \left( \frac{d}{dx_l} \left(i \mathrm{L}\left(\sum_{l=1}^{N-2n} m_{kl} x_l + i(\pi-a_k)\right)\right) \right) \\
=&\ \  m_{kl} \im  \frac{d\mathrm{L}}{dx}\left(\sum_{l=1}^{N-2n} m_{kl} x_l+i(\pi-a_k)\right).
\end{align*}
By Lemma \ref{tildeLitoD}, 
\begin{align*}
&\ \ \mathrm{Re}\left(i \mathrm{L}\left(\sum_{l=1}^{N-2n} m_{kl} x_l+ i(\pi-a_k)\right)\right)  \\
=& \ 
-\mathrm{Im} \left(\mathrm{L}\left(\sum_{l=1}^{N-2n} m_{kl} x_l+ i(\pi-a_k)\right) \right)\\
=& \  - D\left(-e^{\sum_{l=1}^{N-2n} m_{kl} x_l+i(\pi-a_k)}\right) - \left(\sum_{l=1}^{N-2n} m_{kl} h_l \right) \im \frac{d\mathrm{L}}{dx}\left(\sum_{l=1}^{N-2n} m_{kl} x_l+i(\pi-a_k)\right) \\
=& \  - D\left(-e^{\sum_{j=1}^{N-2n} a_{ij} x_l+i(\pi-a_k)}\right) - \sum_{l=1}^{N-2n}  h_l \left(m_{kl} \im \frac{d\mathrm{L}}{dx}\left(\sum_{l=1}^{N-2n} m_{kl}x_l+i(\pi-a_k)\right) \right) \\
=& \  - D\left(-e^{\varphi_k(\boldsymbol{x})}\right) - \sum_{l=1}^{N-2n}  h_l \frac{\partial}{\partial h_l} \mathrm{Re}\left(i \mathrm{L}(\varphi_k(\boldsymbol{x}))\right) .
\end{align*}
Note that the real part of all the remaining non-dilogarithm terms are linear in $\{h_l \mid l=1,\dots, N-2n\}$.
The result then follows from summing up the equations.
\end{proof}

Let $\mathscr{U_{\mathbf{y}}}$ and $\overline{\mathscr{U_{\mathbf{y}}}}$ be the products of horizontal bands defined by
$$
\mathscr{U}_\mathbf{y} = \prod_{k=1}^N (\RR + i (-\pi,0)) \text{ and } \overline{\mathscr{U_\mathbf{y}}} = \prod_{k=1}^N (\RR + i [-\pi,0]) .
$$

\begin{proposition}\label{concavSgen}
Let $\alpha \in \mathscr{A}_{X}$. Given $\mathbf{v} \in \RR^{N-2n}$, consider the horizontal plane $\RR^{N-2n}+i\mathbf{v}$. If $\varphi(\RR^{N-2n}+i\mathbf{v}) \subset \mathscr{U}_{\mathbf{y}}$ (resp. $\varphi(\RR^{N-2n}+i\mathbf{v}) \subset \overline{\mathscr{U}_{\mathbf{y}}}$), then the real part of $\tilde{S}(\boldsymbol{x}; \lambda_X(\alpha))$ is strictly concave (resp. concave) in the variables $(\Re x_1, \dots, \Re x_{N-2n})$ and strictly convex (resp. convex) in the variables $(\mathrm{Im}\text{ } x_1, \dots, \mathrm{Im}\text{ } x_{N-2n})$.
\end{proposition}
\begin{proof}
Note that the hessian of the real part of $S$ is the same as the real part of the holomorphic hessian of $S$. In particular, we have
\begin{align*}
\big(\mathrm{Hess}(\Re (S) \big) (\mathbf{h} + i \mathbf{d};\lambda_X(\alpha))
= \left(\mathrm{Hess}\left(\Re \left( i \sum_{k=1}^N \mathrm{L}({\varphi_k(\boldsymbol{x})})\right) \right)\right)(\mathbf{h} + i \mathbf{d},\lambda_X(\alpha)).
\end{align*}
By a direct computation, for any $\mathbf{y} = \Re \mathbf{y} + i \im \mathbf{y}$ with $ \Re \mathbf{y} = (\Re y_1,\dots,  \Re y_N), \im \mathbf{y} = (\im y_1,\dots,\im y_N) \in \RR^{N}$, we have 
\begin{align*}
\mathrm{Hess}\left(\Re \left(i \sum_{k=1}^N \mathrm{L}(y_k)\right) \right) 
=\Delta
\end{align*}
where $\Delta$ is the $N \times N$ diagonal matrix with entries
$$ - \mathrm{Im} \bigg( \frac{1}{1+e^{-\Re y_k - i \im y_k}}\bigg).$$ 
The conditions that $\im y_k \in (-\pi,0)$ and $\im y_k \in [-\pi,0]$ respectively imply that the matrix above has negative and non-positive entries for all $\Re \mathbf{y} \in \RR^{N}$ respectively. Since the real part of $\tilde{S}$ is the restriction of the above function on an affine subspace, the real part of $\tilde{S}$ is strictly concave in $(\Re x_1,\dots, \Re x_{N-2n})$ on $\mathscr{U}_{\boldsymbol{x}}$ and concave in $(\Re x_1,\dots, \Re x_{N-2n})$ respectively. The second claim follows from the fact that the real part of a holomorphic function is harmonic. 
\end{proof}

Next, we study some properties of $\tilde{S}$. 
\begin{proposition}\label{Hesstotor}
At the critical point $\mathbf{x^c}$ of $S$ described in Proposition \ref{critThurscorrespondence}, we have 
\begin{align*}
\left|\frac{1}{\sqrt{\pm \det \Hess (\tilde{S}(\mathbf{x^c};\xi))}} \right|
&=  \frac{D_2}{\left|\sqrt{\pm \left(\prod_{i=1}^N z_i^{-f_i''} z_i''^{f_i - 1} \right)\tau(\SS^3 \setminus K, l, \mathbf{z^c}, X)}\right|},
\end{align*}
where $D_2 $ is some non-zero constant independent of $\mathbf{x^c}$.  
In particular, the critical point $\mathbf{x^c}$ is nondegenerate.
\end{proposition}
\begin{proof}
We first compute the 1-loop invariant as follows. Note that 
\begin{align*}
\det (\mathbf{A} \Delta_{z''} + \mathbf{B} \Delta_{z}^{-1})
=&\ \  (\det \mathbf{E})^{-1} \left( \prod_{k=1}^N z_k'' \right) \det \left(\mathbf{EA}  - \mathbf{EB} \Delta_{1-z}^{-1}\right)\\
=&\ \ (\det \mathbf{E})^{-1} 
\left( \prod_{k=1}^N z_k'' \right)
\det
\begin{pmatrix}
    (\mathbf{EA})_{N-2n} - (\mathbf{EB})_{N-2n} \Delta_{1-z}^{-1} \\
    (\mathbf{EA})_{2n}
\end{pmatrix}
\end{align*}
Let $k_1,\dots, k_{2n} \in \{1,\dots, N\}$ be the position of the pivots of $(\mathbf{EA})_{2n}$. For $l=1,\dots,2n$, let $e_{k_l} \in \RR^N$ be the column vector with 1 at the $k_{l}$ entry and 0 elsewhere. Let $H \in M_{N\times 2n}(\RR)$ be the matrix with columns $e_{k_1},\dots,e_{k_{2n}}$. Consider the $N\times N$ matrix $ ( (\mathbf{EB})_{N-2n}^{\!\top}\ | \ H)$. Note that
$$
(\mathbf{EA})_{2n} ( (\mathbf{EB})_{N-2n}^{\!\top} \ |\  H)
= ((\mathbf{EA})_{2n} (\mathbf{EB})_{N-2n}^{\!\top} \ |\ (\mathbf{EA})_{2n}  H)
= (\mathbf{O} \ | \ \mathrm{Id}_{2n} ),
$$
where the last equality follows from (\ref{NZmultizero}) and $\mathbf{O}$ is the $2n \times (N-2n)$ zero matrix. We claim that the matrix $ ( (\mathbf{EB})_{N-2n}^{\!\top}\ | \ H)$ is invertible. To see this, let $v_1,\dots, v_{N-2n}$ be the columns of $(\mathbf{EB})_{N-2n}^{\!\top}$. Suppose there exists $p_1,\dots,p_N \in \CC$ such that
$$
p_1 v_1 + \dots + p_{N-2n} v_{N-2n} + p_{N-2n+1} e_{k_1} + \dots + p_{N} e_{k_{2n}} = 0.
$$
Then we have 
$$
(\mathbf{EA})_{2n}(p_1 v_1 + \dots + p_{N-2n} v_{N-2n} + p_{N-2n+1} e_{k_1} + \dots + p_{N} e_{k_{2n}})
= 
\begin{pmatrix}
    p_{N-2n+1} \\ \vdots \\ p_{N}
\end{pmatrix}
= 0,
$$
which implies that $p_{N-2n+1}=\dots=p_N=0$. 
Since $(\mathbf{EB})_{N-2n}^{\!\top}$ has rank $N-2n$, by the linear independence of $v_1,\dots, v_{N-2n}$ we have 
$$p_1=\dots=p_{N-2n}=0.$$ 
Thus, the matrix $( (\mathbf{EB})_{N-2n}^{\!\top} \ |\  H)$ is invertible. 
Let $G = \det( (\mathbf{EB})_{N-2n}^{\!\top} \ |\  H) \in \QQ \setminus\{0\}$. Then we have
\begin{align*}
    &\ \ \det
\begin{pmatrix}
    (\mathbf{EA})_{N-2n} - (\mathbf{EB})_{N-2n} \Delta_{1-z}^{-1} \\
    (\mathbf{EA})_{2n}
\end{pmatrix} \\
=&\ \ \pm G^{-1} \det \left[
\begin{pmatrix}
    (\mathbf{EA})_{N-2n} - (\mathbf{EB})_{N-2n} \Delta_{1-z}^{-1} \\
    (\mathbf{EA})_{2n}
\end{pmatrix} 
( (\mathbf{EB})_{N-2n}^{\!\top} \ |\  H) \right]
\\
=&\ \ \pm G^{-1} \det
\begin{pmatrix}
    (\mathbf{EA})_{N-2n}(\mathbf{EB})_{N-2n}^{\!\top}  -  (\mathbf{EB})_{N-2n} \Delta_{1-z}^{-1} (\mathbf{EB})_{N-2n}^{\!\top}
\end{pmatrix}.
\end{align*}

Next, by chain rule, the Hessian of $\tilde{S}$ is given by
\begin{align*}
\Hess(\tilde{S}) (\mathbf{x^c};\xi)
=&\  i (\mathbf{EB})_{N-2n}\left( -Q - \frac{\mathcal{E}-\rm{Id}_N}{2}
- 
\Delta_{\frac{e^y}{1+e^y}} \right)
(\mathbf{EB})_{N-2n}^{\!\top} \\
= &\  i (\mathbf{EB})_{N-2n}\left(  -\mathscr{G} 
+
\Delta_{1-\frac{e^y}{1+e^y}} \right)
(\mathbf{EB})_{N-2n}^{\!\top} \\
=& -i(\mathbf{EB})_{N-2n}\left( 
\mathscr{G}
- 
\Delta_{1+e^{y}}^{-1} \right)
(\mathbf{EB})_{N-2n}^{\!\top} .
\end{align*}
Under the correspondence $z = -e^{y}$, by Definition \ref{defgenFAMED} (4) and Remark \ref{rmk4}, we have
$$
\Hess(\tilde{S}) (\mathbf{x^c};\xi)
= -i \left[ (\mathbf{EA})_{N-2n}(\mathbf{EB})_{N-2n}^{\!\top}  - (\mathbf{EB})_{N-2n} 
\Delta_{1-z}^{-1}
(\mathbf{EB})_{N-2n}^{\!\top} \right],
$$
which implies
\begin{align*}
\det (\mathbf{A} \Delta_{z''} + \mathbf{B} \Delta_{z}^{-1})
= \pm  G^{-1} (\det \mathbf{E})^{-1} 
\left( \prod_{k=1}^N z_k'' \right) i^{N-2n} \det \Hess(\tilde{S}(\mathbf{x^c})).
\end{align*}
Besides, recall that the 1-loop invariant is defined by 
$$
\tau(\SS^3 \setminus K, l, \mathbf{z}, X)
=
\frac{\det (\mathbf{A} \Delta_{z''} + \mathbf{B} \Delta_{z}^{-1})}{\prod_{i=1}^N z_i^{-f_i''} z_i''^{f_i} }.
$$
Altogether, we have
$$
\Hess(\tilde{S}) (\mathbf{x^c}; \xi)
= \pm i^{N-2n} G(\det \mathbf{E}) \left( \prod_{k=1}^N z_k'' \right)^{-1} \left( \prod_{i=1}^N z_i^{-f_i''} z_i''^{f_i} \right)
\tau(\SS^3 \setminus K, l, \mathbf{z}, X)
$$
Finally, by \cite[Corollary 1.4]{PW}, the 1-loop invariant $\tau(\SS^3 \setminus K, l, \mathbf{z}, X)$ is non-zero. This completes the proof.
\end{proof}

\subsection{Asymptotic expansion formula of the Teichm\"{u}ller TQFT partition function}\label{sub:proofs:thm15}

Suppose $X$ be a semi-geometric triangulation with shape parameters $\mathbf{z^c}$. Let $\alpha \in \mathcal{A}_X$. By Proposition \ref{critThurscorrespondence}, $\mathbf{z^c}$ corresponds to a critical point $\boldsymbol{x^c}$ of $S(\boldsymbol{x};\lambda_X(\alpha))$. Note that under the bijection $y_k = \Log z_k - i\pi$, the case where $\im y^c_k = -\pi$ corresponds to $z_k \in \RR_{>0}$. Since $X$ is semi-geometric, we know that $z^c_k \neq 1$ and therefore $\Re y^c_k \neq 0$. Thus, by choosing sufficiently small $\delta$, we can assume that whenever $\im y^c_k = -\pi$ for some $k=1,\dots, N$, then $|\Re y^c_k| > \delta $. For the prescribed angle structure $\alpha$, by the exponentially decaying property at infinity, we choose $\kappa > 0$ such that 
\begin{align}\label{smallbdyZ}
\Re S(\boldsymbol{x}; \lambda_X(\alpha)) < \Re S(\boldsymbol{x^c}; \lambda_X(\alpha)) = -\Vol\Big(\SS^3 \setminus K, \mathrm{H}^\CC_{X,l} (\mathbf{z^c}) = i \lambda_X(\alpha)\Big)
\end{align}
for all $\boldsymbol{x} \in (\RR^{N-2n}+ i \mathbf{v}_\alpha) \setminus ( (-\kappa,\kappa)^{N-2n} + i \mathbf{v}_\alpha)$, where the last equality follows from Proposition \ref{dilogVolgen}(1). Let $\boldsymbol{x^c} = (x^c_1,\dots,x^c_{N-2n})$ and 
$$
\tilde L^{\text{top}}
= \prod_{k=1}^{N-2n} ([-\kappa,\kappa] + i \im x^c_k)
.$$

We first construct a multi-contour satisfying in Proposition \ref{constructZcont} and \ref{Zcont} (see Figure \ref{Zcontfigure}) that satisfies the assumptions to applying the saddle point method (Proposition \ref{saddle}).

\begin{figure}
    \centering
    \includegraphics[width=0.53\linewidth]{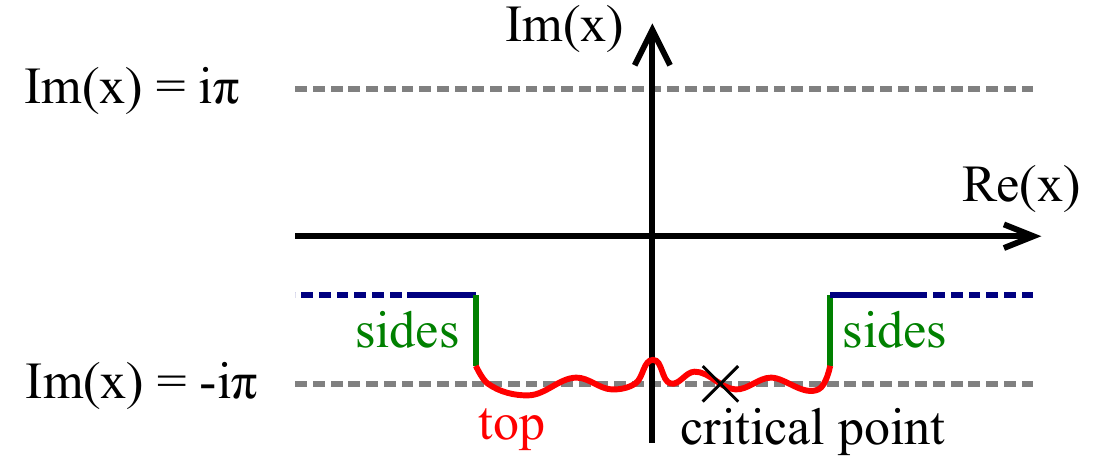}
    \caption{This figure shows a schematic picture of the contours constructed in Proposition \ref{constructZcont}, \ref{Zcont} and used in the proof of Theorem \ref{mainthmZ}. In Proposition \ref{constructZcont}, we first deform the horizontal contour $\tilde L^{\text{top}}$ by pushing it upward around the singularity $-i\pi$ and the boundary points. Then we further deform the contour by following the flow generated by $\mathscr{V}$ that decreases the value of $\Re S(\boldsymbol{x}; \lambda_X(\alpha))$. The critical point is stationary in this process and the resulting red contour is $L^\text{top}$. The blue contour represents $(\RR^{N-2n}+ i \mathbf{v}_\alpha) \setminus ( (-\kappa,\kappa)^{N-2n} + i \mathbf{v}_\alpha)$. We connect the red and blue contours through $L^{\text{sides}}$, which is colored in green in the figure. Note that the union of these three contours is homotopic to $(\RR^{N-2n}+ i \mathbf{v}_\alpha)$, which is the integration contour of the partition function.}
    \label{Zcontfigure}
\end{figure}

\begin{proposition}\label{constructZcont}
$\tilde L^{\text{top}}$ can be deformed along the $(\im x_1,\dots, \im x_{N-2n})$-direction into a new multi-contour $L^{\text{top}}$ such that
\begin{enumerate}
    \item $\boldsymbol{x^c} \in L^{\text{top}}$;
    \item $\Re S(\boldsymbol{x}; \lambda_X(\alpha))$ attains its strict maximum at $\boldsymbol{x^c}$ on $L^{\text{top}}$; 
    \item $\varphi(L^{\text{top}})$ is in the interior of $(\CC \setminus L_\delta)^N$, and
    \item for every $\boldsymbol{x} \in \partial L^{\text{top}}$, $\im \varphi(\boldsymbol{x}) \in (-\pi,0)^{N}$.
\end{enumerate}
\end{proposition}
\begin{proof}
Suppose $\im y^{c}_k \neq -\pi$ and $0$ for all $k=1,\dots, N$. By Proposition \ref{concavSgen}, we can take $L^{\text{top}} = \tilde{L}^{\text{top}}$. Thus, it suffices to consider the case where $\im y^{c}_k = 0$ or $-\pi$ for some $k=1,\dots, N$. By Proposition \ref{concavSgen}, we know that $\Re S(\boldsymbol{x}; \lambda_X(\alpha))$ is concave but not strictly concave on $\tilde{L}^{\text{top}}$ and attains its non-strict maximum at $\boldsymbol{x^c}$. Notice that the function $S(\boldsymbol{x};\lambda_X(\alpha))$ is not holomorphic at the points $\varphi(\boldsymbol{x}) = \mathbf{y} = (y_1,\dots, y_{N})$ with $y_k = -i \pi$ for some $k=1,\dots,N$. To deal with this, let $\alpha \in \overline{\mathcal{A}_X}$ be the extended angle structure corresponding to $\mathbf{z^c}$. Let $(a_1,\dots,a_N)$ be the corresponding $a$-angles. Note that there exists a normal vector at $\alpha$ that points toward $\mathcal{A}_X$. By Lemma \ref{lincorrespond} and \ref{NZpara}, this induces a vector $\mathfrak{v}$ in $\RR^{N-2n}$. By Lemma \ref{singularbehaviour}, we can decrease $\Re S(\boldsymbol{x};\lambda_X(\alpha))$ at those point by pushing $\boldsymbol{x}$ slightly along the direction $\mathfrak{v}$. As a result, we obtain a new contour such that $S(\boldsymbol{x};\lambda_X(\alpha))$ is defined holomorphic on some open set containing the contour (Figure \ref{Zcontfigure}).

Next, for $\boldsymbol{x} \in \partial \tilde L^{\text{top}}$, if $\Re x_{k} = \pm \kappa$ for some $k=1,\dots,N-2n$, by (\ref{smallbdyZ}) and Proposition \ref{concavSgen}, we know that for any $t\in[0,1]$,
\begin{align*}
    \Re S(\boldsymbol{x} - i t ( \im \boldsymbol{x} - \mathbf{v}_\alpha );\lambda_X(\alpha)) 
    &\leq \max \{\Re S(\boldsymbol{x};\lambda_X(\alpha) ) , \Re S(\Re\boldsymbol{x} + i\mathbf{v}_\alpha;\lambda_X(\alpha) ) \} \\
    &\leq \Re S(\boldsymbol{x^c} ;\lambda_X(\alpha)) ,
\end{align*} 
with equality possibly holds only when $t=0$.  
As a result, by pushing the boundary of $\tilde L^{\text{top}}$ slightly along the direction of $i(\im \boldsymbol{x} - \mathbf{v}_\alpha)$, we can make sure that $\Re S(\boldsymbol{x};\lambda_X(\alpha)) < \Re S(\boldsymbol{x^c};\lambda_X(\alpha))$ on the boundary of the new multi-contour and conditions (3) and (4) are satisfied. 

After that, we further deform the interior of the multi-contour as follows. Let $\zeta(\mathbf{y})$ be a smooth bump function that is positive on the interior of $L^{\text{top}}$ and zero outside. 
We claim that under the flow generated by the vector field
$$ \mathscr{V}(\boldsymbol{x}) = \zeta(\boldsymbol{x})\left(-i\frac{\partial \Re S(\boldsymbol{x};\lambda_X(\alpha))}{\partial \im x_1}, - i\frac{\partial\Re S(\boldsymbol{x};\lambda_X(\alpha))}{\partial \im x_2} , \dots, -i\frac{\partial \Re S(\boldsymbol{x};\lambda_X(\alpha))}{\partial \im x_{N-2n}} \right) \in \CC^{N-2n}$$
for a sufficiently small time, 
we can deform $\tilde{L}^{\text{top}}$ to a new contour $L^{\text{top}}$ satisfying properties (1)-(4). First, since $\boldsymbol{x^c}$ is a critical point of $S$, the flow vanishes at that point and therefore $\boldsymbol{x^c} \in L^{\text{top}}$. This proves (1). 

Next, for all $\boldsymbol{x} \in \tilde{L}^{\text{top}} \setminus\{\boldsymbol{x^c}\}$, suppose $\mathscr{V}(\boldsymbol{x})=0$. If $\Re S(\boldsymbol{x};\lambda_X(\alpha)) < \Re S(\boldsymbol{x^c};\lambda_X(\alpha))$, then this inequality still hold under the flow generated by $\mathscr{V}$. If $\Re S(\boldsymbol{x};\lambda_X(\alpha)) = \Re S(\boldsymbol{x^c};\lambda_X(\alpha))$, then $\Re S(\boldsymbol{x};\lambda_X(\alpha))$ attains its maximum at $\boldsymbol{x}$ on $\tilde{L}^{\text{top}}$, which implies that 
$$
\frac{\partial}{\partial \Re x_1} \Re S = \dots = \frac{\partial}{\partial \Re x_{N-2n}} \Re S = 0.
$$
Since $\mathscr{V}(\boldsymbol{x})=0$, $\boldsymbol{x}$ must also be a critical point of $S$. By Proposition \ref{critThurscorrespondence}, under the bijection $y_k = \Log z_k - i\pi$, we get a solution $\mathbf{z}$ of the gluing equation, which induces a representation $\rho:\pi_1(\SS^3\setminus K) \to \mathrm{PSL}(2;\CC)$. However, by Proposition \ref{dilogVolgen} and \cite[Theorem 1.2]{SF}, the representation volume of $\rho$ coincides with the hyperbolic volume of $\SS^3 \setminus K$. By \cite[Theorem 1.4]{SF}, $\rho$ must be discrete and faithful. Since the shape parameters are completely determined by $\rho$, we must have $\boldsymbol{x} = \boldsymbol{x^c}$, which is impossible since $\boldsymbol{x} \in \tilde{L}^{\text{top}} \setminus\{\boldsymbol{x^c}\}$. 

Altogether, for all $\boldsymbol{x} \in L^{\text{top}} \setminus\{\boldsymbol{x^c}\}$, either $\Re S(\boldsymbol{x};\lambda_X(\alpha)) < \Re S(\boldsymbol{x^c};\lambda_X(\alpha))$ or $\mathscr{V}(\boldsymbol{x}) \neq 0$. As a result, by applying the flow generated by $\mathscr{V}$ for a small time, we obtain a deformed multi-contour satisfying (2).

Finally, since the bump function is supported on the interior of $\tilde L^{\text{top}}$, the boundary of $\tilde L^{\text{top}}$ remains unchanged and satisfies (3) and (4). This completes the proof.
\end{proof}

Next, define 
$L^{\text{sides}} = \{ \boldsymbol{x} - i t( \im \boldsymbol{x} - \mathbf{v}_\alpha) \mid \boldsymbol{x} \in \partial L^{\text{top}}, t\in[0,1]\}$ and $L = L^{\text{sides}} \cup L^{\text{top}}$.
\begin{proposition}\label{Zcont}
$L$ is a multi-contour such that $\boldsymbol{x^c} \in L$ and $\Re S$ attains its strict maximum at $\boldsymbol{x^c}$.
\end{proposition}
\begin{proof}
    By Proposition \ref{constructZcont} (1), we have $\boldsymbol{x^c} \in L$. By Proposition \ref{constructZcont} (2), $\Re S(\boldsymbol{x}; \lambda_X(\alpha))$ attains its strict maximum at $\boldsymbol{x^c}$ on $L^{\text{top}}$. 
    On $L^{\text{sides}}$, by Proposition \ref{constructZcont} (4) and Proposition \ref{concavSgen}, we know that 
    \begin{align*}
        &\ \max\{ \Re S(\boldsymbol{x};\lambda_X(\alpha)) \mid \boldsymbol{x} \in L^{\text{sides}} \} \\
        \leq&\  \max\{\Re S(\boldsymbol{x};\lambda_X(\alpha)) , \Re S(\boldsymbol{x} - i(\im \boldsymbol{x} - \mathbf{v}_\alpha) ; \lambda_X(\alpha)) \mid \boldsymbol{x} \in \partial L^{\text{top}} \}.
    \end{align*} 
    By Proposition \ref{constructZcont} (3), we have 
    $$ \max\{\Re S(\boldsymbol{x};\lambda_X(\alpha)) \mid \boldsymbol{x} \in \partial L^{\text{top}}\} < \Re S(\boldsymbol{x^c};\lambda_X(\alpha)).$$
    Besides, by (\ref{smallbdyZ}), 
    $$
\max\{\Re S(\boldsymbol{x} - i(\im \boldsymbol{x} - \mathbf{v}_\alpha), \lambda_X(\alpha)) \mid \boldsymbol{x} \in \partial L^{\text{sides}}\}< \Re S(\boldsymbol{x^c}; \lambda_X(\alpha)) .
$$
This completes the proof.
\end{proof}
Define 
\begin{align}\label{hvalue}
    R(\mathbf{z})
= 
\exp\left(\sum_{k=1}^N \left( -\frac{i(\Log z_k - i\pi)\Log(1-z_k)}{2\pi}  - \frac{i}{\pi} \Li(z_k) \right) \right) \left(\prod_{i=1}^N z_i^{-f_i''} z_i''^{f_i - 1} \right)^{-1}.
\end{align}
We can now prove Theorem \ref{mainthmZ}.
\begin{proof}[Proof of Theorem \ref{mainthmZ}]
Assume $X$ is generalized FAMED with respect to $l$ (see Definition \ref{defgenFAMED}). 
For any $\hbar>0$ and $\alpha = (a_1,b_1,c_1,\dots, a_N,b_N,c_N) \in \mathscr{A}_{X}$, by Propositions \ref{semihbar} and \ref{Tpartiexpress2}, we have
$$
|\mathscr{Z}_{\hbar}(X, \alpha) |
= 
\left|\det(\mathcal{E}_1\mathcal{E}_2) \Big(\frac{1}{2\pi \hbar}\Big)^{N-2n} 
\int_{\RR^{N-2n} + i\mathbf{v}_\alpha} h(\varphi(\boldsymbol{x}))
e^{\frac{1}{2\pi \hbar} S_{\hbar}(\varphi(\boldsymbol{x});\lambda_X(\alpha))} d\boldsymbol{x} \right|,
$$
where 
\begin{align}\label{hdefn}
h( \mathbf{y}) = \exp\left( \sum_{k=1}^N \left( \frac{iy_k}{2\pi} \frac{d}{dy}\mathrm{L}(y_k) - \frac{i}{\pi} \mathrm{L}(y_k) \right) \right) 
\end{align}
and
$$
S_{\hbar}(\boldsymbol{x};\lambda_X(\alpha))
= S(\boldsymbol{x};\lambda_X(\alpha)) + \kappa_\hbar(\boldsymbol{x})\hbar^2
$$
for some holomorphic function $\kappa_{\hbar}$ such that $|\kappa_{\hbar}(\mathrm{w})|$ is bounded above by some constant independent of $\hbar$ on any given compact subset of $\RR^{N-2n}+i\mathbf{v}_\alpha$. 

To prove the first statement, for any angle structure $\alpha \in \mathcal{A}_X$, by Proposition \ref{Tpartiexpress2}, 
$$
|\mathscr{Z}_{\hbar}(X,\alpha)|
= |\mathscr{Z}_{\hbar}(X,\alpha')|
= 
\left|\det(\mathcal{E}_1\mathcal{E}_2) \Big(\frac{1}{2\pi \hbar}\Big)^{N-2n} 
\int_{\RR^{N-2n} + i\mathbf{v}_\alpha} h(\varphi(\boldsymbol{x}))
e^{\frac{1}{2\pi \hbar} S_{\hbar}(\boldsymbol{x};\lambda_X(\alpha))} d\boldsymbol{x} \right|.
$$
Let $\boldsymbol{x}_{\alpha'}$ be the critical point of $\Re S$ in Proposition \ref{critThurscorrespondence} that corresponds to the angle structure $\alpha'$. We let $r_0>0$ and $\Gamma_{\alpha'} = \{\mathbf{y} \in \RR^{N-2n}+i\mathbf{v}_{\alpha'} \mid || \boldsymbol{x} - \boldsymbol{x}_{\alpha'} || \leq r_0\}$ be a ball of real dimension $N-2n$ inside $\RR^{N-2n}+i\mathbf{v}_{\alpha'}$ centered at the critical point $\boldsymbol{x}_{\alpha'}$. We split the integral into two parts: one over the compact part $\Gamma_{\alpha'}$ and another one over the remaining part
$\mathscr{Y}_{\alpha'} \setminus \Gamma_{\alpha'}$
of the multi-contour.
\begin{align*}
&\int_{\RR^{N-2n}+i\mathbf{v}_\alpha} 
h(\varphi(\boldsymbol{x}))
e^{\frac{1}{2\pi \hbar} S_{\hbar}(\boldsymbol{x};\lambda_X(\alpha))} d\boldsymbol{x} \\
=&   
\int_{\Gamma_{\alpha'}} 
h(\varphi(\boldsymbol{x}))
e^{\frac{1}{2\pi \hbar} S_{\hbar}(\boldsymbol{x};\lambda_X(\alpha))} d\boldsymbol{x} 
+ 
\int_{(\RR^{N-2n} \setminus [-\kappa,\kappa]^{N-2n}+i\mathbf{v}_\alpha) \setminus \Gamma_{\alpha'}} 
h(\varphi(\boldsymbol{x}))
e^{\frac{1}{2\pi \hbar} S_{\hbar}(\boldsymbol{x};\lambda_X(\alpha))} d\boldsymbol{x}.
\end{align*}
Similar to \cite[Lemma 7.10]{BAGPN}, there exist constants $A(\alpha'),B(\alpha') > 0$ such that for all $\hbar \in (0,A(\alpha'))$,
$$
\Bigg|
\int_{(\RR^{N-2n} \setminus [-\kappa,\kappa]^{N-2n}+i\mathbf{v}_\alpha) \setminus \Gamma_{\alpha'}} 
h(\varphi(\boldsymbol{x}))
e^{\frac{1}{2\pi \hbar} S_{\hbar}(\boldsymbol{x};\lambda_X(\alpha))} d\boldsymbol{x}
\Bigg|
\leq B(\alpha') e^{\frac{M'}{2\pi \hbar}},
$$
where $M' = \max\{\Re S (\mathbf{y}) \mid  \mathbf{y} \in \partial\Gamma_{\alpha}\}$. By Propositions \ref{dilogVolgen} and  \ref{concavSgen}, we have 
$$M' < -\Vol(\alpha').$$
Besides, on the compact part $\Gamma_{\alpha'}$, by Proposition \ref{dilogVolgen} and \ref{concavSgen}, $\Re S$ attains its maximum at $\boldsymbol{x}_{\alpha'}$ with maximum value $-\Vol(\alpha')$. This proves Theorem \ref{mainthmZ} (1).

To prove the second part, for the prescribed angle structure $\alpha$, by the exponentially decaying property at infinity, we choose $\kappa > 0$ such that (\ref{smallbdy}) holds.
 We write the integral as
\begin{align*}
&\ \ \int_{\RR^{N-2n}+i\mathbf{v}_\alpha} 
h(\varphi(\boldsymbol{x}))
e^{\frac{1}{2\pi \hbar} S_{\hbar}(\boldsymbol{x};\lambda_X(\alpha))} d\boldsymbol{x} \\
=&\ \   
\int_{[-\kappa,\kappa]^{N-2n}+i\mathbf{v}_\alpha} 
h(\varphi(\boldsymbol{x}))
e^{\frac{1}{2\pi \hbar} S_{\hbar}(\boldsymbol{x};\lambda_X(\alpha))} d\boldsymbol{x} \\
& \quad +\int_{\RR^{N-2n} \setminus [-\kappa,\kappa]^{N-2n}+i\mathbf{v}_\alpha} 
h(\varphi(\boldsymbol{x}))
e^{\frac{1}{2\pi \hbar} S_{\hbar}(\boldsymbol{x};\lambda_X(\alpha))} d\boldsymbol{x}.
\end{align*}
For the second integral, similar to \cite[Lemma 7.10]{BAGPN}, there exists constants $A(\alpha),B(\alpha) > 0$ such that for all $\hbar \in (0,A(\alpha'))$,
$$
\Bigg|
\int_{\RR^{N-2n} \setminus [-\kappa,\kappa]^{N-2n}+i\mathbf{v}_\alpha} 
h(\varphi(\boldsymbol{x}))
 e^{\frac{1}{2\pi \hbar} S_\hbar(\boldsymbol{x};\lambda_X(\alpha))} d\boldsymbol{x} 
\Bigg|
\leq B e^{\frac{M'}{2\pi \hbar}},
$$
where $M' = \max\left\{\Re S (\boldsymbol{x};\lambda_X(\alpha)) \mid  \boldsymbol{x} \in \partial(\RR^{N-2n} \setminus [-\kappa,\kappa]^{N-2n}+i\mathbf{v}_\alpha)\right\}$. By (\ref{smallbdy}), we have $M' < S( \boldsymbol{x^c};\lambda_X(\alpha))$. Next, for the first integral, by deformation of multi-contour we have
$$
\int_{[-\kappa,\kappa]^{N-2n}+i\mathbf{v}_\alpha} 
h(\varphi(\boldsymbol{x}))e^{\frac{1}{2\pi \hbar} S_{\hbar}(\boldsymbol{x};\lambda_X(\alpha))} d\boldsymbol{x}
= \int_{L} 
h(\varphi(\boldsymbol{x}))
e^{\frac{1}{2\pi \hbar} S_{\hbar}(\boldsymbol{x};\lambda_X(\alpha))} d\boldsymbol{x},
$$
where $L$ is the multi-contour in Proposition \ref{Zcont}. By applying Proposition \ref{saddle}, the desired result follows from Proposition \ref{dilogVolgen} and \ref{Hesstotor}.
\end{proof}

\section{Asymptotics of Jones functions}
\subsection{Existence of Jones functions} 
In this section, we assume that $X$ is generalized FAMED with respect to $(l,m)$ (see Definition \ref{defgenFAMED5}). By Lemma \ref{NZpara} and Definition \ref{defgenFAMED5}(2), the affine subspace can be parametrized in a way that is independent of the choices of angle structure. Together with Proposition \ref{Tpartiexpress2}, we see that in the partition function, the information about the angle structure $\alpha\in \mathscr{A}_X$ is contained in the last entry of the expression
\begin{align*}
-\boldsymbol{x}^{\!\top}
\left(
\mathbf{E_{N-2n}'}\boldsymbol{u}
\right)
= -\left( \left(\mathbf{E'_{N-2n}}\right)^{\!\top}
\boldsymbol{x} \right) \cdot \boldsymbol{u},
\end{align*}
which can be written in the form
$$ \frac{\tilde{\mathrm{w}} \lambda_X(\alpha)}{2},$$
where $\tilde{\mathrm{w}}$ is a linear combination of $x$'s given by the last entry of
$$- 2 \left(\mathbf{E'_{N-2n}}\right)^{\!\top}
\boldsymbol{x} . $$
Throughout this section, up to renumbering, assume that the pivot positions of $\mathbf{(EB)}^{\!\top}_{N-2n}$ are $1,2,$ $\dots, N-2n$. 
Lemma \ref{tildexinter} below provides an interpretation of the quantity $\mathrm{\tilde w}$. 
\begin{lemma}\label{tildexinter}
Suppose $X$ is generalized FAMED with respect to $(l,m)$.  Under the correspondence in Proposition \ref{critThurscorrespondence}, at the critical point of $\tilde S$, we have
$$ \mathrm{\tilde w} = w_m + C,
$$
where $w_m$ is the logarithmic holonomy of $m$ and $C$ is a constant independent of $\alpha$. 
\end{lemma}
\begin{proof}
Let $\alpha^0=(a_1^0,\dots,c_N^0) \in \mathcal{A}_X$ be the fixed angle structure we picked in Lemma \ref{NZpara}. Under the bijection $y_k = \Log z_k - i\pi$, we have $\Log z_{k} = x_k + i a_{k}^0$. Thus, by Lemma \ref{pretildexinter}, we have
\begin{align*}
\mathrm{\tilde w} 
&= C_1 x_1 + \dots + C_{N-2n} x_{N-2n} \\
&= C_1 \Log z_{1} + \dots + C_{N-2n} \Log z_{N-2n} - i(C_1a_{1}^0 + \dots + C_{N-2n} a_{N-2n}^0) \\
&= w_{m} - i \mu_X(\alpha^0).
\end{align*}
Let $C = - i\mu_X(\alpha^0)$. Note that $C$ depends only on the fixed angle structure $\alpha^0$ but not on $\alpha$. This completes the proof.
\end{proof}

From Lemma \ref{tildexinter}, if we define a new quantity $\mathrm{w} = \mathrm{\tilde w} - C$, then $\mathrm{w}$ can be interpreted as the holonomy of the meridian.

\begin{proposition}\label{pfexistJ} Suppose $X$ is a generalized FAMED with respect to $(l,m)$. Then there exists a Jones function $J_X\colon \R_{>0} \times \mathcal{W} \to \C$ (where $\mathcal{W}\subset \CC$ is an open horizontal band) that is independent of $\alpha\in\mathscr{A}_X$ such that 
\begin{align*}
&|\mathscr{Z}_{\hbar}(X, \alpha)| 
= \left| 
\int_{\RR + i \mu_X(\alpha) } 
\mathfrak{J}_X(\hbar,\mathrm{w})
e^{\frac{\mathrm{w} \lambda_X(\alpha)}{4\pi \hbar}} d\mathrm{w} \right|,
\end{align*}
where $\lambda_X(\alpha)$ and $\mu_X(\alpha)$ are angular holonomies of $l$ and $m$ respectively.
\end{proposition}
\begin{proof} 
From Proposition \ref{Tpartiexpress2}, we can write
\begin{align*}
\ \ \left|\mathcal{Z}_{\hbar}(X,\alpha)\right| = &\left(\frac{1}{2\pi \sqrt{\hbar}}\right)^{N-2n} D_1\\
\ \  \Bigg|\int_{\boldsymbol{x} \in \RR^{N-2n} + i\mathbf{v}_\alpha } 
&\left[\exp\left(\frac{1}{2\pi \hbar} \left( - \frac{i}{2} \varphi(\boldsymbol{x})^{\!\top} Q \varphi(\boldsymbol{x})  - \frac{i}{2} \sum_{k=1}^N \left(\frac{\varepsilon(T_k)-1}{2}\right) \varphi_k(\boldsymbol{x})^2 \right) \right) \right.  \\
& \ \exp\left(\frac{1}{2\pi\hbar}(- \boldsymbol{x}^{\!\top}  (\mathbf{E_{N-2n}'\boldsymbol \nu - 
\mathbf{(EB)}_{N-2n}\mathscr{G}
\boldsymbol{\pi}) )}\right) \\
& \left.\left. \ \left(\prod_{k=1}^N \Phi_\B\left(\frac{\varphi_k(\boldsymbol{x})}{2\pi \sqrt{\hbar}}\right)\right)^{-1} \right]  
e^{\frac{\mathrm{\tilde w}\lambda_X(\alpha)}{4\pi \hbar} }
d\boldsymbol{x}\right| 
\end{align*}
Write $\mathrm{w} = \sum_{k=1}^{N-2n} C_k x_k - C$, where $C\in i\RR$ is the constant in Lemma \ref{tildexinter}. Without loss of generality, assume that $C_1 \neq 0$. Consider the affine isomorphism $\mathscr{L}: \CC^{N-2n} \to \CC^{N-2n}$ with determinant $C_1 \neq 0$ defined by sending 
$$(x_1,x_2,\dots, x_{N-2n}) \mapsto (\mathrm{w}, x_2,\dots, x_{N-2n}) = \left( \sum_{k=1}^{N-2n} C_k x_k -C, x_2,\dots, x_{N-2n} \right).$$
Note that from Lemma \ref{tildexinter}, since $C = -i\sum_{k=1}^{N-2n} C_k a_k^0$, we have
$$
\sum_{k=1}^{N-2n} i C_k (a_k - a_k^0) - C
= i\sum_{k=1}^{N-2n} C_k a_k 
= i\mu_X(\alpha),
$$
where the last equality follows from Definition \ref{defgenFAMED5}(3). Thus, we have
\begin{align}\label{ZtoJconv}
&|\mathscr{Z}_{\hbar}(X, \alpha)| 
= \left| 
\int_{\RR + i\mu(\alpha) } 
\mathfrak{J}_X(\mathrm{w},\hbar)
e^{\frac{\mathrm{w} \lambda_X(\alpha)}{4\pi \hbar}} d\mathrm{w} \right| ,
\end{align}
where
\begin{align}\label{ZtoJconv2}
\mathfrak{J}_X(\mathrm{w},\hbar)
=&\  \left(\frac{1}{2\pi \sqrt{\hbar}}\right)^{N-2n} D_1 \notag\\
\Bigg|\int_{\mathbf{\hat{x}} \in \RR^{N-2n-1} + i\mathbf{\hat{v}}_\alpha } 
 &\left[\exp\left(\frac{1}{2\pi \hbar} \left( - \frac{i}{2} \varphi(\boldsymbol{x})^{\!\top} Q \varphi(\boldsymbol{x})  - \frac{i}{2} \sum_{k=1}^N \left(\frac{\varepsilon(T_k)-1}{2}\right) \varphi_k(\boldsymbol{x})^2 \right) \right) \right.   \notag  \\
& \exp\left(\frac{1}{2\pi\hbar}(- \boldsymbol{x}^{\!\top}  (\mathbf{E_{N-2n}'\boldsymbol \nu - 
\mathbf{(EB)}_{N-2n}\mathscr{G}
\boldsymbol{\pi}) )}\right) \notag  \\
& \left. \left(\prod_{k=1}^N \Phi_\B\left(\frac{\varphi_k(\boldsymbol{x})}{2\pi \sqrt{\hbar}}\right)\right)^{-1} \right]  d\boldsymbol{x} 
\end{align}
with 
$
\mathbf{\hat{x}} = (x_2,\dots, x_{N-2n}), \mathbf{\hat{v}}_\alpha = ( a_2 - a_2^0, \dots, a_{N-2n} - a_{N-2n}^0)
$
and
$$ x_1 = \frac{\mathrm{w} -  \sum_{k=2}^{N-2n} C_k x_k + C }{C_1}. $$ 
Note that (\ref{ZtoJconv}) holds for any $\alpha \in \mathcal{A}_X$. In particular, the integrand in (\ref{ZtoJconv2}) decays exponentially at infinity and the integral converges absolutely for any $\alpha \in \mathcal{A}_X$. Similar to the proof of Proposition \ref{Tpartiexpress2}, we can deform the integration multi-contour to any $\RR^{N-2n-1} + i\mathbf{\hat{v}}_{\alpha'}  $ for any $\alpha' \in \mathcal{A}_X$ without changing the integral. This shows that $\mathfrak{J}_X$ is independent on the choice of the angle structure. 
\end{proof} 

Lemma \ref{pretildexinter} below is a supporting evidence for the expectation discussed in Remark \ref{12implies3}. We will not use this lemma in the rest of this paper. Readers who are not interested in the lemma below can skip to the next subsection.
\begin{lemma}\label{pretildexinter} Let $\mathcal{V}_X$ be the gluing variety of $X$ (see Section \ref{NZD}). Let $X$ be an ideal triangulation satisfying Definition \ref{defgenFAMED5}(1) and (2). Then there exists $n_1,n_2\in \QQ$ such that for every $\mathbf{z} \in \mathcal{V}_X$, we have 
    $$ C_1\Log z_1 + \dots + C_{N-2n} \Log z_{N-2n} = \mathrm{H}^\CC_{X,m}(\mathbf{z}) + n_1 \mathrm{H}^\CC_{X,l}(\mathbf{z}) + n_2 i \pi.$$
\end{lemma}
\begin{proof}
    Suppose $\mathrm{H}^\CC_{X,m}(\mathbf{z}) = P_1 \Log z_1 + \dots + P_N \Log z_{N-2n} + Q_1 \Log z_1'' + \dots + Q_N \Log z_N''$ for some $P_1,\dots, P_N, Q_1, \dots, Q_N \in \ZZ$. Recall that each row of the matrix $(\mathbf{A} | \mathbf{B})$ corresponds a curve on $\partial(\SS^3 \setminus \nu(K))$, which is either the longitude of $K$ or a loop traveling around an edge of the triangulation. It is known that with respect to the Neumann-Zagier symplectic form $\omega$, for any row $r = (r_A | r_B)$ of the matrix $(\mathbf{A} | \mathbf{B})$, where $r_A, r_B \in M_{1,N}(\ZZ)$, we have 
    $$\omega( (P_1,\dots, P_N, Q_1,\dots Q_N), r) =   (P_1,\dots, P_N) \cdot r_B -  (Q_1,\dots, Q_N) \cdot r_A = 2i(m,r),$$ 
    where $\cdot$ denotes the usual dot product of $\CC^N$ and $i(m,r)$ is the algebraic intersection number of the meridian $m$ and the curve corresponding to the row $r$ (see \cite{NZ} and \cite[Proposition 15.2.14]{BM}). In terms of matrix multiplication, this can be written as
    $$ (\mathbf{A} | \mathbf{B})
    (-Q_1 , \dots , -Q_N , P_1 , \dots , P_N)^{\!\top}
    = (0,\dots, 0 , -2)^{\!\top}. $$
    Multiplying both sides by $\mathbf{E}$, we have
    \begin{align*}
    \begin{pmatrix}
    (\mathbf{EA})_{N-2n} & (\mathbf{EB})_{N-2n}   \\
   (\mathbf{EA})_{2n} & \mathbf{O} 
\end{pmatrix} (-Q_1 , \dots , -Q_N , P_1 , \dots , P_N)^{\!\top}
    = \begin{pmatrix}
        \mathbf{E}_{N-2n} \\ \mathbf{E}_{2n}
    \end{pmatrix}
    (0,\dots,0,2)^{\!\top}.
    \end{align*}
    By Definition \ref{defgenFAMED5}(2), we have
    \begin{align}\label{NZequlemma}
        \begin{pmatrix}
    (\mathbf{EA})_{N-2n} & (\mathbf{EB})_{N-2n}   \\
   (\mathbf{EA})_{2n} & \mathbf{O} 
\end{pmatrix} (-Q_1 , \dots , -Q_N , P_1 , \dots , P_N)^{\!\top}
= (C_1,\dots, C_{N-2n},0,\dots, 0)^{\!\top},
    \end{align}
    where $(C_1,\dots, C_{N-2n})^{\!\top}$ is the last column of $\mathbf{E}_{N-2n}$. By considering the last $2n$ equations in (\ref{NZequlemma}), we have
    $$ (\mathbf{EA})_{2n} (Q_1,\dots, Q_N)^{\!\top} = (0,\dots,0).$$
    By Proposition \ref{NZpara}, we have $(Q_1,\dots, Q_N)^{\!\top} \in \ker (\mathbf{EA})_{2n} = \im ((\mathbf{EB})_{N-2n}^{\!\top})$, which implies that $(Q_1,\dots, Q_N)^{\!\top}$ is in the row space of $(\mathbf{EB})_{N-2n}$. As a result, by adding a linear combination of rows of $(\mathbf{A}|\mathbf{B})$ that in particular contributes to the constants $n_1,n_2$ in the Lemma, we can assume that $E_1=\dots = E_N  = 0$. Besides, by Proposition \ref{preNZpara}, since the pivot positions of $\mathbf{(EB)}^{\!\top}_{N-2n}$ are $1,\dots,N-2n$, we know that $\Log z_{N-2n+1}, \dots, \Log z_{N}$ can be expressed in terms of $\Log z_1, \dots, \Log z_{N-2n}$. Thus, by redefining $n_2$ if necessary, we can assume $D_{N-2n+1} = \dots = D_N = 0$. By considering the first $N-2n$ equations in (\ref{NZequlemma}), we have 
    \begin{align*}
        (\mathbf{EB})_{N-2n} (P_1,\dots, P_{N-2n}, 0,\dots, 0)^{\!\top}
        = (C_1,\dots, C_{N-2n})^{\!\top}.
    \end{align*}
    Since $\mathbf{EB} = \begin{pmatrix}
            \mathrm{Id}_{N-2n} | \star
        \end{pmatrix}$ for some matrix $\star \in M_{N-2n \times N}(\QQ)$, we have $P_k = C_k$ for $k=1,\dots, N-2n$. This completes the proof. 
\end{proof}

\subsection{Potential function and its properties}
Note that under the affine isomorphism $\mathscr{L}$, the potential function can be written as
\begin{align*}
\tilde{S}^J(\mathrm{w}, x_2,\dots, x_{N-2n} ; \xi) 
&= \tilde{S}(\mathscr{L}^{-1}(\mathrm{w}, x_2,\dots, x_{N-2n}) ; \xi )
= J(\mathrm{w},x_2,\dots, x_{N-2n}) - \frac{i\xi \mathrm{w} }{2}, 
\end{align*}
where
\begin{align*}
    J(\mathrm{w},x_2,\dots, x_{N-2n}) 
    =&\ -\frac{i}{2} \varphi(\boldsymbol{x})^{\!\top}
Q \varphi(\boldsymbol{x})
 - \frac{i}{2}\sum_{k=1}^N \left(\frac{\varepsilon(T_k)-1}{2}\right) \varphi_k(\boldsymbol{x})^2
 + i \sum_{k=1}^N \mathrm{L}(\varphi_k(\boldsymbol{x})) \\
&\ - \boldsymbol{x}^{\!\top}  (\mathbf{E}_{N-2n}'\boldsymbol \nu- 
\mathbf{(EB)}_{N-2n}\mathscr{G}
\boldsymbol{\pi}),
\end{align*}
is independent on the angle structure $\alpha$.

\begin{lemma}\label{xbiholo}
The map that sends $w_l$ to $w_m$ is a local biholomorphism at $0$. Moreover, this map sends $0$ to $0$.
\end{lemma}
\begin{proof}
By \cite[Lemma 4.1 (a)]{NZ} at the complete hyperbolic structure, we have
$$
\frac{\partial w_m}{\partial w_l} \neq 0 .$$ 
The first claim follows from the inverse function theorem. The second claim follows from the fact that at the complete hyperbolic structure, we have $w_l = w_m = 0$.
\end{proof}

We study the properties of the potential function $J$ for $\mathrm{w}$ sufficiently close to $0$. Recall that $\mathscr{U_{\mathbf{y}}}$ and $\overline{\mathscr{U_{\mathbf{y}}}}$ are the products of horizontal bands defined by
$$
\mathscr{U}_\mathbf{y} = \prod_{k=1}^N (\RR + i (-\pi,0)) \text{ and } \overline{\mathscr{U_\mathbf{y}}} = \prod_{k=1}^N (\RR + i [-\pi,0]) .
$$
\begin{proposition}\label{concavSxgen}
Let $\mathrm{w} \in \CC$. Given $\mathbf{\hat{v}} \in \RR^{N-2n-1}$, consider the horizontal plane $\RR^{N-2n-1}+i\mathbf{\hat{v}}$. If 
$$\varphi\big(\mathscr{L}^{-1}( \{\mathrm{w}\} \times ( \RR^{N-2n-1}+i\mathbf{\hat{v}}))\big) \subset \mathscr{U}_{\mathbf{y}} \qquad (\text{resp. } \varphi(\mathscr{L}^{-1}( \{\mathrm{w}\} \times (\RR^{N-2n-1}+i\mathbf{\hat{v}}))\big) \subset \overline{\mathscr{U}_{\mathbf{y}}}), $$ 
then the real part of $J(\mathrm{w},x_2,\dots, x_{N-2n}) $ is strictly concave (resp. concave) in the variables \linebreak $(\Re x_2, \dots, \Re x_{N-2n})$ and strictly convex (resp. convex) in the variables $(\mathrm{Im}\text{ } x_2, \dots, \mathrm{Im}\text{ } x_{N-2n})$.
\end{proposition}
\begin{proof}
The proof is similar to that of Proposition \ref{concavSgen}.
\end{proof}

From Lemma \ref{xbiholo}, given $\mathrm{w}= w_m$ sufficiently close to $0$, we have a corresponding value $w_l = w_l(\mathrm{w})$. Let 
$$\boldsymbol{x^c}(\mathrm{w}) = (x_{1}^c(\mathrm{w}), x_{2}^c(\mathrm{w}), \dots, x_{N-2n}^c(\mathrm{w}))$$ 
be the corresponding critical point of $\tilde{S}(x_1,\dots, x_{N-2n} ; w_l(\mathrm{w}) )$.

\begin{proposition}\label{JandNZ}
For any $\mathrm{w}$ sufficiently close to 0,  the point $(x^c_{2}(\mathrm{w}), \dots, x^c_{N-2n}(\mathrm{w}))$ is a critical point of the holomorphic function $J(\mathrm{w}, x_2,\dots, x_{N-2n})$ with respect to $x_2,\dots, x_{N-2n}$. Furthermore, 
the critical value $J(\mathrm{w}, x^c_{2}(\mathrm{w}),\dots, x^c_{N-2n}(\mathrm{w}))$ satisfies
\begin{align*}
\frac{\partial}{\partial \mathrm{w}} J(\mathrm{w}, x_{2}^c(\mathrm{w}),\dots, x_{N-2n}^c(\mathrm{w})) = \frac{iw_l}{2} \ \text{,} \ \Re J(0, x_{2}^c(0),\dots, x_{N-2n}^c(0)) = - \Vol(\SS^3 \setminus K).
\end{align*}
In particular, we have $ J(\mathrm{w}, x^c_{2}(\mathrm{w}),\dots, x^c_{N-2n}(\mathrm{w})) = i\phi_{m,l}(\mathrm{w}) + C$ for some imaginary constant $C\in \CC$, where $\phi_{m,l}$ is the Neumann-Zagier potential function with respect to $(m,l)$ (see Section \ref{NZpotentintro}).
\end{proposition}
\begin{proof}
Note that since $S^J=S \circ \mathscr{L}^{-1}$ and $\mathscr{L}$ is an affine isomorphism,
$$\nabla \tilde{S}(x_1,\dots, x_{N-2n}) = ( \tilde{S}_{x_1} ,\dots, \tilde{S}_{x_{N-2n}}) = 0$$
 if and only if
\begin{align}\label{JSiff}
    \nabla \tilde{S}^J(\mathrm{w}, x_2, \dots, x_{N-2n} ;\xi) = ( J_{\mathrm{w}} - i\xi/2, J_{x_2},\dots, J_{x_{N-2n}}) = 0.
\end{align} 
This implies that 
$$ J_{x_2} (\mathrm{w}, x_{2}^c(\mathrm{w}), \dots, x_{N-2n}^c(\mathrm{w})) = \dots= 
J_{x_{N-2n}}(\mathrm{w}, x_{2}^c(\mathrm{w}), \dots, x_{N-2n}^c(\mathrm{w})) = 0$$
for all $\mathrm{w}$. Moreover, 
$$
\frac{\partial}{\partial \mathrm{w}} \Big( J(\mathrm{w}, x_{2}^c(\mathrm{w}),\dots, x_{N-2n}^c(\mathrm{w}) \Big)
= J_{\mathrm{w}} + \sum_{k=2}^{N-2n} J_{x_k} \cdot \frac{d x_{k}^c}{d\mathrm{w}}(\mathrm{w})  
= J_{\mathrm{w}} ,
$$
where the last equality follows from the fact that $(x_{2}^c(\mathrm{w}),\dots, x_{N-2n}^c(\mathrm{w}))$ is the critical point of $J$ with respect to $x_2,\dots, x_{N-2n}$. As a result, by (\ref{JSiff}), we have
$$
\frac{\partial}{\partial \mathrm{w}} \Big(J(\mathrm{w}, x_{2}^c(\mathrm{w}),\dots, x_{N-2n}^c(\mathrm{w})) \Big)
= \frac{i\xi}{2} = \frac{i w_l}{2}.
$$
When $w_l=0$, by Lemma \ref{xbiholo}, we have $\mathrm{w} = 0$. Moreover, when $\mathrm{w}=0$, we have 
\begin{align*}
\tilde{S}^J\left(0 , x_{2}^c(0),\dots, x_{N-2n}^c(0); 0 \right) 
&= J\left(0 , x_{2}^c(0),\dots, x_{N-2n}^c(0)\right). 
\end{align*}
Thus, by Proposition \ref{dilogVolgen}, we have
\begin{align}\label{ReJgiveVol}
\mathrm{Re} \left(J\left(0 , x_{2}^c(0),\dots, x_{N-2n}^c(0)\right)  \right)
= -\Vol(\SS^3 \setminus K). 
\end{align}
Finally, by (\ref{NZprop}), since 
$$
\frac{\partial}{\partial \mathrm{w}} \Big(J(\mathrm{w}, x_{2}^c(\mathrm{w}),\dots, x_{N-2n}^c(\mathrm{w})) - i\phi_{m,l}(\mathrm{w}) \Big)
= \frac{i\mathrm{H}(l)}{2} - \frac{i\mathrm{H}(l)}{2} = 0,
$$
we have $J(\mathrm{w}, x_{2}^c(\mathrm{w}),\dots, x_{N-2n}^c(\mathrm{w})) - i\phi_{m,l}(\mathrm{w}) = C$ for some constant $C$. Note that when $\mathrm{w}=0$, by (\ref{NZprop}) and (\ref{ReJgiveVol}), we have
$\mathrm{Re}(C) = 0$. Thus, $C$ is an imaginary constant.
\end{proof}

The next proposition relates the determinant of the Hessian of $\tilde{S}$ to that of $J$.
\begin{proposition}\label{1loopJ} For $\mathrm{w} = w_m$ sufficiently small, 
we have
$$ \det(\Hess \tilde{S}) (x_1^c(\mathrm{w}), \dots, x_{N-2n}^c(\mathrm{w}); w_l(\mathrm{w})) = \frac{i}{2C_1^2 } \frac{\partial w_l}{\partial w_m}  \det (\Hess J)(\mathrm{w}, x_2^c(\mathrm{w}), \dots, x_{N-2n}^c(\mathrm{w})) . $$
\end{proposition}
\begin{proof}
Recall that $\tilde{S}^J(x,x_2\dots, x_{N-2n};\xi) = \tilde{S}(\mathscr{L}^{-1}(x,x_2\dots, x_{N-2n}))$, where $\mathscr{L}: \CC^{N-2n} \to \CC^{N-2n}$ is the affine isomorphism with determinant $C_1 \neq 0$ defined by  
$$(x_1,x_2,\dots, x_{N-2n}) \mapsto (\mathrm{w}, x_2,\dots, x_{N-2n}) = \left( \sum_{k=1}^{N-2n} C_k x_k -C, x_2,\dots, x_{N-2n} \right).$$
In particular, we have
\begin{align*}
&\ (\Hess_{\mathrm{w},x_2,\dots,x_{N-2n}} \tilde{S}^J) (\mathrm{w}, x_2(\mathrm{w}), \dots, x_{N-2n}(\mathrm{w})))  \\
=&\  \mathscr{L}^{\!\top} (\Hess_{x_1,\dots,x_{N-2n}} \tilde{S}) (x_1(\mathrm{w}), x_2(\mathrm{w}), \dots, x_{N-2n}(\mathrm{w})) \mathscr{L}
\end{align*} 
and 
\begin{align*}
&\ \det(\Hess_{\mathrm{w},x_2,\dots,x_{N-2n}} \tilde{S}^J) (\mathrm{w}, x_2^c(\mathrm{w}), \dots, x_{N-2n}^c(\mathrm{w})))  \\
=&\ \det(\mathscr{L}^{\!\top}) \det[(\Hess_{x_1,\dots,x_{N-2n}} \tilde{S}) (x_1^c(\mathrm{w}), x_2^c(\mathrm{w}), \dots, x_{N-2n}^c(\mathrm{w}))] \det(\mathscr{L}) \\
=&\ C_1^2\det[(\Hess_{x_1,\dots,x_{N-2n}} \tilde{S}) (x_1^c(\mathrm{w}), x_2^c(\mathrm{w}), \dots, x_{N-2n}^c(\mathrm{w}))] .
\end{align*} 
Thus, 
\begin{align*}
&\ \det(\Hess_{x_1,\dots,x_{N-2n}} \tilde{S}) (x_1^c(\mathrm{w}), x_2^c(\mathrm{w}), \dots, x_{N-2n}^c(\mathrm{w})) \\
=& \ \frac{1}{C_1^2}\det(\Hess_{\mathrm{w},x_2,\dots,x_{N-2n}} \tilde{S}^J) (\mathrm{w}, x_2^c(\mathrm{w}), \dots, x_{N-2n}^c(\mathrm{w}))  .
\end{align*}
and it suffices to compute $\det(\Hess_{\mathrm{w},x_2,\dots,x_{N-2n}} \tilde{S}^J) (\mathrm{w}, x_2^c(\mathrm{w}), \dots, x_{N-2n}^c(\mathrm{w}))$. 
We claim that 
\begin{enumerate}
\item[(1)] for $i,j \in \{2,\dots, N-2n\}$, 
\begin{align*}
\frac{\partial^2 \tilde{S}^J}{\partial x_i x_j} (\mathrm{w}, x_2^c(\mathrm{w}), \dots, x_{N-2n}^c(\mathrm{w})) = \frac{\partial^2 J}{\partial x_i \partial x_j} (\mathrm{w}, x_2^c(\mathrm{w}), \dots, x_{N-2n}^c(\mathrm{w}))),
\end{align*}
\item[(2)] for $i \in \{2,\dots, N-2n\}$, 
\begin{align*}
&\ \ \frac{\partial^2 \tilde{S}^J}{\partial \mathrm{w} \partial x_i} (\mathrm{w}, x_2^c(\mathrm{w}), \dots, x_{N-2n}^c(\mathrm{w})) \\
=&\ \  - 
\sum_{k=2}^{N-2n}\frac{\partial^2 J}{ \partial x_k\partial x_i}(x_2^c(\mathrm{w}), \dots, x_{N-2n}^c(\mathrm{w})) \cdot \Bigg( \frac{\partial x_k^c}{\partial \mathrm{w}} \Bigg),
\end{align*} 
\item[(3)] we have
\begin{align*}
&\ \ \frac{\partial^2 \tilde{S}^J}{\partial \mathrm{w}^2}(\mathrm{w}, x_2^c(\mathrm{w}), \dots, x_{N-2n}^c(\mathrm{w}))  \\
=&\ \  \frac{i}{2} \frac{\partial w_l}{\partial \mathrm{w}} + \sum_{k_1,k_2=2}^{N-2n} \Bigg(\frac{\partial^2 J}{\partial x_{k_1} \partial x_{k_2}} (\mathrm{w}, x_2^c(\mathrm{w}), \dots, x_{N-2n}^c(\mathrm{w})) \cdot \bigg( \frac{\partial x_{k_1}^c}{\partial \mathrm{w}} \bigg)  \bigg( \frac{\partial x_{k_2}^c}{\partial \mathrm{w}} \bigg) \Bigg).
\end{align*}
\end{enumerate}
Assuming these claims, if we write 
$$k_i = - \frac{\partial x_i^c (\mathrm{w})}{\partial \mathrm{w}} $$
 for $i=2,\dots, N-2n$ and $\tilde D = (\Hess J)(\mathrm{w}, x_2(\mathrm{w}), \dots, x_{N-2n}(\mathrm{w})) $,  then we have
$$
\det(\Hess_{\mathrm{w},x_2,\dots,x_{N-2n}} S) (\mathrm{w}, x_2^c(\mathrm{w}), \dots, x_{N-2n}^c(\mathrm{w})))
= \det (P \cdot D \cdot P^{\!\top}),
$$
where 
\begin{align*}
    P=
\begin{pNiceArray}{c | c c c}
1& k_2 & \dots &   k_{N-2n}  \\ \hline
  0  &  \Block{3-3}{\mathrm{Id}_{N-2n-1}}  \\
  \vdots & \\
  0 & \\
\end{pNiceArray}
\quad \text{ and } \quad
D=
\begin{pNiceArray}{c | c c c}
\frac{i}{2} \frac{\partial w_l}{\partial w_m} & 0 & \dots &  0 \\ \hline
  0  &  \Block{3-3}{\tilde D} \\
  \vdots & \\
  0 & \\
\end{pNiceArray}.
\end{align*}

This implies the desired result. Thus, it suffices to prove Claims (1)-(3). Claim (1) follows from the definitions of $\tilde{S}$. Next, for $i\in \{2,\dots, N-2n\}$, since
$$  \frac{\partial \tilde{S}^J}{ \partial x_i}(\mathrm{w}, x_2^c(\mathrm{w}),\dots, x_{N-2n}^c(\mathrm{w})) = 0,$$
by differentiating both sides with respect to $\mathrm{w}$, we have
$$
\frac{\partial^2 \tilde{S}^J}{\partial \mathrm{w} \partial x_i}(\mathrm{w}, x_2^c(\mathrm{w}),\dots, x_{N-2n}(\mathrm{w})) + 
\sum_{k=2}^{N-2n}\frac{\partial^2 \tilde{S}^J}{ \partial x_k\partial x_i}(\mathrm{w}, x_2^c(\mathrm{w}),\dots, x_{N-2n}^c(\mathrm{w})) \cdot \Bigg( \frac{\partial x_k^c}{\partial \mathrm{w}} \Bigg) = 0
$$
and Claim (2) follows. Finally, by Proposition \ref{JandNZ}, since 
$$\frac{\partial}{\partial \mathrm{w}} J(\mathrm{w}, x_2^c(\mathrm{w}),\dots, x_{N-2n}^c(\boldsymbol{x})) = \frac{i w_l}{2},$$ 
by differentiating both sides with respect to $\mathrm{w}$, we have
$$
\frac{\partial^2 J}{\partial \mathrm{w}^2}(\mathrm{w}, x_2^c(\mathrm{w}), \dots, x_{N-2n}^c(\mathrm{w}))
+ \sum_{k=2}^{N-2n} \frac{\partial^2 J}{\partial \mathrm{w} \partial x_k}(\mathrm{w}, x_2^c(\mathrm{w}), \dots, x_{N-2n}^c(\mathrm{w}))
\cdot \Bigg( \frac{\partial x_k^c}{\partial \mathrm{w}} \Bigg) 
= \frac{i}{2} \frac{\partial w_l}{\partial \mathrm{w}}.
$$
Observe that
$$
\frac{\partial^2 \tilde{S}^J}{\partial \mathrm{w}^2}(\mathrm{w}, x_2^c(\mathrm{w}), \dots, x_{N-2n}^c(\mathrm{w}))= \frac{\partial^2 J}{\partial \mathrm{w}^2}(\mathrm{w}, x_2^c(\mathrm{w}), \dots, x_{N-2n}^c(\mathrm{w}))
$$
and
$$
\frac{\partial^2 \tilde{S}^J}{\partial \mathrm{w} \partial x_l}(\mathrm{w}, x_2^c(\mathrm{w}), \dots, x_{N-2n}^c(\mathrm{w})) = \frac{\partial^2 J}{\partial \mathrm{w} \partial x_l}(\mathrm{w}, x_2^c(\mathrm{w}), \dots, x_{N-2n}^c(\mathrm{w})).$$
Thus, we have
$$
 \frac{\partial^2 \tilde{S}^J}{\partial \mathrm{w}^2}(\mathrm{w}, x_2^c(\mathrm{w}), \dots, x_{N-2n}^c(\mathrm{w}))
= \frac{i}{2} \frac{\partial w_l}{\partial \mathrm{w}} - 
 \sum_{k=2}^{N-2n}\frac{\partial^2 \tilde{S}^J}{\partial \mathrm{w} \partial x_k}(\mathrm{w}, x_2^c(\mathrm{w}), \dots, x_{N-2n}^c(\mathrm{w}))\cdot \Bigg( \frac{\partial x_k^c}{\partial \mathrm{w}} \Bigg).
$$
Claim (3) then follows from Claim (2).
\end{proof}

\begin{corollary}\label{HessJtor}
Assume that $X$ is generalized FAMED with respect to $(l,m)$. Then 
\begin{align*}
&\ \ \frac{1}{\left|\sqrt{\pm \det (\Hess_{x_2,\dots,x_{N-2n}} J)(\mathrm{w}, x_2^c(\mathrm{w}), \dots, x_{N-2n}^c(\mathrm{w}))}\right|}\\
=&\ \   
\frac{D_2'}{\left| \sqrt{\pm 
 \left(\prod_{i=1}^N z_i^{-f_i''} z_i''^{f_i - 1} \right) \tau(\SS^3 \setminus K, m, \mathbf{z^c}, X)}\right|}
\end{align*}
for some constant $D_2'$ independent of $\hbar$ and $\mathbf{x}$. 
\end{corollary}
\begin{proof}
By Proposition \ref{1loopJ} and \ref{Hesstotor},
\begin{align*}
&\ \ \frac{1}{\left|\sqrt{\pm \det (\Hess_{x_2,\dots,x_{N-2n}} J(\mathrm{w}, x_2^c(\mathrm{w}), \dots, x_{N-2n}^c(\mathrm{w})))}\right|}\\
=&\ \ 
    \frac{D_2}{\left|\sqrt{\pm \frac{2C_1^2 }{i} \frac{\partial w_m}{\partial w_l}
\det(\Hess \tilde{S}) (x_1^c(\mathrm{w}), \dots, x_{N-2n}^c(\mathrm{w}))}\right|}\\
=&\ \  
\frac{D_2'}{\left|\sqrt{\pm 
 \left(\prod_{i=1}^N z_i^{-f_i''} z_i''^{f_i - 1} \right) \frac{\partial w_m}{\partial w_l}\tau(\SS^3 \setminus K, l, \mathbf{z^c}, X)}\right|},
\end{align*}
where $D_2' = D_2 /\sqrt{2C_1^1}$ is a constant independent of $\hbar$ and $\mathrm{w}$.
Then the result follows from the change of curve formula of the 1-loop invariants \cite[Theorem 1.19]{PW}.
\end{proof}

\subsection{Asymptotic expansion formula of the Jones function}
Suppose $X$ is a semi-geometric triangulation with shape parameters $\mathbf{z^c}$. We first study the asymptotics of $\mathfrak{J}(\hbar, 0)$ and construct a multi-contour $L$ for applying the saddle point method. The construction is essentially the same as that in Section \ref{sub:proofs:thm15}. Precisely, by the exponentially decaying property at infinity, we choose $\kappa > 0$ such that 
\begin{align}\label{smallbdy}
\Re J(0, x_2,\dots, x_{N-2n}) < \Re J(0, x_2^c(0),\dots, x^c_{N-2n}(0))
\end{align}
for all $(x_2,\dots,x_{N-2n} )\in (\RR^{N-2n-1}+ i (a_2-a_2^0,\dots,a_N-a_N^0)) \setminus ( (-\kappa,\kappa)^{N-2n-1} + i (a_2-a_2^0,\dots,a_N-a_N^0))$. Let
$
\tilde L^{\text{top}}
= \prod_{k=2}^{N-2n} ([-\kappa,\kappa] + i \im x^c_k(0))
$.
\begin{proposition}\label{constructJcont}
$\tilde L^{\text{top}}$ can be deformed along the $(\im x_2,\dots, \im x_{N-2n})$-direction into a new multi-contour $L^{\text{top}}$ such that
\begin{enumerate}
    \item $(x_2^c(0),\dots,x_N^c(0)) \in L^{\text{top}}$;
    \item $\Re J(x_2,\dots,x_{N-2n})$ attains its strict maximum at $(x_2^c(0),\dots,x_N^c(0))$ on $L^{\text{top}}$; 
    \item $\varphi(\mathscr{L}^{-1}(0,x_2,\dots,x_{N-2n}))$ is in the interior of $(\CC \setminus L_\delta)^N$, and
    \item for every $(x_2,\dots,x_{N-2n}) \in \partial L^{\text{top}}$, $\im \varphi(x_1,\dots,x_{N-2n}) \in (0,\pi)^{N}$,
\end{enumerate}
where in (3) and (4), $x_1 = (-\sum_{k=2}^{N-2n} C_k x_k + C)/C_1$.
\end{proposition}
\begin{proof}
  The proof is the same as that of Proposition \ref{constructZcont}.  
\end{proof}
Define $L^{\text{sides}}$ to be the set of points of the form
$$(x_2,\dots,x_{N-2n}) - i t( \im (x_2,\dots,x_{N-2n}) - (v_2,\dots,v_{N-2n})) $$
with $(x_2,\dots,x_{N-2n}) \in \partial L^{\text{top}}$ and $t\in[0,1]$. 
Let $L = L^{\text{sides}} \cup L^{\text{top}}$.
\begin{proposition}\label{Jcont}
$L$ is a multi-contour such that $\boldsymbol{x^c}(0) \in L$ and $\Re J$ attains its strict maximum at $\boldsymbol{x^c}(0)$.
\end{proposition}
\begin{proof}
   The proof is the same as that of Proposition \ref{Zcont}.  
\end{proof}

Note that when $\mathrm{w}$ is not zero, the shape parameters corresponding to the critical point of $\mathfrak{J}$ in the variables $x_2,\dots, x_{N-2n}$ may not have non-negative imaginary parts. In particular, the construction in Proposition \ref{constructJcont} and \ref{Jcont} do not apply. Nonetheless, for $\mathrm{w}$ sufficiently small, the next proposition shows that we can obtain a multi-contour for applying saddle point method to $\mathfrak{J}(\mathrm{w},x_2,\dots,x_{N-2n})$ by deforming $L$.

\begin{figure}
    \centering
    \includegraphics[width=0.9\linewidth]{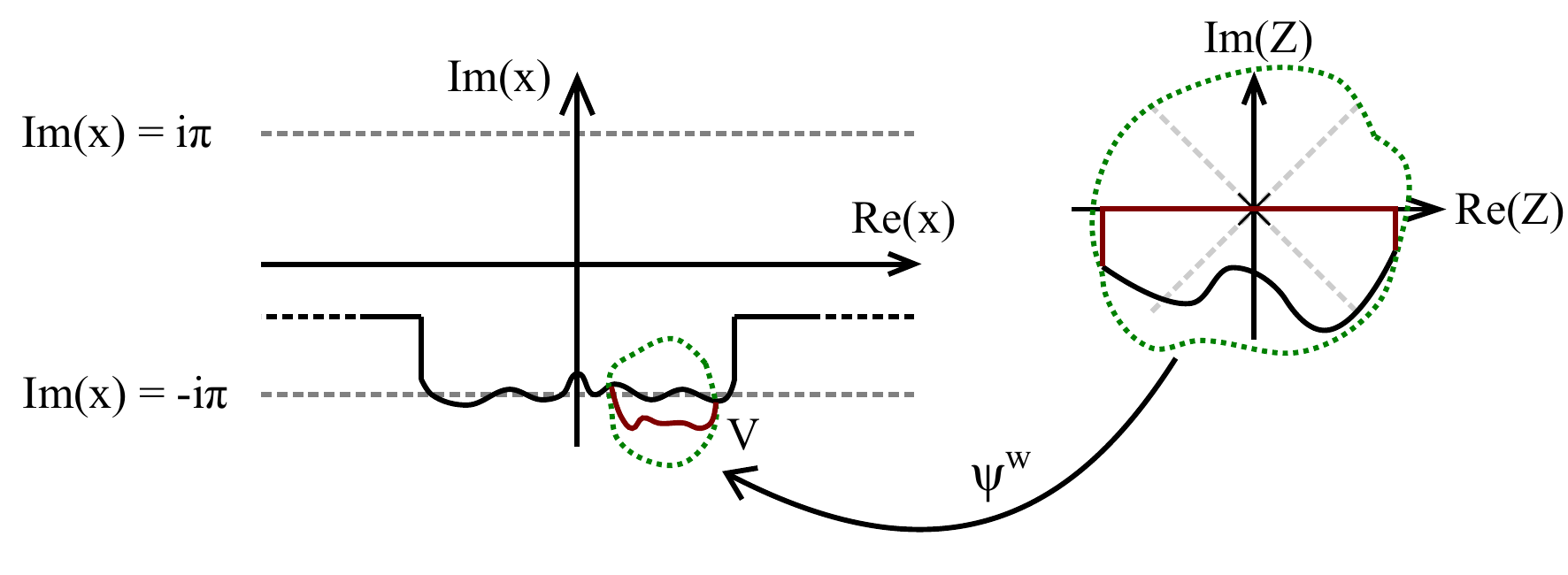}
    \caption{This figure shows a schematic picture of the contours constructed in Proposition \ref{Jcont2} and used in the proof of Theorem \ref{thm:Jones:genFAMED}. As shown in the right figure, Complex Morse Lemma provides a nice local coordinate chart for us to deform the black contour to the maroon contour that passes through the critical point (the origin). 
    }
    \label{Jcontfigure}
\end{figure}

\begin{proposition}\label{Jcont2}
    For $\mathrm{w}$ sufficiently small, there exists a multi-contour $L_{\mathrm{w}}$ such that $L_{0} = L$ and $\Re J(\mathrm{w},x_2,\dots, x_{N-2n})$ attains its strict maximum at $\boldsymbol{x^c}(\mathrm{w})$ on $L_{\mathrm{w}}$.
\end{proposition}
\begin{proof}
We start from the contour $L$ constructed in Proposition \ref{Jcont}. Consider the open set $V \subset \CC^{N-2n-1}$ containing $0$, the open set $A$ containing $\mathrm{w} = 0$, together with the smooth function $\psi: V \times A \to D_{\boldsymbol{x}}$ described in Lemma \ref{OPFCML}. We decompose $L$ into $L = (L \cap \psi^0(V) ) \cup (L \setminus \psi^0(V))$. By Proposition \ref{Jcont}, we know that on $(L \setminus \psi^0(V))$, we have
$$ \Re J(0, x_2,\dots, x_{N-2n}) < \Re J(0, x^c_{2}(0),\dots, x^c_{N-2n}(0)).$$
By continuity, there exists a sufficiently small open set $\mathcal{O}$ containing $0\in \CC$ such that for any $\mathrm{w} \in \mathcal{O}$, on $(L \setminus \psi^0(V))$ we have
$$ \Re J(\mathrm{w}, x_2,\dots, x_{N-2n}) < \Re J(\mathrm{w}, x^c_{2}(\mathrm{w}),\dots, x^c_{N-2n}(\mathrm{w})).$$
Next, consider the preimage $(\psi^{\mathrm{w}})^{-1}(L \cap \psi^0(V) )$. By continuity, for $\mathrm{w}$ sufficiently small, every points  
$(Z_2,\dots,Z_{N-2n}) \in \partial((\psi^{\mathrm{w}})^{-1}(L \cap \psi^0(V) ))$  lies in the region with 
$$ \Re J(\mathrm{w}, \psi^{\mathrm{w}}(Z_2,\dots, Z_{N-2n})) < \Re J(\mathrm{w}, \psi^{\mathrm{w}}(0)).$$
Define $L_{(\psi^\mathrm{w})^{-1}} = L_{(\psi^\mathrm{w})^{-1}}^{\text{top}} \cup  L_{(\psi^\mathrm{w})^{-1}}^{\text{sides}} $, where
$$ L_{(\psi^\mathrm{w})^{-1}}^{top} = \{ (\Re Z_2, \dots, \Re Z_{N-2n}) \mid (Z_2,\dots, Z_{N-2n}) \in (\psi^{\mathrm{w}})^{-1}(L \cap \psi^0(V) ) \} $$
and $L_{(\psi^\mathrm{w})^{-1}}^{sides} $ is the set of points of the form 
$$
(Z_2, \dots, Z_{N-2n} ) - it(\im Z_2,\dots, \im Z_{N-2n}) 
$$
with 
$(Z_2,\dots, Z_{N-2n}) \in \partial((\psi^{\mathrm{w}})^{-1}(L \cap \psi^0(V) ))$ and $ t\in [0,1] $ (see the maroon curve in Figure \ref{Jcontfigure}).
Note that $L_{(\psi^\mathrm{w})^{-1}}$ and $(\psi^{\mathrm{w}})^{-1}(L \cap \psi^0(V) )$ are homotopic to each other. This implies that $\psi^{\mathrm{w}}(L_{(\psi^\mathrm{w})^{-1}})$ and  $L \cap \psi^0(V)$ are homotopic to each other.
Besides, in $(Z_2,\dots, Z_{N-2n})$-coordinate, we have
$$
J(\mathrm{w}, \psi^{\mathrm{w}}(Z_2,\dots,Z_{N-2n})) = J(\mathrm{w}, \boldsymbol{x^c}(\mathrm{w})) - Z_2^2 - \dots - Z_{N-2n}^2.
$$
This implies that on the multi-contour $L_{(\psi^\mathrm{w})^{-1}}$, $\Re J(\mathrm{w}, \psi^{\mathrm{w}}(Z_2,\dots,Z_{N-2n}))$ attains its strict maximum at $(0,\dots,0)$. As a result, the multi-contour $L_{\mathrm{w}} = \psi^{\mathrm{w}}(L_{(\psi^\mathrm{w})^{-1}}) \cup (L \setminus \psi^0(V)) $ satisfies the desired property.
\end{proof}

\begin{proof}[Proof of Theorem \ref{thm:Jones:genFAMED}]
The proof is essentially the same as that of Theorem \ref{mainthmZ}. We write the integral in (\ref{ZtoJconv2}) as
\begin{align*}
&\ \int_{\RR^{N-2n-1}+i\mathbf{\hat{v}}_\alpha} 
h(\varphi( \mathscr{L}^{-1}(\mathrm{w},x_2,\dots,x_{N-2n})))
e^{\frac{1}{2\pi \hbar} J_{\hbar}(\mathrm{w},x_2,\dots,x_{N-2n})} d\boldsymbol{x} \\
=& \ \   
\int_{[-\kappa,\kappa]^{N-2n-1}+i\mathbf{\hat{v}}_\alpha} 
h(\varphi( \mathscr{L}^{-1}(\mathrm{w},x_2,\dots,x_{N-2n})))
e^{\frac{1}{2\pi \hbar} J_{\hbar}(\mathrm{w},x_2,\dots,x_{N-2n})} d\boldsymbol{x} \\
&\ \ + \int_{\RR^{N-2n-1} \setminus ([-\kappa,\kappa]^{N-2n-1}+i\mathbf{\hat{v}}_\alpha)} 
h(\varphi( \mathscr{L}^{-1}(\mathrm{w},x_2,\dots,x_{N-2n})))
e^{\frac{1}{2\pi \hbar} J_{\hbar}(\mathrm{w},x_2,\dots,x_{N-2n})} d\boldsymbol{x}.
\end{align*}
For the second integral, similar to \cite[Lemma 7.10]{BAGPN}, there exists constants $A,B > 0$ such that for all $\hbar \in (0,A)$,
$$
\Bigg|
\int_{\RR^{N-2n-1} \setminus [-\kappa,\kappa]^{N-2n}+i\mathbf{v}_\alpha} 
h(\varphi( \mathscr{L}^{-1}(\mathrm{w},x_2,\dots,x_{N-2n})))
e^{\frac{1}{2\pi \hbar} J_{\hbar}(\mathrm{w},x_2,\dots,x_{N-2n})} d\boldsymbol{x} 
\Bigg|
\leq B e^{\frac{M'}{2\pi \hbar}},
$$
where $M' = \max\left\{\Re S (\mathbf{\varphi(\boldsymbol{x});\lambda_X(\alpha)}) \mid  \boldsymbol{x} \in \partial(\RR^{N-2n} \setminus [-\kappa,\kappa]^{N-2n}+i\mathbf{v}_\alpha)\right\}$. By (\ref{smallbdy}), we have $M' < S( \boldsymbol{x^c};\lambda_X(\alpha))$. Next, for the first integral, by deformation of multi-contour we have
\begin{align*}
&\ \ \int_{[-\kappa,\kappa]^{N-2n-1} + i\mathbf{\hat{v}}_\alpha} 
h(\varphi( \mathscr{L}^{-1}(\mathrm{w},x_2,\dots,x_{N-2n})))
e^{\frac{1}{2\pi \hbar} J_{\hbar}(\mathrm{w},x_2,\dots,x_{N-2n})} d\boldsymbol{x}\\
=&\ \  \int_{L_\mathrm{w}} 
h(\varphi( \mathscr{L}^{-1}(\mathrm{w},x_2,\dots,x_{N-2n})))
e^{\frac{1}{2\pi \hbar} J_{\hbar}(\mathrm{w},x_2,\dots,x_{N-2n})} d\boldsymbol{x},
\end{align*}
where $L_\mathrm{w}$ is the multi-contour in Proposition \ref{Jcont2}. By applying Proposition \ref{saddle}, the desired result follows from Proposition \ref{JandNZ}, Corollary \ref{HessJtor} and Proposition \ref{hvalue}.
\end{proof}

\end{document}